\documentclass{amsart}

\usepackage[latin1]{inputenc}
\usepackage{amssymb,amsmath,amsthm}
\usepackage{amsfonts}
\usepackage{ifthen}
\usepackage[neverdecrease,pointedenum]{paralist}
\usepackage{color}
\usepackage{multirow}
\usepackage{bigdelim}
\usepackage{rotating}
\usepackage{longtable}

\newcommand{\vect}[1]{\ensuremath{\mathbf{#1}}}
\newcommand{\card}[1]{\ensuremath{\lvert{#1}\rvert}}
\newcommand{\cl}[1]{\ensuremath{\mathcal{#1}}}
\newcommand{\nset}[1]{\ensuremath{[{#1}]}}
\newcommand{\couples}[1][n]{\ensuremath{\binom{#1}{2}}} 

\DeclareMathOperator{\fullsig}{Sgn}

\DeclareMathOperator{\reducedsig}{sgn}
\DeclareMathOperator{\deck}{deck}                
\DeclareMathOperator{\range}{Im}                 
\DeclareMathOperator{\pr}{pr}                    

\newcommand{\Sp}[1]{\mathcal{#1}}
\newcommand{\complementX}[1]{\overline{#1}}
\newcommand{\Qsets}{\Omega}
\newcommand{\QsetsE}{\Qsets^{\star}}

\newcommand{\rot}[1]{\langle{#1}\rangle}
\newcommand{\XYset}[2]{[{#1}|{#2}]}
\newcommand{\XYrot}[2]{\rot{{#1}|{#2}}}

\newcommand{\Sgena}{A}
\newcommand{\Sgenb}{B}
\newcommand{\Sgenc}{C}
\newcommand{\Sbuildall}{F}
\newcommand{\Sbuildmon}{G}
\newcommand{\Sbuildsd}{H}
\newcommand{\Smon}{M}
\newcommand{\Sclique}{U}
\newcommand{\Ssd}{S}

\theoremstyle{plain}
\newtheorem{theorem}{Theorem}[section]
\newtheorem{proposition}[theorem]{Proposition}
\newtheorem{lemma}[theorem]{Lemma}
\newtheorem{corollary}[theorem]{Corollary}

\theoremstyle{definition}
\newtheorem{definition}[theorem]{Definition}
\newtheorem{fact}[theorem]{Fact}
\newtheorem{example}[theorem]{Example}

\theoremstyle{remark}
\newtheorem{remark}[theorem]{Remark}

\usepackage{tikz}
\newcommand{\PostsLattice}[1]{
  \begin{tikzpicture}[scale=#1, transform shape]
    \tikzstyle{every node} = [circle, fill=black,scale=0.5]
    \tikzstyle{every label} = [scale=2,draw=none, fill=none, label distance=-4]
    \node (P2) [label=above:$\Omega$] at (7.0+1.0   ,10.5+2.5+1.2) {};
    \node (T0) [label=above left:$T_0$]  at (7.0-3.0,10.5-.5+2.5+.85) {};
    \node (T1) [label=above right:$T_1$] at (7.0+3.0,10.5-.5+2.5+.85) {};
    \node (T)  [label=above:$T_c$] at (7.0-1.0,10.5-1.0+3) {};
    \node (M)  [label=above left:$M$] at (6.76+.1+.4   ,9.38+1.1) {};
    \node (T0M)[label=left:$M_0\;$] at (7.23-1.5-.4+.2-.2, 9.51-.5+1-.1) {};
    \node (T1M)[label=right:$\;M_1$] at (6.76+1.0+.4+.3+.2, 9.38-.5+1-.1) {};
    \node (TM) [label=above right:\raisebox{1.25ex}{$M_c$}] at (7.23-.5, 9.51-1.0+1-.3) {};

    \node (L)  [label=above right:$L$] at (7.0+.5   , 5.0-.2+.2) {};
    \node (T0L)[label=left:$L_0$] at (7.0-1.6, 5.0-.2-.5) {};
    \node (T1L)[label=right:$L_1$] at (7.0+1.6, 5.0-.2-.5) {};
    \node (TL) [label=below left:$L_c$] at (7.0-.5, 5.0-.2-1.0-.2) {};
    \node (SL) [label=right:$LS$] at (7.0+0.0, 5.0-.2-0.5) {};
    
    \node (S) [label=above right:$S$] at (7+.1,7.0-0.05+.4) {};
    \node (ST) [label=left:$S_c$] at (6.5,6.5-0.2+.2) {};
    \node (SM) [label=below left:$SM$] at (6-.1,6.0-0.35) {};

    \node (P21) [label=right:$\Omega(1)$] at (7.0+0.5,3.5-1) {};
    \node (SP21)[label=right:$I^*$] at (7.0-0.2, 3.5-2) {};
    \node (AV)  [label=below:$I$] at (7.0 +.5  , 3.0-4 +.75) {};
    \node (T0AV)[label=below left:$I_0$] at (7.0-2, 3.0-4) {};
    \node (T1AV)[label=below right:$I_1$] at (7.0+2, 3.0-4) {};
    \node (TAV) [label=below:$I_c$] at (7.0-.5,  3.0-4-.75) {}; 

    \foreach \from/\to in {
          P2/T0, P2/T1, T0/T, T1/T,
          M/T0M, M/T1M, T0M/TM, T1M/TM,
          P2/M, T0/T0M, T1/T1M, T/TM,
          L/T0L, L/T1L, T0L/TL, T1L/TL,
          P2/L, T0/T0L, T1/T1L, T/ST, ST/TL,
          AV/T0AV, AV/T1AV, T0AV/TAV, T1AV/TAV,
          L/P21, P21/AV, T0L/T0AV, T1L/T1AV, TL/TAV,
          P2/S, S/ST, S/SL, L/SL, SL/TL, SL/SP21, P21/SP21, SP21/TAV,
          ST/SM, SM/TAV}
    \draw [-] (\from) -- (\to);

    \node (T02)  [label=left:$U_2$] at (0, 8.5) {};
    \node (T02T) [label=left:$T_cU_2$] at (0+1.0, 8.5-.7) {};
    \node (T02M) [label=right:$MU_2$] at (0+2.0, 8.5-.3) {};
    \node (T02TM)[label=right:$M_cU_2$] at (0+3.0, 8.5-1.0) {};
    \node (T03)  [label=left:$U_3$] at (0, 7.0) {};
    \node (T03T) [label=left:$T_cU_3$] at (0+1.0, 7.0-.7) {};
    \node (T03M) [label=right:$MU_3$] at (0+2.0, 7.0-.3) {};
    \node (T03TM)[label=right:$M_cU_3$] at (0+3.0, 7.0-1.0) {};
    \node (T0H) [draw=none, fill=none, scale=0.1] at (0, 6.5) {};
    \node (T0HT) [draw=none, fill=none, scale=0.1] at (0+1.0, 6.5-.7) {};
    \node (T0HM) [draw=none, fill=none, scale=0.1] at (0+2.0, 6.5-.3) {};
    \node (T0HTM) [draw=none, fill=none, scale=0.1] at (0+3.0, 6.5-1.0) {};
    \node (T0e)  [label=left:$U_\infty$] at (0, 5.0) {};
    \node (T0eT) [label=left:$T_cU_\infty$] at (0+1.0, 5.0-.7) {};
    \node (T0eM) [label=right:$MU_\infty$] at (0+2.0, 5.0-.3) {};
    \node (T0eTM)[label=right:$M_cU_\infty$] at (0+3.0, 5.0-1.0) {};
    
    \node (A) [label=above right:$\Lambda$] at (0+3.5,3.5-.5) {};
    \node (AT1) [label=right:$\Lambda_1$] at (0+4.5,3.0-.5) {};
    \node (AT0) [label=left:$\Lambda_0$] at (0+2.0,3.0-.5) {};
    \node (AT) [label=below:$\Lambda_c$] at (0+3.0,2.5-.5) {}; 

    \node (T12)  [label=right:$W_2$] at (14, 8.5) {};
    \node (T12T) [label=right:\raisebox{-3ex}{$\!T_cW_2$}] at (14-1.0, 8.5-.7) {};
    \node (T12M) [label=left:$MW_2\;$] at (14-2.0, 8.5-.3) {};
    \node (T12TM)[label=left:$M_cW_2$] at (14-3.0, 8.5-1.0) {};
    \node (T13)  [label=right:$W_3$] at (14, 7.0) {};
    \node (T13T) [label=right:\raisebox{-3ex}{$\!T_cW_3$}] at (14-1.0, 7.0-.7) {};
    \node (T13M) [label=left:$MW_3\;$] at (14-2.0, 7.0-.3) {};
    \node (T13TM)[label=left:$M_cW_3$] at (14-3.0, 7.0-1.0) {};
    \node (T1H) [draw=none, fill=none, scale=0.1] at (14, 6.5) {};
    \node (T1HT) [draw=none, fill=none, scale=0.1] at (14-1.0, 6.5-.7) {};
    \node (T1HM) [draw=none, fill=none, scale=0.1] at (14-2.0, 6.5-.3) {};
    \node (T1HTM) [draw=none, fill=none, scale=0.1] at (14-3.0, 6.5-1.0) {};
    \node (T1e)  [label=right:$W_\infty$] at (14, 5.0) {};
    \node (T1eT) [label=right:\raisebox{-3ex}{$\!T_cW_\infty$}] at (14-1.0, 5.0-.7) {};
    \node (T1eM) [label=left:$MW_\infty\;$] at (14-2.0, 5.0-.3) {};
    \node (T1eTM)[label=left:$M_cW_\infty$] at (14-3.0, 5.0-1.0) {};

    \node (V) [label=above left:$V$] at (14-3.5,3.5-.5) {};
    \node (VT0) [label=left:$V_0$] at (14-4.5,3.0-.5) {};
    \node (VT1) [label=right:$V_1$] at (14-2.0,3.0-.5) {};
    \node (VT) [label=below:$V_c$] at (14-3.0,2.5-.5) {}; 

    \foreach \from/\to in {
          T02/T02T, T02/T02M, T02T/T02TM, T02M/T02TM,
          T03/T03T, T03/T03M, T03T/T03TM, T03M/T03TM,
          T0e/T0eT, T0e/T0eM, T0eT/T0eTM, T0eM/T0eTM,
          T12/T12T, T12/T12M, T12T/T12TM, T12M/T12TM,
          T13/T13T, T13/T13M, T13T/T13TM, T13M/T13TM,
          T1e/T1eT, T1e/T1eM, T1eT/T1eTM, T1eM/T1eTM,
          T02/T03, T02T/T03T, T02M/T03M, T02TM/T03TM,
          T03/T0H, T03T/T0HT, T03M/T0HM, T03TM/T0HTM,
          T12/T13, T12T/T13T, T12M/T13M, T12TM/T13TM,
          T13/T1H, T13T/T1HT, T13M/T1HM, T13TM/T1HTM,
          T0eM/AT0, T0eTM/AT,
          T1eM/VT1, T1eTM/VT,
          A/AT0, A/AT1, AT0/AT, AT1/AT,
          V/VT0, V/VT1, VT0/VT, VT1/VT,
          A/AV, AT0/T0AV, AT1/T1AV, AT/TAV,
          V/AV, VT0/T0AV, VT1/T1AV, VT/TAV,
          T02TM/SM, T12TM/SM,
          T0/T02, T0M/T02M, T/T02T, TM/T02TM,
          T1/T12, T1M/T12M, T/T12T, TM/T12TM}
    \draw [-] (\from) -- (\to);
    \path (M) edge [out=215, in=90] (A);
    \path (T1M) edge [out=215, in=90] (AT1);
    \path (M) edge [out=325, in=90] (V);
    \path (T0M) edge [out=325, in=90] (VT0);

    \foreach \from/\to in {
          T0H/T0e, T0HT/T0eT, T0HM/T0eM, T0HTM/T0eTM,
          T1H/T1e, T1HT/T1eT, T1HM/T1eM, T1HTM/T1eTM}
    \draw [dotted] (\from) -- (\to);
  \end{tikzpicture}
}

\begin{document}
\title{Hypomorphic Sperner systems and nonreconstructible functions}

\author{Miguel Couceiro}
\address[M. Couceiro]{LAMSADE \\
Université Paris-Dauphine \\
Place du Maréchal de Lattre de Tassigny \\
75775 Paris Cedex 16 \\
France}
\email{miguel.couceiro@dauphine.fr}

\author{Erkko Lehtonen}
\address[E. Lehtonen]{University of Luxembourg \\
Computer Science and Communications Research Unit \\
6, rue Richard Coudenhove-Kalergi \\
L--1359 Luxembourg \\
Luxembourg}
\email{erkko.lehtonen@uni.lu}

\author{Karsten Schölzel}
\address[K. Schölzel]{University of Luxembourg \\
Mathematics Research Unit \\
6, rue Richard Coudenhove-Kalergi \\
L--1359 Luxembourg \\
Luxembourg}
\email{karsten.schoelzel@uni.lu}

\date{\today}

\begin{abstract}
A reconstruction problem is formulated for Sperner systems, and infinite families of nonreconstructible Sperner systems are presented.
This has an application to a reconstruction problem for functions of several arguments and identification minors.
Sperner systems being representations of certain monotone functions, infinite families of nonreconstructible functions are thus obtained.
The clones of Boolean functions are completely classified in regard to reconstructibility.
\end{abstract}

\maketitle


\section{Introduction}
\label{sec:intro}

Reconstruction problems have drawn the attention of researchers over the past decades.
Generally speaking, a reconstruction problem asks whether a mathematical object can be recovered from partial information.
The reconstruction problems we discuss here fit into the following general framework. Given a combinatorial object, we derive some ``subobjects'' by applying a certain operation in all possible ways to the elements of our initial object. Then we may ask whether the initial object is uniquely determined (up to some kind of isomorphism) by the collection of the derived subobjects.

Perhaps the most famous reconstruction problem is the following: Is every graph with at least three vertices uniquely determined, up to isomorphism, by the collection of its one-vertex-deleted subgraphs?
It was conjectured by Kelly~\cite{Kelly1942} (see also Ulam's problem book~\cite{Ulam}) that the answer is positive.
While the conjecture has been shown to hold for various classes of graphs, such as trees, regular graphs, and disconnected graphs, it still remains one of the most important open problems in graph theory.

The reconstruction problem for graphs can be varied in different ways. For example, we may form subgraphs by deleting edges instead of vertices (see Ellingham~\cite{Ellingham} and Harary~\cite{Harary1964}), or we could consider directed graphs or hypergraphs. Let us mention that directed graphs and hypergraphs are not in general reconstructible from one-vertex-deleted subgraphs: infinite nonreconstructible families of directed graphs have been presented by Stockmeyer~\cite{Stockmeyer}, and infinite nonreconstructible families of hypergraphs have been presented by Kocay~\cite{Kocay} and by Kocay and Lui~\cite{KocLui}.
Analogous reconstruction problems have been formulated for many other kinds of mathematical objects.

In this paper, we investigate the following reconstruction problem:
Is a function $f \colon A^n \to B$ uniquely determined, up to equivalence, by the collection of its identification minors, i.e., the functions obtained from $f$ by identifying a pair of its arguments?
This problem was first formulated in~\cite{LehtonenSymmetric}, and in~\cite{LehtonenSymmetric,LehtonenLinear} some classes of functions were shown to be reconstructible or weakly reconstructible, such as totally symmetric functions, functions determined by the order of first occurrence, and affine functions over nonassociative semirings.

The initial objective of the work reported in this paper was to take another small step towards understanding the reconstructibility of functions of several arguments and to determine whether order-preserving functions are reconstructible.
We started with an important special case, namely monotone Boolean functions.
Since an $n$-ary monotone Boolean function is completely and uniquely determined by the set of its minimal true points, and the minimal true points constitute an antichain in the componentwise ordering of $n$-tuples -- which may be viewed as subsets of an $n$-element set -- this led us into reformulating the reconstruction problem in this special case in terms of Sperner systems.

We construct several infinite families of pairs of nonisomorphic Sperner systems with the same deck.
This translates into the statement that the class of monotone Boolean functions is not reconstructible, and it is not even weakly reconstructible, and there exist pairs of nonequivalent monotone Boolean functions of arbitrarily high arity with the same deck.
Furthermore, the members of some of our families have some special properties that guarantee that the associated Boolean functions belong to certain subclones of the clone of all monotone functions, namely to the clone of self-dual monotone functions and to the clone of monotone constant-preserving $1$-separating functions.

Having constructed several infinite families of nonreconstructible Sperner systems, we will also have some positive results on reconstructibility. Namely, we show that $1$-homogeneous Sperner systems are reconstuctible, and so are Sperner systems with exactly one block. Considering the associated Boolean functions, this means that the clones of disjunctions and conjunctions are reconstructible. Combining these results with the theorem from~\cite{LehtonenLinear} that asserts that linear functions over finite fields are reconstructible, we arrive at a complete classification of the clones of Boolean functions in regard to their reconstructibility.

Sperner systems can be seen not only as representations of monotone Boolean functions but more generally as representations of term operations over a distributive lattice (or, even more generally, as representations of polynomial operations of a certain special form over a distributive lattice). This means that the class of term operations over a distributive lattice is not reconstructible.
All functions represented by Sperner systems (as described above) are monotone, but we also present ways of extending them into functions over larger domains that are not necessarily monotone but remain nonreconstructible.

Our work has a surprising connection to the reconstruction problem of hypergraphs and one-vertex-deleted subhypergraphs that was mentioned earlier in this introduction.
One of the families of Sperner systems we present here (namely the family $\Sp{\Smon}^m_i$ ($m \geq 3$, $i \in \{1, 2\}$) as in Definition~\ref{def:monotone}) actually serves as yet another example of an infinite family of nonreconstructible hypergraphs, in addition to those already presented by Kocay~\cite{Kocay} and Kocay and Lui~\cite{KocLui}; see Corollary~\ref{cor:nonreconstructiblehypergraphs}.

This paper is organised as follows.
In Section~\ref{sec:preliminaries}, we present all necessary definitions and formulate reconstruction problems for Sperner systems and for functions of several arguments, and we explain how these two reconstruction problems are related.
In Section~\ref{sec:main}, we construct infinite families of nonreconstructible Sperner systems. Several families are provided so as to prove that various clones of Boolean functions are not weakly reconstructible.
We have also some positive results about the reconstructibility of Sperner systems. In Section~\ref{sec:positive}, we show that $1$-homogeneous Sperner systems over sufficiently large sets are reconstructible, and so are Sperner systems with just one block.
In Section~\ref{sec:classification}, we classify the clones of Boolean functions in regard to reconstructibility.
The nonreconstructible functions that we have encountered so far are all order-preserving, but in Section~\ref{sec:further}, we discuss how to build other kinds of nonreconstructible functions.
In Section~\ref{sec:hypergraph}, we show that the Sperner systems $\Sp{\Smon}^m_i$ actually constitute an example of an infinite family of nonreconstructible hypergraphs.
Appendix~\ref{app:Post} provides a list of the clones of Boolean functions.
Appendix~\ref{app:Sperner} provides a list of all Sperner systems over sets with at most five elements, together with their decks, and the nonreconstructible ones are indicated.


\section{Preliminaries}
\label{sec:preliminaries}

\subsection{General}
\label{sec:preliminaries:general}

Let $\mathbb{N} := \{0, 1, 2, \dots\}$.
Throughout this paper, $k$, $\ell$, $m$ and $n$ stand for positive integers, and $A$ and $B$ stand for arbitrary finite sets with at least two elements. The set $\{1, \dots, n\}$ is denoted by $\nset{n}$. The set of all $2$-element subsets of a set $A$ is denoted by $\couples[A]$; we will write simply $\couples$ for $\couples[\nset{n}]$. Tuples are denoted by bold-face letters and components of a tuple are denoted by the corresponding italic letters with subscripts, e.g., $\vect{a} = (a_1, \dots, a_n)$.

Let $\vect{a} \in A^n$, and let $\sigma \colon \nset{m} \to \nset{n}$. We will write $\vect{a} \sigma$ to denote the $m$-tuple $(a_{\sigma(1)}, \dots, a_{\sigma(m)})$.
Since the $n$-tuple $\vect{a}$ can be formally seen as the map $\vect{a} \colon \nset{n} \to A$, $i \mapsto a_i$, the $m$-tuple $\vect{a} \sigma$ is just the composite map $\vect{a} \circ \sigma \colon \nset{m} \to A$.

A \emph{finite multiset} $M$ on a set $S$ is a couple $(S, \mathbf{1}_M)$, where $\mathbf{1}_M \colon S \to \mathbb{N}$ is a map, called a \emph{multiplicity function,} such that the set $\{x \in S : \mathbf{1}_M(x) \neq 0\}$ is finite. Then the sum $\sum_{x \in S} \mathbf{1}_M(x)$ is a well-defined natural number, and it is called the \emph{cardinality} of $M$. For each $x \in S$, the number $\mathbf{1}_M(x)$ is called the \emph{multiplicity} of $x$ in $M$.
If $(a_i)_{i \in I}$ is a finite indexed family of elements of $S$, then we will write $\{a_i : i \in I\}$ to denote the multiset in which the multiplicity of each $x \in S$ equals $\card{\{i \in I : a_i = x\}}$. While this notation is similar to that used for sets, it will always be clear from the context whether we refer to a set or to a multiset.

\subsection{Reconstruction problems}

Before going into our specific problems, we recall some usual general terminology of reconstruction problems.
In very abstract terms, a \emph{reconstruction problem} comprises the following pieces of data:
\begin{itemize}
\item a collection $\mathcal{O}$ of \emph{objects},
\item an equivalence relation $\equiv$ on $\mathcal{O}$,
\item for each object $O \in \mathcal{O}$, an associated natural number called the \emph{size} of $O$,
\item for each $n \in \mathbb{N}$, an index set $I_n$,
\item for every object $O$ of size $n$ and for every $i \in I_n$, a \emph{derived object} $O_i \in \mathcal{O}$.
\end{itemize}

Let $O \in \mathcal{O}$ be an object of size $n$.
The equivalence classes $O_i / {\equiv}$ of the derived objects $O_i$ for each $i \in I_n$ are referred to as the \emph{cards} of $O$.
The \emph{deck} of $O$, denoted $\deck O$, is the multiset $\{O_i / {\equiv} : i \in I_n\}$ of the cards of $O$.

We may now ask whether an object $O$ is uniquely determined, up to equivalence, by its deck.
In order to discuss whether and to which extent this is the case, we will use the following terminology that is more or less standard.

Let $O$ and $O'$ be objects of size $n$.
We say that $O'$ is a \emph{reconstruction} of $O$, or that $O$ and $O'$ are \emph{hypomorphic,} if $\deck O = \deck O'$, or, equivalently, if there exists a bijection $\phi \colon I_n \to I_n$ such that $O_i \equiv O'_{\phi(i)}$ for every $i \in I_n$.
If the last condition holds with $\phi$ equal to the identity map on $I_n$, i.e., if $O_i \equiv O'_i$ for every $i \in I_n$, then we say that $O$ and $O'$ are \emph{strongly hypomorphic.}
Note that strongly hypomorphic objects are necessarily hypomorphic, but the converse is not true in general.

An object is \emph{reconstructible} if it is equivalent to all of its reconstructions.
A class $\mathcal{C} \subseteq \mathcal{O}$ of objects is \emph{reconstructible} if all members of $\mathcal{O}$ are reconstructible.
A class $\mathcal{C} \subseteq \mathcal{O}$ is \emph{weakly reconstructible} if for every $O \in \mathcal{C}$, all reconstructions of $O$ that are members of $\mathcal{C}$ are equivalent to $O$.
A class $\mathcal{C} \subseteq \mathcal{O}$ is \emph{recognizable} if all reconstructions of the members of $\mathcal{C}$ are members of $\mathcal{C}$.
Note that a reconstructible class of objects is necessarily weakly reconstructible, but the converse is not true in general.
If a class of objects is recognizable and weakly reconstructible, then it is reconstructible.

\subsection{Sperner systems}

The set $\mathcal{P}(A)$ of all subsets of $A$ is called the \emph{power set} of $A$.
Ordered by inclusion, $\mathcal{P}(A)$ constitutes a lattice.
Any subset of $\mathcal{P}(A)$ is called a \emph{set system} over $A$, its elements are called \emph{blocks}, and the set $A$ is referred to as its \emph{ground set.}
A set system in which no block is included in another is called a \emph{Sperner system}.
Equivalently, a Sperner system is an antichain in the power set lattice $(\mathcal{P}(A); \subseteq)$.
The set of the minimal elements of any set system is a Sperner system.

A set system is \emph{$k$-homogeneous} if each one of its blocks has cardinality $k$.
A set system is \emph{homogeneous} if it is $k$-homogeneous for some $k$.

For any function $f \colon A \to B$ and any subset $S$ of $A$, we write $f(S)$ for the set $\{f(x) : x \in S\}$. For any set system $\Sp{\Sgena}$ over $A$, we write $f(\Sp{\Sgena})$ for $\{f(S) : S \in \Sp{\Sgena}\}$.

Let $\Sp{\Sgena}$ be a set system over $A$ and let $\Sp{\Sgenb}$ be a set system over $B$. We say that $\Sp{\Sgena}$ and $\Sp{\Sgenb}$ are \emph{isomorphic} and we write $\Sp{\Sgena} \equiv \Sp{\Sgenb}$ if there exists a bijection $\sigma \colon A \to B$ such that $\sigma(\Sp{\Sgena}) = \Sp{\Sgenb}$. In this case, such a bijection $\sigma$ is called an \emph{isomorphism.} We denote the isomorphism type of $\Sp{\Sgena}$ by $\Sp{\Sgena} / {\equiv}$.

Let $\theta$ be an equivalence relation on $A$. The $\theta$-class of an element $x \in A$ is denoted by $x / \theta$.
For an arbitrary subset $S$ of $A$, we let $S / \theta := \{x / \theta : x \in S\}$.
For $I \in \couples[A]$, let $\theta_I$ be the equivalence relation on $A$ whose only nonsingleton equivalence class is $I$.
We will write $S_I$ for $S / \theta_I$, and for a set system $\Sp{\Sgena}$ over $A$, we will write $\Sp{\Sgena}_I$ for $\{S_I : S \in \Sp{\Sgena}\}$.
As usual, we will often simplify notation and will denote each equivalence class by any one of its representatives.

Let $\Sp{\Sgena}$ be a Sperner system over $A$.
For $I \in \couples[A]$, let $\Sp{\Sgena}^*_I$ be the set of minimal elements of $\Sp{\Sgena}_I$. Then $\Sp{\Sgena}^*_I$ is a Sperner system over $A / \theta_I$.

\begin{definition}
We can now specify the data for the \emph{reconstruction problem for Sperner systems.}
The objects are all Sperner systems over finite sets.
The equivalence is given by the isomorphism between Sperner systems.
The size of a Sperner system is the cardinality of its ground set.
For each $n \in \mathbb{N}$, the index set $I_n$ is the set $\couples$ of all two-element subsets of $\nset{n}$.
We may assume that we have, for every finite set $A$, a fixed bijection $\sigma_A \colon \nset{\card{A}} \to A$.
Then for a Sperner system $\Sp{\Sgena}$ over a set $A$ of cardinality $n$ and for $I \in \couples$, the derived object $\Sp{\Sgena}_I$ is $\Sp{\Sgena}^*_{\sigma_A(I)}$.
Hence the cards of a Sperner system $\Sp{\Sgena}$ over $A$ are the isomorphism types $\Sp{\Sgena}^*_I / {\equiv}$ of the Sperner systems $\Sp{\Sgena}^*_I$ for $I \in \couples[A]$, and the deck of $\Sp{\Sgena}$ is the multiset $\{\Sp{\Sgena}^*_I / {\equiv} : I \in \couples[A]\}$.
\end{definition}

It clearly holds that isomorphic Sperner systems are hypomorphic. The converse is not true, as illustrated by the following simple examples. Example~\ref{ex:Sp4} also illustrates that hypomorphic Sperner systems are not necessarily strongly hypomorphic nor are they even necessarily equivalent to a pair of strongly hypomorphic Sperner systems.

\begin{example}
\label{ex:Sp2}
Let $A = \{1, 2\}$, $\Sp{\Sgena} = \{\{1\}\}$, $\Sp{\Sgenb} = \{\{1\}, \{2\}\}$, $\Sp{\Sgenc} = \{\{1, 2\}\}$. Obviously $\Sp{\Sgena}$, $\Sp{\Sgenb}$ and $\Sp{\Sgenc}$ are pairwise nonisomorphic, and it is easy to verify that $\Sp{\Sgena}^*_{\{1, 2\}} = \Sp{\Sgenb}^*_{\{1, 2\}} = \Sp{\Sgenc}^*_{\{1, 2\}} = \{\{1\}\}$. Hence $\deck \Sp{\Sgena} = \deck \Sp{\Sgenb} = \deck \Sp{\Sgenc}$. Moreover, $\Sp{\Sgena}$, $\Sp{\Sgenb}$ and $\Sp{\Sgenc}$ are strongly hypomorphic.
\end{example}

\begin{example}
\label{ex:Sp3}
Let $A = \{1, 2, 3\}$, $\Sp{\Sgena} = \{\{1, 2\}, \{1, 3\}, \{2, 3\}\}$, $\Sp{\Sgenb} = \{\{1\}\}$. Obviously $\Sp{\Sgena} \not\equiv \Sp{\Sgenb}$, and it is easy to verify that $\Sp{\Sgena}^*_I \equiv \{\{1\}\} \equiv \Sp{\Sgenb}^*_I$ for all $I \in \couples[3]$. Hence $\deck \Sp{\Sgena} = \deck \Sp{\Sgenb}$. Moreover, $\Sp{\Sgena}$ and $\Sp{\Sgenb}$ are strongly hypomorphic.
\end{example}

\begin{example}
\label{ex:Sp4}
Let $A = \{1, 2, 3, 4\}$, $\Sp{\Sgena} = \{\{1, 2\}, \{1, 3\}, \{1, 4\}, \{2, 3, 4\}\}$, $\Sp{\Sgenb} = \{\{1, 2\}, \{1, 3\}, \{2, 3\}\}$. Obviously $\Sp{\Sgena} \not\equiv \Sp{\Sgenb}$, and it is easy to verify that
\begin{itemize}
\item $\Sp{\Sgena}^*_I \equiv \{\{1\}\}$ for $I \in \{\{1, 2\}, \{1, 3\}, \{1, 4\}\}$,
\item $\Sp{\Sgena}^*_I \equiv \{\{1, 2\}, \{1, 3\}, \{2, 3\}\}$ for $I \in \{\{2, 3\}, \{2, 4\}, \{3, 4\}\}$,
\item $\Sp{\Sgenb}^*_I \equiv \{\{1\}\}$ for $I \in \{\{1, 2\}, \{1, 3\}, \{2, 3\}\}$,
\item $\Sp{\Sgenb}^*_I \equiv \{\{1, 2\}, \{1, 3\}, \{2, 3\}\}$ for $I \in \{\{1, 4\}, \{2, 4\}, \{3, 4\}\}$.
\end{itemize}
Hence $\Sp{\Sgena}$ and $\Sp{\Sgenb}$ are hypomorphic but not strongly hypomorphic.
Moreover, there do not exist Sperner systems $\Sp{\Sgena}'$ and $\Sp{\Sgenb}'$ such that $\Sp{\Sgena} \equiv \Sp{\Sgena}'$, $\Sp{\Sgenb} \equiv \Sp{\Sgenb}'$ and $\Sp{\Sgena}'$ and $\Sp{\Sgenb}'$ are strongly hypomorphic.
\end{example}

In fact, Examples~\ref{ex:Sp2}--\ref{ex:Sp4} exhibit, up to isomorphism, all nonreconstructible Sperner systems over an $n$-element set, for $2 \leq n \leq 5$. In particular, every Sperner system over a $5$-element set is reconstructible. (The case $n = 2$ is trivial. For $n = 3$, this is quite easy to see. It gets more tedious than difficult to verify the claim for $n = 4$, and a computer may be extremely helpful in dealing with the case $n = 5$. In order to assist the reader in verifying these claims, we provide in Appendix~\ref{app:Sperner} a list of all Sperner systems over sets with at most five elements, up to isomorphism, together with their decks. Unfortunately, the authors are not aware of any simpler proof of these claims than an exhaustive search.)
One might be led into thinking that these examples of nonreconstructible Sperner systems are just some anomalies that only arise on small ground sets and maybe all Sperner systems over sufficiently large ground sets are reconstructible. However, as we will see in this paper, this is not true; there exist nonisomorphic hypomorphic pairs of Sperner systems over every set with at least six elements.

\subsection{Functions of several arguments and identification minors}

A \emph{function} (\emph{of several arguments}) from $A$ to $B$ is a map $f \colon A^n \to B$ for some positive integer $n$, called the \emph{arity} of $f$. Functions of several arguments from $A$ to $A$ are called \emph{operations} on $A$. Operations on $\{0, 1\}$ are called \emph{Boolean functions.}
We denote the set of all $n$-ary functions from $A$ to $B$ by $\cl{F}_{AB}^{(n)}$, and we denote the set of all functions from $A$ to $B$ of any finite arity by $\cl{F}_{AB}$. We also write $\cl{F}_{AB}^{(\geq n)}$ for $\bigcup_{m \geq n} \cl{F}_{AB}^{(m)}$. In other words, $\cl{F}_{AB}^{(n)} = B^{A^n}$ and $\cl{F}_{AB} = \cl{F}_{AB}^{(\geq 1)}$.
We also denote by $\cl{O}_A$ the set of all operations on $A$.
For any class $\mathcal{C} \subseteq \cl{F}_{AB}$, we let $\mathcal{C}^{(n)} := \mathcal{C} \cap \cl{F}_{AB}^{(n)}$ and $\mathcal{C}^{(\geq n)} := \mathcal{C} \cap \cl{F}_{AB}^{(\geq n)}$.

Let $f \colon A^n \to B$. For $i \in \nset{n}$, the $i$-th argument of $f$ is \emph{essential,} or $f$ \emph{depends} on the $i$-th argument, if there exist tuples $\vect{a}, \vect{b} \in A^n$ such that $a_j = b_j$ for all $j \in \nset{n} \setminus \{i\}$ and $f(\vect{a}) \neq f(\vect{b})$. Arguments that are not essential are \emph{inessential.}

We say that a function $f \colon A^n \to B$ is a \emph{minor} of another function $g \colon A^m \to B$, and we write $f \leq g$, if there exists a map $\sigma \colon \nset{m} \to \nset{n}$ such that $f(\vect{a}) = g(\vect{a} \sigma)$ for all $\vect{a} \in A^m$.
The minor relation $\leq$ is a quasiorder on $\cl{F}_{AB}$, and, as for all quasiorders, it induces an equivalence relation on $\cl{F}_{AB}$ by the following rule: $f \equiv g$ if and only if $f \leq g$ and $g \leq f$. We say that $f$ and $g$ are \emph{equivalent} if $f \equiv g$. Furthermore, $\leq$ induces a partial order on the quotient $\cl{F}_{AB} / {\equiv}$. (Informally speaking, $f$ is a minor of $g$, if $f$ can be obtained from $g$ by permutation of arguments, addition of inessential arguments, deletion of inessential arguments, and identification of arguments. If $f$ and $g$ are equivalent, then each one can be obtained from the other by permutation of arguments, addition of inessential arguments, and deletion of inessential arguments.)
We denote the $\equiv$-class of $f$ by $f / {\equiv}$.
Note that equivalent functions have the same number of essential arguments and every nonconstant function is equivalent to a function with no inessential arguments.
Note also in particular that if $f, g \colon A^n \to B$, then $f \equiv g$ if and only if there exists a bijection $\sigma \colon \nset{n} \to \nset{n}$ such that $f(\vect{a}) = g(\vect{a} \sigma)$ for all $\vect{a} \in A^n$.

Of particular interest to us are the following minors. Let $n \geq 2$, and let $f \colon A^n \to B$. For each $I \in \couples$, we define the function $f_I \colon A^{n-1} \to B$ by the rule
$f_I(\vect{a}) = f(\vect{a} \delta_I)$ for all $\vect{a} \in A^{n-1}$,
where $\delta_I \colon \nset{n} \to \nset{n - 1}$ is given by the rule
\[
\delta_I(i) =
\begin{cases}
i, & \text{if $i < \max I$,} \\
\min I, & \text{if $i = \max I$,} \\
i - 1, & \text{if $i > \max I$.}
\end{cases}
\]
In other words, if $I = \{i, j\}$ with $i < j$, then
\[
f_I(a_1, \dots, a_{n-1}) =
f(a_1, \dots, a_{j-1}, a_i, a_j, \dots, a_{n-1}).
\]
Note that $a_i$ occurs twice on the right side of the above equality: both at the $i$-th and at the $j$-th position. We will refer to the function $f_I$ as an \emph{identification minor} of $f$. This nomenclature is motivated by the fact that $f_I$ is obtained from $f$ by identifying the arguments indexed by the pair $I$.

\begin{definition}
We can now specify the data for the \emph{reconstruction problem for functions of several arguments and identification minors.}
The objects are all functions of several arguments from $A$ to $B$, i.e., the elements of the set $\cl{F}_{AB}$.
The equivalence relation is the relation $\equiv$ on $\cl{F}_{AB}$ as defined above.
The size of a function $f \colon A^n \to B$ is its arity $n$.
For each $n \in \mathbb{N}$, the index set $I_n$ is the set $\couples$ of all two-element subsets of $\nset{n}$.
For a function $f \colon A^n \to B$ and for $I \in \couples$, the derived object $f_I$ is the identification minor $f_I$ of $f$ as defined above.
Hence the cards of $f$ are the equivalence classes $f_I / {\equiv}$ of the various identification minors $f_I$ of $f$, and the deck of $f$ is the multiset $\{f_I / {\equiv} : I \in \couples\}$.
\end{definition}

\subsection{Clones}

If $f \colon B^n \to C$ and $g_1, \dots, g_n \colon A^m \to B$, then the \emph{composition} of $f$ with $g_1, \dots, g_n$ is the function $f(g_1, \dots, g_n) \colon A^m \to C$ given by the rule
\[
f(g_1, \dots, g_n)(\vect{a}) =
f \bigl( g_1(\vect{a}), \dots, g_n(\vect{a}) \bigr),
\]
for all $\vect{a} \in A^m$.

For integers $n$ and $i$ such that $1 \leq i \leq n$, the $i$-th $n$-ary \emph{projection} on $A$ is the operation $\pr_i^{(n)} \colon A^n \to A$, $(a_1, \dots, a_n) \mapsto a_i$ for all $(a_1, \dots, a_n) \in A^n$. 

A \emph{clone} on $A$ is a class of operations on $A$ that contains all projections on $A$ and is closed under functional composition. Trivial examples of clones are the set $\cl{O}_A$ of all operations on $A$ and the set of all projections on $A$.

The clones on the two-element set $\{0, 1\}$ were completely described by Post~\cite{Post}, and they are presented in Appendix~\ref{app:Post}.
In the sequel, we will make specific reference to the following clones of Boolean functions:
\begin{itemize}
\item the clone $M$ of monotone functions,
\item the clone $SM$ of self-dual monotone functions,
\item the clone $M_c U_\infty$ of monotone constant-preserving $1$-separating functions,
\item the clone $M_c W_\infty$ of monotone constant-preserving $0$-separating functions,
\item the clone $\Lambda$ of polynomial operations of the two-element meet-semilattice,
\item the clone $V$ of polynomial operations of the two-element join-semilattice,
\item the clone $L$ of polynomial operations of the group of addition modulo $2$.
\end{itemize}

\subsection{The relationship between functions and Sperner systems}
\label{sub:relationship}

We assume that the reader is familiar with the notions of terms, polynomials, term operations, polynomial operations; see \cite{DenWis} for standard definitions and background. For an algebra $\mathbf{A}$ of type $\tau$ and for a term $t$ of the same type, we denote by $t^\mathbf{A}$ the term operation of $\mathbf{A}$ induced by $t$.

Throughout this section, we assume that $\mathbf{A} = (A; \wedge, \vee, 0, 1)$ is a bounded distributive lattice with least element $0$ and greatest element $1$.
To each Sperner system $\Sp{\Sgena}$ over $\nset{m}$, we associate an $m$-ary term
\[
t_{\Sp{\Sgena}} := \bigvee_{S \in \Sp{\Sgena}} \bigl( \bigwedge_{i \in S} x_i \bigr)
\]
in the language of bounded distributive lattices.
By definition, the term $t_{\Sp{\Sgena}}$ induces a term operation $t_{\Sp{\Sgena}}^\mathbf{A}$ of the lattice $\mathbf{A}$.
On the other hand, every $m$-ary term operation of $\mathbf{A}$ is induced by a unique term of the form $t_{\Sp{\Sgena}}$ for some Sperner system $\Sp{\Sgena}$ over $\nset{m}$. Consequently, the $m$-ary term operations of $\mathbf{A}$ are in a one-to-one correspondence with Sperner systems over an $m$-element set.

Furthermore, $(t_{\Sp{\Sgena}}^\mathbf{A})_I = t_{\Sp{\Sgena}^*_I}^\mathbf{A}$ for every $I \in \couples$. Hence, $t_{\Sp{\Sgena}}^\mathbf{A}$ and $t_{\Sp{\Sgenb}}^\mathbf{A}$ are hypomorphic if and only if $\Sp{\Sgena}$ and $\Sp{\Sgenb}$ are hypomorphic.
Consequently, the reconstruction problem for Sperner systems is essentially the same as the reconstruction problem for functions and identification minors when restricted to the class of term functions of a distributive lattice.

More generally, for fixed elements $a, b \in A$ such that $a < b$ in the lattice order, we can associate to each Sperner system $\Sp{\Sgena}$ over $\nset{m}$ the $m$-ary lattice polynomial
\[
t^{ab}_{\Sp{\Sgena}} := a \vee \bigl(b \wedge \bigvee_{S \in \Sp{\Sgena}} \bigwedge_{i \in S} x_i \bigr).
\]
Polynomials of the above form are referred to as \emph{$(a,b)$-truncated terms;} we also refer as \emph{truncated terms} to $(a,b)$-truncated terms for some $a, b \in A$. 
We will also speak of \emph{$(a,b)$-truncated term operations} and \emph{truncated term operations} of $\mathbf{A}$, the meaning being obvious.
Note that the $(0,1)$-truncated term operations of $\mathbf{A}$ are precisely the term operations of $\mathbf{A}$.
As above, the reconstruction problem for Sperner systems is essentially the same as the reconstruction problem for functions and identification minors when restricted to the class of $(a,b)$-truncated term operations of $\mathbf{A}$.

The relationship between the reconstructibility of Sperner systems and that of truncated term operations of a bounded distributive lattice will be made precise in Proposition~\ref{prop:termopreconstr}. We need a few auxiliary results.

For $I \subseteq \nset{n}$, the \emph{characteristic tuple} of $I$, denoted by $\vect{e}_I$, is the $n$-tuple whose $i$-th component is $1$ if $i \in I$ and $0$ otherwise.

\begin{theorem}[{Goodstein~\cite{Goodstein}}]
\label{thm:Goodstein}
Let $\mathbf{A} = (A; \wedge, \vee, 0, 1)$ be a bounded distributive lattice.
A function $f \colon A^n \to A$ is a polynomial operation of $\mathbf{A}$ if and only if
\[
f(x_1, \dots, x_n) = \bigvee_{I \subseteq \nset{n}} \bigl( f(\vect{e}_I) \wedge \bigwedge_{i \in I} x_i \bigr).
\]
\end{theorem}

\begin{remark}
An immediate consequence of Theorem~\ref{thm:Goodstein} is that an $n$-ary polynomial operation of a bounded distributive lattice is completely and uniquely determined by its restriction to $\{0, 1\}^n$.
\end{remark}

\begin{fact}
\label{fact:range}
Let $f \colon A^n \to A$ be a polynomial operation of a bounded distributive lattice $\mathbf{A}$.
Then $f$ is an $(a,b)$-truncated term operation of $\mathbf{A}$ if and only if $\range f|_{\{0,1\}^n} \subseteq \{a, b\}$.
\end{fact}

For lattice elements $a, b \in A$, we denote by $[a, b]$ the interval $\{x \in A : a \leq x \leq b\}$.
We also write $\vect{0} = (0, \dots, 0)$ and $\vect{1} = (1, \dots, 1)$.

\begin{theorem}[{Couceiro, Marichal~\cite{CouMar}}]
\label{thm:latticepolynomial}
Let $\mathbf{A} = (A; \wedge, \vee, 0, 1)$ be a bounded distributive lattice, let $n \geq 1$, and let $f \colon A^n \to A$ be a function preserving the lattice order of $\mathbf{A}$. Then $f$ is a polynomial operation of $\mathbf{A}$ if and only if for every $c \in [f(\vect{0}), f(\vect{1})]$, the following identities hold:
\begin{align*}
f(x_1 \wedge c, \dots, x_n \wedge c) &= f(x_1, \dots, x_n) \wedge c, \\
f(x_1 \vee c, \dots, x_n \vee c) &= f(x_1, \dots, x_n) \vee c.
\end{align*}
\end{theorem}

\begin{theorem}[{\cite[Proposition~3.16, Example~3.17]{LehtonenSymmetric}}]
\label{thm:monrecognizable}
Let $(A; \leq_A)$ and $(B; \leq_B)$ be partially ordered sets. The class of order-preserving functions from $A$ to $B$ of arity at least $\card{A} + 2$ is recognizable.
\end{theorem}

\begin{proposition}
\label{prop:recognizable}
Let $\mathbf{A} = (A; \wedge, \vee, 0, 1)$ be a bounded distributive lattice.
\begin{enumerate}[\rm (a)]
\item\label{item:poly}
The class of polynomial operations of $\mathbf{A}$ of arity at least $\card{A} + 2$ is recognizable.
\item\label{item:trunc}
For any $a, b \in A$ with $a < b$, the class of $(a,b)$-truncated term operations of $\mathbf{A}$ of arity at least $\card{A} + 2$ is recognizable.
\end{enumerate}
\end{proposition}

\begin{proof}
\eqref{item:poly}
Let $f \colon A^n \to A$ be a polynomial operation of $\mathbf{A}$ and assume that $n \geq \card{A} + 2$.
Let $g \colon A^n \to A$ be a reconstruction of $f$.
Since the polynomial operations of a lattice are order-preserving, also $g$ is order-preserving by Theorem~\ref{thm:monrecognizable}.
Suppose, on the contrary, that $g$ is not a polynomial operation of $\mathbf{A}$.
By Theorem~\ref{thm:latticepolynomial}, there exist $c \in [g(\vect{0}), g(\vect{1})]$ and $(a_1, \dots, a_n) \in A^n$ such that $g(a_1 \wedge c, \dots, a_n \wedge c) \neq g(a_1, \dots, a_n) \wedge c$ or $g(a_1 \vee c, \dots, a_n \vee c) \neq g(a_1, \dots, a_n) \vee c$.
Since $n > \card{A}$, there exist indices $i, j \in \nset{n}$ such that $i < j$ and $a_i = a_j$.
Taking $I = \{i, j\}$, we have
\begin{multline*}
g_I(a_1 \wedge c, \dots, a_{j-1} \wedge c, a_{j+1} \wedge c, \dots, a_n \wedge c)
= g(a_1 \wedge c, \dots, a_n \wedge c) \\
\neq g(a_1, \dots, a_n) \wedge c
= g_I(a_1, \dots, a_{j-1}, a_{j+1}, \dots, a_n) \wedge c
\end{multline*}
or, similarly,
\[
g_I(a_1 \vee c, \dots, a_{j-1} \vee c, a_{j+1} \vee c, \dots, a_n \vee c)
\neq g_I(a_1, \dots, a_{j-1}, a_{j+1}, \dots, a_n) \vee c.
\]
Since $g_I(\vect{0}) = g(\vect{0})$ and $g_I(\vect{1}) = g(\vect{1})$, we have $c \in [g_I(\vect{0}), g_I(\vect{1})]$, and it follows from Theorem~\ref{thm:latticepolynomial} that $g_I$ is not a polynomial operation of $\mathbf{A}$.
But then $f$ and $g$ cannot have the same deck, because all identification minors of any lattice polynomial operation are lattice polynomial operations. We have reached a contradiction.

\eqref{item:trunc}
Let $f \colon A^n \to A$ be an $(a,b)$-truncated term operation of $\mathbf{A}$ and assume that $n \geq \card{A} + 2$.
Let $g \colon A^n \to A$ be a reconstruction of $f$.
By part~\eqref{item:poly}, $g$ is a polynomial operation of $\mathbf{A}$.
Suppose, on the contrary, that $g$ is not an $(a,b)$-truncated term operation.
By Fact~\ref{fact:range}, there exists a tuple $(a_1, \dots, a_n) \in \{0, 1\}^n$ such that $g(a_1, \dots, a_n) \notin \{a, b\}$.
Since $n \geq 3$, there exist indices $i, j \in \nset{n}$ such that $i < j$ and $a_i = a_j$.
Taking $I = \{i, j\}$, we have
\[
g_I(a_1, \dots, a_{j-1}, a_{j+1}, \dots, a_n) = g(a_1, \dots, a_n)
\notin \{a, b\},
\]
which implies that $g_I$ is not an $(a,b)$-truncated term operation of $\mathbf{A}$.
But then $f$ and $g$ cannot have the same deck, because all identification minors of $(a,b)$-truncated term operations are $(a,b)$-truncated term operations. We have reached a contradiction.
\end{proof}

\begin{proposition}
\label{prop:termopreconstr}
Let $\mathbf{A} = (A; \wedge, \vee, 0, 1)$ be a bounded distributive lattice, and let $a, b \in A$ be elements satisfying $a < b$.
Let $\Sp{\Sgena}$ be a Sperner system over $\nset{n}$, and assume that $n \geq \card{A} + 2$.
Then $\Sp{\Sgena}$ is reconstructible if and only if the $(a,b)$-truncated term operation $(t^{ab}_{\Sp{\Sgena}})^\mathbf{A}$ is reconstructible.
\end{proposition}

\begin{proof}
Assume first that $\Sp{\Sgena}$ is not reconstructible, and let $\Sp{\Sgenb}$ be a nonisomorphic reconstruction of $\Sp{\Sgena}$. Then the functions $(t^{ab}_{\Sp{\Sgena}})^\mathbf{A}$ and $(t^{ab}_{\Sp{\Sgenb}})^\mathbf{A}$ are not equivalent but they have the same deck, i.e., $(t^{ab}_{\Sp{\Sgena}})^\mathbf{A}$ is not reconstructible.

Assume then that $\Sp{\Sgena}$ is reconstructible.
Then any reconstruction of $(t^{ab}_{\Sp{\Sgena}})^\mathbf{A}$ that is an $(a,b)$-truncated term operation of $\mathbf{A}$ is equivalent to $(t^{ab}_{\Sp{\Sgena}})^\mathbf{A}$.
By Proposition~\ref{prop:recognizable}\eqref{item:trunc}, every reconstruction of $(t^{ab}_{\Sp{\Sgena}})^\mathbf{A}$ is an $(a,b)$-truncated term operation of $\mathbf{A}$.
Therefore, $(t^{ab}_{\Sp{\Sgena}})^\mathbf{A}$ is reconstructible.
\end{proof}


\section{Nonreconstructible Sperner systems}
\label{sec:main}

We are going to construct a few different infinite families of pairs of strongly hypomorphic nonisomorphic Sperner systems.
As explained in Section~\ref{sub:relationship}, for any bounded distributive lattice, there is a one-to-one correspondence between the Sperner systems over an $m$-element set and the $m$-ary term operations of the lattice. The existence of infinite families of nonisomorphic hypomorphic pairs of Sperner systems shows that the class of lattice term operations is not weakly reconstructible and there exist nonreconstructible lattice term functions of arbitrarily large arities.

An important special case of bounded distributive lattices is the two-element lattice $\mathbf{B} = (\{0, 1\}; \wedge, \vee, 0, 1)$.
The monotone Boolean functions are precisely the term operations of $\mathbf{B}$.
We will construct families of Sperner systems in different ways so that the associated Boolean functions belong to certain clones, namely to the clone $SM$ of self-dual monotone functions and to the clone $M_c U_\infty$ of monotone constant-preserving $1$-separating functions.

As a fundamental building block that will be used in all the constructions that follow, we first define families $\Sp{\Sbuildall}^m_1$ and $\Sp{\Sbuildall}^m_2$ ($m \geq 3$) of Sperner systems.
For each $m$, the systems $\Sp{\Sbuildall}^m_1$ and $\Sp{\Sbuildall}^m_2$ are isomorphic and hence hypomorphic, but we will then add to both some new blocks that will break the isomorphism but maintain the hypomorphism. In fact, in each case, the resulting pairs of Sperner systems will be not only hypomorphic but also strongly hypomorphic.

The ground sets of the Sperner systems that we will construct are subsets of the set $\mathbb{N} \times \{0, 1\}$.
In order to simplify exposition, we will identify $n$ with $(n, 0)$ and we will write $n'$ for $(n, 1)$, for each $n \in \mathbb{N}$.
For any subset $S$ of $\mathbb{N}$, we write $S'$ for the set $\{n' : n \in S\}$.
Denote $E_m := \nset{m} \cup \nset{m}'$.

In what follows, we will perform arithmetic with elements of $\nset{m}$. It will be understood that $+$ denotes modulo-$m$ addition, i.e., for $a, b \in \nset{m}$, we write $a + b$ to denote the unique element $c \in \nset{m}$ such that the sum of $a$ and $b$ is congruent to $c$ modulo $m$. The modulus $m$ will be clear from the context.

Let $J \subseteq E_m$. We denote the complement of $J$ with respect to $E_m$ by $\complementX{J}$.
For a set system $\Sp{\Sgena}$, we write $\complementX{\Sp{\Sgena}}$ for $\{\complementX{S} : S \in \Sp{\Sgena}\}$.

In the sequel, we will often refer to certain permutations of $E_m$. For $i \in \nset{m}$, the \emph{transposition} of $i$ and $i'$ is the permutation $(i \; i')$, and it is denoted by $\tau_i$. The \emph{rotation} $\rho$ is the permutation $(1 \; 2 \; \cdots \; m) (1' \; 2' \; \cdots \; m')$, a composition of two disjoint $m$-cycles.

For a set system $\Sp{\Sgena}$ over $E_m$, we write $\rot{\Sp{\Sgena}}$ for $\bigcup_{i \in \nset{m}} \rho^i(\Sp{\Sgena})$.
For $X \subseteq \nset{m}$ and $q \in \nset{m}$, we write $X + q$ for $\{x + q : x \in X\}$.

\subsection{The basic building blocks $\Sp{\Sbuildall}^m_1$ and $\Sp{\Sbuildall}^m_2$}

\begin{definition}
Let $m$ be an integer at least $2$.
For $J \subseteq \nset{m}$, we denote $\Sbuildall^m_J := J \cup (\nset{m} \setminus J)'$.
Let $\Sp{\Sbuildall}^m_1$ and $\Sp{\Sbuildall}^m_2$ be the following Sperner systems over $E_m$:
\begin{align*}
\Sp{\Sbuildall}^m_1 &:= \{\Sbuildall^m_J : \text{$J \subseteq \nset{m}$, $\card{J}$ odd}\}, \\
\Sp{\Sbuildall}^m_2 &:= \{\Sbuildall^m_J : \text{$J \subseteq \nset{m}$, $\card{J}$ even}\}.
\end{align*}
\end{definition}

\begin{example}
For $m \in \{3, 4\}$, we have
\begin{align*}
\Sp{\Sbuildall}^3_1 &= \{\{1, 2, 3\}, \{1, 2', 3'\}, \{1', 2, 3'\}, \{1', 2', 3\}\}, \\
\Sp{\Sbuildall}^3_2 &= \{\{1', 2', 3'\}, \{1', 2, 3\}, \{1, 2', 3\}, \{1, 2, 3'\}\}, \\
\Sp{\Sbuildall}^4_1 &= \{\{1, 2', 3', 4'\}, \{1', 2, 3', 4'\}, \{1', 2', 3, 4'\}, \{1', 2', 3', 4\}, \\
&\phantom{{}= \{ \hspace{0.1pt}}
           \{1', 2, 3, 4\}, \{1, 2', 3, 4\}, \{1, 2, 3', 4\}, \{1, 2, 3, 4'\}\}, \\
\Sp{\Sbuildall}^4_2 &= \{\{1, 2, 3, 4\}, \{1, 2, 3', 4'\}, \{1, 2', 3, 4'\}, \{1, 2', 3', 4\}, \\
&\phantom{{}= \{ \hspace{0.1pt}}
           \{1', 2, 3, 4'\}, \{1', 2, 3', 4\}, \{1', 2', 3, 4\}, \{1', 2', 3', 4'\}\}. \\
\end{align*}
See Table~\ref{table:G34} for a more visual presentation of $\Sp{\Sbuildall}^3_1$, $\Sp{\Sbuildall}^3_2$, $\Sp{\Sbuildall}^4_1$ and $\Sp{\Sbuildall}^4_2$.
\end{example}

\begin{table}
\[
\begin{array}{cc}
\begin{array}{c}
\Sp{\Sbuildall}^3_1 \\[1ex]
\begin{array}{ccc|ccc}
1 & 2 & 3 &    &    &    \\
1 &   &   &    & 2' & 3' \\
  & 2 &   & 1' &    & 3' \\
  &   & 3 & 1' & 2' &
\end{array}
\end{array}
&
\begin{array}{c}
\Sp{\Sbuildall}^3_2 \\[1ex]
\begin{array}{ccc|ccc}
  &   &   & 1' & 2' & 3' \\
  & 2 & 3 & 1' &    &    \\
1 &   & 3 &    & 2' &    \\
1 & 2 &   &    &    & 3'
\end{array}
\end{array}
\\
\\[2ex]
\begin{array}{c}
\Sp{\Sbuildall}^4_1 \\[1ex]
\begin{array}{cccc|cccc}
1 &   &   &   &    & 2' & 3' & 4' \\
  & 2 &   &   & 1' &    & 3' & 4' \\
  &   & 3 &   & 1' & 2' &    & 4' \\
  &   &   & 4 & 1' & 2' & 3' &    \\
  & 2 & 3 & 4 & 1' &    &    &    \\
1 &   & 3 & 4 &    & 2' &    &    \\
1 & 2 &   & 4 &    &    & 3' &    \\
1 & 2 & 3 &   &    &    &    & 4'
\end{array}
\end{array}
&
\begin{array}{c}
\Sp{\Sbuildall}^4_2 \\[1ex]
\begin{array}{cccc|cccc}
  &   &   &   & 1' & 2' & 3' & 4' \\
1 & 2 &   &   &    &    & 3' & 4' \\
1 &   & 3 &   &    & 2' &    & 4' \\
1 &   &   & 4 &    & 2' & 3' &    \\
  & 2 & 3 &   & 1' &    &    & 4' \\
  & 2 &   & 4 & 1' &    & 3' &    \\
  &   & 3 & 4 & 1' & 2' &    &    \\
1 & 2 & 3 & 4 &    &    &    &    \\
\end{array}
\end{array}
\end{array}
\]

\bigskip
\caption{Sperner systems $\Sp{\Sbuildall}^3_1$, $\Sp{\Sbuildall}^3_2$, $\Sp{\Sbuildall}^4_1$ and $\Sp{\Sbuildall}^4_2$.}
\label{table:G34}
\end{table}

\begin{remark}
\label{rem:Ginv}
Both $\Sp{\Sbuildall}^m_1$ and $\Sp{\Sbuildall}^m_2$ are $m$-homogeneous and invariant under the rotation $\rho$.
\end{remark}

\begin{remark}
\label{rem:Gcomplements}
For odd $m$, we have $\complementX{\Sp{\Sbuildall}^m_1} = \Sp{\Sbuildall}^m_2$, i.e., $\Sp{\Sbuildall}^m_2$ is exactly the set of the complements of the blocks of $\Sp{\Sbuildall}^m_1$ with respect to the set $E_m$.
For even $m$, we have $\complementX{\Sp{\Sbuildall}^m_i} = \Sp{\Sbuildall}^m_i$ for $i \in \{1, 2\}$, i.e., the complement of each block of $\Sp{\Sbuildall}^m_i$ is a block of $\Sp{\Sbuildall}^m_i$.
\end{remark}

\begin{remark}
\label{rem:G1G2isomorphic}
An important thing to notice is that $\Sp{\Sbuildall}^m_1$ and $\Sp{\Sbuildall}^m_2$ are isomorphic, an isomorphism being given by the transposition $\tau_i$ for any $i \in \nset{m}$, i.e., $\tau_i(\Sp{\Sbuildall}^m_1) = \Sp{\Sbuildall}^m_2$ for every $i \in \nset{m}$.
\end{remark}

\begin{remark}
\label{rem:G1G2iiprime}
A consequence of Remark~\ref{rem:G1G2isomorphic} is that $(\Sp{\Sbuildall}^m_1)_{\{i, i'\}} = (\Sp{\Sbuildall}^m_2)_{\{i, i'\}}$ for all $i \in \nset{m}$.
\end{remark}

We say that a subset $S$ of $E_m$ is \emph{unprimed odd} (\emph{unprimed even}) if $\card{S \cap \nset{m}}$ is odd (even, respectively).

\begin{remark}
The blocks of $\Sp{\Sbuildall}^m_1$ are unprimed odd.
The blocks of $\Sp{\Sbuildall}^m_2$ are unprimed even.
\end{remark}

\subsection{Construction for monotone functions}

As explained in the beginning of this section, we now define another family $\Sp{\Sbuildmon}^m$ of Sperner systems. Adding the blocks of $\Sp{\Sbuildmon}^m$ to both $\Sp{\Sbuildall}^m_1$ and $\Sp{\Sbuildall}^m_2$ will break the isomorphism of $\Sp{\Sbuildall}^m_1$ and $\Sp{\Sbuildall}^m_2$ but the resulting set systems will nevertheless be hypomorphic. In this way we obtain our first example of an infinite family of pairs of strongly hypomorphic nonisomorphic Sperner systems.

\begin{definition}
\label{def:monotone}
For $m \geq 3$ and $p \in \nset{m}$, let $\Sbuildmon^m_p := E_m \setminus \{p, p', (p + 1)'\}$. (Recall that $+$ denotes modulo-$m$ addition.)
Let $\Sp{\Sbuildmon}^m := \{\Sbuildmon^m_p : p \in \nset{m}\}$.
For $i \in \{1, 2\}$, define $\Sp{\Smon}^m_i := \Sp{\Sbuildall}^m_i \cup \Sp{\Sbuildmon}^m$.
\end{definition}

\begin{remark}
Note that $\tau_s(\Sbuildmon^m_p) = \Sbuildmon^m_p$ for every $s \in \nset{m} \setminus \{p + 1\}$ but $\tau_{p+1}(\Sbuildmon^m_p) \neq \Sbuildmon^m_p$; in fact, $\tau_{p+1}(\Sbuildmon^m_p) \notin \Sp{\Sbuildmon}^m$.
\end{remark}

\begin{remark}
Note that $\Sp{\Sbuildmon}^m$ is a $(2m - 3)$-homogeneous Sperner system and it is invariant under the rotation $\rho$ but not under any transposition $\tau_i$ with $i \in \nset{m}$.
Consequently, in view of Remark~\ref{rem:Ginv}, both $\Sp{\Smon}^m_1$ and $\Sp{\Smon}^m_2$ are invariant under the rotation $\rho$.
\end{remark}

\begin{remark}
Observe that for each $i \in \{1, 2\}$, the set systems $\Sp{\Sbuildmon}^m$ and $\Sp{\Sbuildall}^m_i$ are independent, i.e., no block of one of these systems is a subset of a member of the other; hence $\Sp{\Smon}^m_i$ is indeed a Sperner system.
For $m = 3$, this fact is easy to check by hand.
For $m > 3$, it is obvious from cardinalities that no block of $\Sp{\Sbuildmon}^m$ is a subset of a block of $\Sp{\Sbuildall}^m_i$. It also holds that no block of $\Sp{\Sbuildall}^m_i$ is a subset of any block of $\Sp{\Sbuildmon}^m$, because each block $\Sbuildmon^m_p$ ($p \in \nset{m}$) of $\Sp{\Sbuildmon}^m$ contains neither $p$ nor $p'$, but every block of $\Sp{\Sbuildall}^m_i$ contains either $p$ or $p'$.
\end{remark}

\begin{example}
For $m \in \{3, 4\}$, we have
\begin{align*}
\Sp{\Sbuildmon}^3 &= \{\{2, 3, 3'\}, \{1, 3, 1'\}, \{1, 2, 2'\}\}, \\
\Sp{\Sbuildmon}^4 &= \{\{2, 3, 4, 3', 4'\}, \{1, 3, 4, 1', 4'\}, \{1, 2, 4, 1', 2'\}, \{1, 2, 3, 2', 3'\}\}.
\end{align*}
Taking the unions $\Sp{\Sbuildall}^m_i \cup \Sp{\Sbuildmon}^m$, we obtain
\begin{align*}
\Sp{\Smon}^3_1 &= \{\{2, 3, 3'\}, \{1, 3, 1'\}, \{1, 2, 2'\},
           \{1, 2, 3\}, \{1, 2', 3'\}, \{1', 2, 3'\}, \{1', 2', 3\}\}, \\
\Sp{\Smon}^3_2 &= \{\{2, 3, 3'\}, \{1, 3, 1'\}, \{1, 2, 2'\},
           \{1', 2', 3'\}, \{1', 2, 3\}, \{1, 2', 3\}, \{1, 2, 3'\}\}, \\
\Sp{\Smon}^4_1 &= \{\{2, 3, 4, 3', 4'\}, \{1, 3, 4, 1', 4'\}, \{1, 2, 4, 1', 2'\}, \{1, 2, 3, 2', 3'\}, \\
&\phantom{{}= \{ \hspace{0.1pt}}
           \{1, 2', 3', 4'\}, \{1', 2, 3', 4'\}, \{1', 2', 3, 4'\}, \{1', 2', 3', 4\}, \\
&\phantom{{}= \{ \hspace{0.1pt}}
           \{1', 2, 3, 4\}, \{1, 2', 3, 4\}, \{1, 2, 3', 4\}, \{1, 2, 3, 4'\}\}, \\
\Sp{\Smon}^4_2 &= \{\{2, 3, 4, 3', 4'\}, \{1, 3, 4, 1', 4'\}, \{1, 2, 4, 1', 2'\}, \{1, 2, 3, 2', 3'\}, \\
&\phantom{{}= \{ \hspace{0.1pt}}
           \{1, 2, 3, 4\}, \{1, 2, 3', 4'\}, \{1, 2', 3, 4'\}, \{1, 2', 3', 4\}, \\
&\phantom{{}= \{ \hspace{0.1pt}}
           \{1', 2, 3, 4'\}, \{1', 2, 3', 4\}, \{1', 2', 3, 4\}, \{1', 2', 3', 4'\}\}.
\end{align*}
See Tables~\ref{table:F34} and~\ref{table:M34} for a more visual presentation of $\Sp{\Sbuildmon}^3$, $\Sp{\Sbuildmon}^4$, $\Sp{\Smon}^3_1$, $\Sp{\Smon}^3_2$, $\Sp{\Smon}^4_1$ and $\Sp{\Smon}^4_2$.

\begin{table}
\[
\begin{array}{cc}
  \begin{array}[t]{ccc|ccc}
    \multicolumn{6}{c}{\Sp{\Sbuildmon}^3} \\[1ex]
      & \color{red} 2 & 3 &   &   & 3' \\
    1 &   & \color{red} 3 & 1' &   &   \\
    \color{red} 1 & 2 &   &   & 2' &
  \end{array}
&
  \begin{array}[t]{cccc|cccc}
  \multicolumn{8}{c}{\Sp{\Sbuildmon}^4} \\[1ex]
      & \color{red} 2 & 3 & 4 &    &    & 3' & 4' \\
    1 &   & \color{red} 3 & 4 & 1' &    &    & 4' \\
    1 & 2 &   & \color{red} 4 & 1' & 2' &    &    \\
    \color{red} 1 & 2 & 3 &   &    & 2' & 3' &    
  \end{array}
\end{array}
\]

\bigskip
\caption{Sperner systems $\Sp{\Sbuildmon}^3$ and $\Sp{\Sbuildmon}^4$.}
\label{table:F34}
\end{table}

\begin{table}
\[
\begin{array}{cc}
  \begin{array}{cccc|ccc}
  & \multicolumn{6}{c}{\Sp{\Smon}^3_1} \\[1ex]
\multirow{4}{*}{$\Sp{\Sbuildall}^3_1$}
   & 1 & 2 & 3 &    &    &    \\
   & 1 &   &   &    & 2' & 3' \\
   &   & 2 &   & 1' &    & 3' \\
   &   &   & 3 & 1' & 2' &    \\
  \hline
\multirow{3}{*}{$\Sp{\Sbuildmon}^3$}
   &   & 2 & 3 &    &    & 3' \\
   & 1 &   & 3 & 1' &    &    \\
   & 1 & 2 &   &    & 2' &
  \end{array}
&
  \begin{array}{cccc|ccc}
  & \multicolumn{6}{c}{\Sp{\Smon}^3_2} \\[1ex]
\multirow{4}{*}{$\Sp{\Sbuildall}^3_2$}
   &   &   &   & 1' & 2' & 3' \\
   &   & 2 & 3 & 1' &    &    \\
   & 1 &   & 3 &    & 2' &    \\
   & 1 & 2 &   &    &    & 3' \\
  \hline
\multirow{3}{*}{$\Sp{\Sbuildmon}^3$}
   &   & 2 & 3 &    &    & 3' \\
   & 1 &   & 3 & 1' &    &    \\
   & 1 & 2 &   &    & 2' &
  \end{array}
\\ \\[1ex]
  \begin{array}{ccccc|cccc}
& \multicolumn{8}{c}{\Sp{\Smon}^4_1} \\[1ex]
\multirow{8}{*}{$\Sp{\Sbuildall}^4_1$}
& 1 &   &   &   &    & 2' & 3' & 4' \\
&   & 2 &   &   & 1' &    & 3' & 4' \\
&   &   & 3 &   & 1' & 2' &    & 4' \\
&   &   &   & 4 & 1' & 2' & 3' &    \\
&   & 2 & 3 & 4 & 1' &    &    &    \\
& 1 &   & 3 & 4 &    & 2' &    &    \\
& 1 & 2 &   & 4 &    &    & 3' &    \\
& 1 & 2 & 3 &   &    &    &    & 4' \\
\hline
\multirow{4}{*}{$\Sp{\Sbuildmon}^4$}
&   & 2 & 3 & 4 &    &    & 3' & 4' \\
& 1 &   & 3 & 4 & 1' &    &    & 4' \\
& 1 & 2 &   & 4 & 1' & 2' &    &    \\
& 1 & 2 & 3 &   &    & 2' & 3' &    
  \end{array}
&
  \begin{array}{ccccc|cccc}
& \multicolumn{8}{c}{\Sp{\Smon}^4_2} \\[1ex]
\multirow{8}{*}{$\Sp{\Sbuildall}^4_2$}
&   &   &   &   & 1' & 2' & 3' & 4' \\
& 1 & 2 &   &   &    &    & 3' & 4' \\
& 1 &   & 3 &   &    & 2' &    & 4' \\
& 1 &   &   & 4 &    & 2' & 3' &    \\
&   & 2 & 3 &   & 1' &    &    & 4' \\
&   & 2 &   & 4 & 1' &    & 3' &    \\
&   &   & 3 & 4 & 1' & 2' &    &    \\
& 1 & 2 & 3 & 4 &    &    &    &    \\
\hline
\multirow{4}{*}{$\Sp{\Sbuildmon}^4$}
&   & 2 & 3 & 4 &    &    & 3' & 4' \\
& 1 &   & 3 & 4 & 1' &    &    & 4' \\
& 1 & 2 &   & 4 & 1' & 2' &    &    \\
& 1 & 2 & 3 &   &    & 2' & 3' &    
  \end{array}
\end{array}
\]

\bigskip
\caption{Sperner systems $\Sp{\Smon}^3_1$, $\Sp{\Smon}^3_2$, $\Sp{\Smon}^4_1$ and $\Sp{\Smon}^4_2$.}
\label{table:M34}
\end{table}
\end{example}

It remains to show that $\Sp{\Smon}^m_1$ and $\Sp{\Smon}^m_2$ are nonisomorphic and strongly hypomorphic.

\begin{proposition}
\label{prop:Cnoniso}
For every $m \geq 3$, the Sperner systems $\Sp{\Smon}^m_1$ and $\Sp{\Smon}^m_2$ are nonisomorphic.
\end{proposition}

\begin{proof}
Suppose, on the contrary, that there exists a permutation $\sigma \colon E_m \to E_m$ such that $\sigma(\Sp{\Smon}^m_1) = \Sp{\Smon}^m_2$.
Let us first count in how many blocks of $\Sp{\Smon}^m_1$ and $\Sp{\Smon}^m_2$ each element of $E_m$ occurs.
In $\Sp{\Sbuildall}^m_1$ and $\Sp{\Sbuildall}^m_2$, each element of $E_m$ occurs in exactly $2^{m-2}$ blocks.
In $\Sp{\Sbuildmon}^m$, each element $x \in \nset{m}$ occurs in $m - 1$ blocks, while each element $x \in \nset{m}'$ occurs in $m - 2$ blocks.
Thus, in both $\Sp{\Smon}^m_1$ and $\Sp{\Smon}^m_2$, each element $x \in \nset{m}$ occurs in $2^{m-2} + m - 1$ blocks and each element $x \in \nset{m}'$ occurs in $2^{m-2} + m - 2$ blocks.
It follows that $\sigma$ maps the set $\nset{m}$ onto itself and it maps the set $\nset{m}'$ onto itself.

Consider a block $\Sbuildall^m_J$ of $\Sp{\Sbuildall}^m_1$ for some $J \subseteq \nset{m}$ with $\card{J}$ odd.
Since $\Sbuildall^m_J$ is unprimed odd, $\sigma(\Sbuildall^m_J)$ is also unprimed odd, so it cannot be a block of $\Sp{\Sbuildall}^m_2$.
The set $\sigma(\Sbuildall^m_J)$ is not a block of $\Sp{\Sbuildmon}^m$ either: for $m > 3$ this is obvious from cardinalities; for $m = 3$ this holds because the blocks of $\Sp{\Sbuildmon}^m$ are unprimed even.
We have reached a contradiction.
\end{proof}

\begin{proposition}
\label{prop:Mhypo}
For all $m \geq 3$, the Sperner systems $\Sp{\Smon}^m_1$ and $\Sp{\Smon}^m_2$ are strongly hypomorphic.
\end{proposition}

\begin{proof}
Let $I \in \couples[E_m]$.
Assume first that $I = \{p, p'\}$ for some $p \in \nset{m}$.
Since $\tau_p(\Sp{\Sbuildall}^m_1) = \Sp{\Sbuildall}^m_2$, it follows that $(\Sp{\Sbuildall}^m_1)_I = (\Sp{\Sbuildall}^m_2)_I$.
Consequently,
\[
(\Sp{\Smon}^m_1)_I = (\Sp{\Sbuildall}^m_1 \cup \Sp{\Sbuildmon}^m)_I = (\Sp{\Sbuildall}^m_1)_I \cup (\Sp{\Sbuildmon}^m)_I = (\Sp{\Sbuildall}^m_2)_I \cup (\Sp{\Sbuildmon}^m)_I = (\Sp{\Sbuildall}^m_2 \cup \Sp{\Sbuildmon}^m)_I = (\Sp{\Smon}^m_2)_I;
\]
hence $(\Sp{\Smon}^m_1)^*_I = (\Sp{\Smon}^m_2)^*_I$.

Assume then that $I$ is not of the form $\{p, p'\}$ for some $p \in \nset{m}$. We will split the analysis into four cases. In each case, we will specify an element $r \in \nset{m}$ and sets $S_1, S_2 \subseteq \nset{m}$ such that $\{r + 1, (r + 1)'\} \cap I = \emptyset$ and $(\Sbuildall^m_{S_i})_I \subseteq (\Sbuildmon^m_r)_I$ for each $i \in \{1, 2\}$. This means that we can ignore $(\Sbuildmon^m_r)_I$ when we consider the minimal blocks of $(\Sp{\Smon}^m_1)_I$ and $(\Sp{\Smon}^m_2)_I$. The transposition of $r + 1$ and $(r + 1)'$ keeps the remaining blocks of $(\Sp{\Sbuildmon}^m)_I$ (i.e., the blocks $(\Sbuildmon^m_s)_I$ for $s \in \nset{m} \setminus \{r\}$) unchanged and maps $(\Sp{\Sbuildall}^m_1)_I$ onto $(\Sp{\Sbuildall}^m_2)_I$.
Consequently, $(\Sp{\Smon}^m_1)^*_I \equiv (\Sp{\Smon}^m_2)^*_I$, an isomorphism being given by $\tau_{r + 1}$.

\begin{itemize}
\item Case 1: $I = \{p, q\}$ for some distinct $p, q \in \nset{m}$. Swapping $p$ and $q$ if necessary, we may assume that $q \neq p + 1$.
Choose $r := p$ and, according to the parity of $m$, choose $S_1$ and $S_2$ among the sets $\nset{m}$ and $\nset{m} \setminus \{p + 2\}$.

\item Case 2: $I = \{p, q'\}$ for some distinct $p, q \in \nset{m}$ such that $q \neq p + 1$.
Choose $r := p$ and, according to the parity of $m$, choose $S_1$ and $S_2$ among the sets $\nset{m}$ and $\nset{m} \setminus \{q\}$.

\item Case 3: $I = \{p, q'\}$ for some distinct $p, q \in \nset{m}$ such that $p \neq q + 1$.
Choose $r := q$ and, according to the parity of $m$, choose $S_1$ and $S_2$ among the sets $\nset{m} \setminus \{q\} \cup \{q'\}$ and $\nset{m} \setminus \{p, q\}$.

\item Case 4: $I = \{p', q'\}$ for some distinct $p, q \in \nset{m}$. Swapping $p$ and $q$ if necessary, we may assume that $q \neq p + 1$.
Choose $r := p$.
If $m > 3$, then, according to the parity of $m$, choose $S_1$ and $S_2$ among the sets $\nset{m} \setminus \{p\}$ and $\nset{m} \setminus \{p, s\}$ for any $s \in \nset{m} \setminus \{p, p + 1, q\}$.
If $m = 3$, then choose $S_1 := \{p + 1\}$, $S_2 := \{p + 1, q\}$.
\end{itemize}

These four cases exhaust all possibilities, and we conclude that $(\Sp{\Smon}^m_1)^*_I \equiv (\Sp{\Smon}^m_2)^*_I$ for every $I \in \couples[E_m]$.
\end{proof}

\begin{example}
In order to illustrate Proposition~\ref{prop:Mhypo}, let us consider the case $m = 4$.

We present in Table~\ref{table:minorsM4} the set systems $(\Sp{\Smon}^4_i)_I$ for $i \in \{1, 2\}$ and $I \in \{\{1, 1'\}, \{1, 2\},\linebreak[0] \{1, 3\}, \{1, 2'\}, \{1, 3'\}, \{1, 4'\}, \{1', 2'\}, \{1', 3'\}\}$; the remaining cases are similar by the invariance of $\Sp{\Smon}^4_1$ and $\Sp{\Smon}^4_2$ under the permutation $(1 \; 2 \; 3 \; 4)(1' \; 2' \; 3' \; 4')$.
In each set system except for $(\Sp{\Smon}^4_i)_{\{1, 1'\}}$, $i \in \{1, 2\}$, the block written in bold and red is not minimal, because it includes the block written in italics and blue.
The non-bold non-red blocks of $(\Sp{\Smon}^4_1)_I$ are mapped onto the non-bold non-red blocks of $(\Sp{\Smon}^4_2)_I$ by the permutation $\sigma$.
Hence $(\Sp{\Smon}^4_1)^*_I \equiv (\Sp{\Smon}^4_2)^*_I$.

\newcommand{\Mpair}[2]{\begin{array}{c@{\quad}c}#1&#2\end{array}}

\newcommand{\MxaA}
{
\begin{array}{ccccccc}
1 &   &   &   & 2' & 3' & 4' \\
1 & 2 &   &   &    & 3' & 4' \\
1 &   & 3 &   & 2' &    & 4' \\
1 &   &   & 4 & 2' & 3' &    \\
1 & 2 & 3 & 4 &    &    &    \\
1 &   & 3 & 4 & 2' &    &    \\
1 & 2 &   & 4 &    & 3' &    \\
1 & 2 & 3 &   &    &    & 4' \\
\hline
  & 2 & 3 & 4 &    & 3' & 4' \\
1 &   & 3 & 4 &    &    & 4' \\
1 & 2 &   & 4 & 2' &    &    \\
1 & 2 & 3 &   & 2' & 3' &    \\
\end{array}
}

\newcommand{\MyaA}
{
\begin{array}{ccccccc}
1 &   &   &   & 2' & 3' & 4' \\
1 & 2 &   &   &    & 3' & 4' \\
1 &   & 3 &   & 2' &    & 4' \\
1 &   &   & 4 & 2' & 3' &    \\
1 & 2 & 3 &   &    &    & 4' \\
1 & 2 &   & 4 &    & 3' &    \\
1 &   & 3 & 4 & 2' &    &    \\
1 & 2 & 3 & 4 &    &    &    \\
\hline
  & 2 & 3 & 4 &    & 3' & 4' \\
1 &   & 3 & 4 &    &    & 4' \\
1 & 2 &   & 4 & 2' &    &    \\
1 & 2 & 3 &   & 2' & 3' &    \\
\end{array}
}

\newcommand{\MaA}{\Mpair{\MxaA}{\MyaA}}

\newcommand{\Mxab}
{
\begin{array}{ccccccc}
1 &   &   &    & 2' & 3' & 4' \\
1 &   &   & 1' &    & 3' & 4' \\
  & 3 &   & 1' & 2' &    & 4' \\
  &   & 4 & 1' & 2' & 3' &    \\
1 & 3 & 4 & 1' &    &    &    \\
1 & 3 & 4 &    & 2' &    &    \\
1 &   & 4 &    &    & 3' &    \\
\color{blue}\it 1 & \color{blue}\it 3 &   &    &    &    & \color{blue}\it 4' \\
\hline
1 & 3 & 4 &    &    & 3' & 4' \\
\color{red}\bf 1 & \color{red}\bf 3 & \color{red}\bf 4 & \color{red}\bf 1' &    &    & \color{red}\bf 4' \\
1 &   & 4 & 1' & 2' &    &    \\
1 & 3 &   &    & 2' & 3' &    \\
\end{array}
}

\newcommand{\Myab}
{
\begin{array}{ccccccc}
  &   &   & 1' & 2' & 3' & 4' \\
1 &   &   &    &    & 3' & 4' \\
1 & 3 &   &    & 2' &    & 4' \\
1 &   & 4 &    & 2' & 3' &    \\
1 & 3 &   & 1' &    &    & 4' \\
1 &   & 4 & 1' &    & 3' &    \\
  & 3 & 4 & 1' & 2' &    &    \\
\color{blue}\it 1 & \color{blue}\it 3 & \color{blue}\it 4 &    &    &    &    \\
\hline
1 & 3 & 4 &    &    & 3' & 4' \\
\color{red}\bf 1 & \color{red}\bf 3 & \color{red}\bf 4 & \color{red}\bf 1' &    &    & \color{red}\bf 4' \\
1 &   & 4 & 1' & 2' &    &    \\
1 & 3 &   &    & 2' & 3' &    \\
\end{array}
}

\newcommand{\Mab}{\Mpair{\Mxab}{\Myab}}

\newcommand{\Mxac}
{
\begin{array}{ccccccc}
1 &   &   &    & 2' & 3' & 4' \\
  & 2 &   & 1' &    & 3' & 4' \\
1 &   &   & 1' & 2' &    & 4' \\
  &   & 4 & 1' & 2' & 3' &    \\
1 & 2 & 4 & 1' &    &    &    \\
1 &   & 4 &    & 2' &    &    \\
\color{blue}\it 1 & \color{blue}\it 2 & \color{blue}\it 4 &    &    & \color{blue}\it 3' &    \\
1 & 2 &   &    &    &    & 4' \\
\hline
\color{red}\bf 1 & \color{red}\bf 2 & \color{red}\bf 4 &    &    & \color{red}\bf 3' & \color{red}\bf 4' \\
1 &   & 4 & 1' &    &    & 4' \\
1 & 2 & 4 & 1' & 2' &    &    \\
1 & 2 &   &    & 2' & 3' &    \\
\end{array}
}

\newcommand{\Myac}
{
\begin{array}{ccccccc}
  &   &   & 1' & 2' & 3' & 4' \\
1 & 2 &   &    &    & 3' & 4' \\
1 &   &   &    & 2' &    & 4' \\
1 &   & 4 &    & 2' & 3' &    \\
1 & 2 &   & 1' &    &    & 4' \\
  & 2 & 4 & 1' &    & 3' &    \\
1 &   & 4 & 1' & 2' &    &    \\
\color{blue}\it 1 & \color{blue}\it 2 & \color{blue}\it 4 &    &    &    &    \\
\hline
\color{red}\bf 1 & \color{red}\bf 2 & \color{red}\bf 4 &    &    & \color{red}\bf 3' & \color{red}\bf 4' \\
1 &   & 4 & 1' &    &    & 4' \\
1 & 2 & 4 & 1' & 2' &    &    \\
1 & 2 &   &    & 2' & 3' &    \\
\end{array}
}

\newcommand{\Mac}{\Mpair{\Mxac}{\Myac}}

\newcommand{\MxaB}
{
\begin{array}{ccccccc}
1 &   &   &   &    & 3' & 4' \\
  & 2 &   &   & 1' & 3' & 4' \\
1 &   & 3 &   & 1' &    & 4' \\
1 &   &   & 4 & 1' & 3' &    \\
  & 2 & 3 & 4 & 1' &    &    \\
\color{blue}\it 1 &   & \color{blue}\it 3 & \color{blue}\it 4 &    &    &    \\
1 & 2 &   & 4 &    & 3' &    \\
1 & 2 & 3 &   &    &    & 4' \\
\hline
  & 2 & 3 & 4 &    & 3' & 4' \\
\color{red}\bf 1 &   & \color{red}\bf 3 & \color{red}\bf 4 & \color{red}\bf 1' &    & \color{red}\bf 4' \\
1 & 2 &   & 4 & 1' &    &    \\
1 & 2 & 3 &   &    & 3' &    \\
\end{array}
}

\newcommand{\MyaB}
{
\begin{array}{ccccccc}
1 &   &   &   & 1' & 3' & 4' \\
1 & 2 &   &   &    & 3' & 4' \\
1 &   & 3 &   &    &    & 4' \\
1 &   &   & 4 &    & 3' &    \\
  & 2 & 3 &   & 1' &    & 4' \\
  & 2 &   & 4 & 1' & 3' &    \\
\color{blue}\it 1 &   & \color{blue}\it 3 & \color{blue}\it 4 & \color{blue}\it 1' &    &    \\
1 & 2 & 3 & 4 &    &    &    \\
\hline
  & 2 & 3 & 4 &    & 3' & 4' \\
\color{red}\bf 1 &   & \color{red}\bf 3 & \color{red}\bf 4 & \color{red}\bf 1' &    & \color{red}\bf 4' \\
1 & 2 &   & 4 & 1' &    &    \\
1 & 2 & 3 &   &    & 3' &    \\
\end{array}
}

\newcommand{\MaB}{\Mpair{\MxaB}{\MyaB}}

\newcommand{\MxaC}
{
\begin{array}{ccccccc}
1 &   &   &   &    & 2' & 4' \\
1 & 2 &   &   & 1' &    & 4' \\
  &   & 3 &   & 1' & 2' & 4' \\
1 &   &   & 4 & 1' & 2' &    \\
  & 2 & 3 & 4 & 1' &    &    \\
1 &   & 3 & 4 &    & 2' &    \\
\color{blue}\it 1 & \color{blue}\it 2 &   & \color{blue}\it 4 &    &    &    \\
1 & 2 & 3 &   &    &    & 4' \\
\hline
\color{red}\bf 1 & \color{red}\bf 2 & \color{red}\bf 3 & \color{red}\bf 4 &    &    & \color{red}\bf 4' \\
1 &   & 3 & 4 & 1' &    & 4' \\
1 & 2 &   & 4 & 1' & 2' &    \\
1 & 2 & 3 &   &    & 2' &    \\
\end{array}
}

\newcommand{\MyaC}
{
\begin{array}{ccccccc}
1 &   &   &   & 1' & 2' & 4' \\
1 & 2 &   &   &    &    & 4' \\
1 &   & 3 &   &    & 2' & 4' \\
1 &   &   & 4 &    & 2' &    \\
  & 2 & 3 &   & 1' &    & 4' \\
1 & 2 &   & 4 & 1' &    &    \\
  &   & 3 & 4 & 1' & 2' &    \\
\color{blue}\it 1 & \color{blue}\it 2 & \color{blue}\it 3 & \color{blue}\it 4 &    &    &    \\
\hline
\color{red}\bf 1 & \color{red}\bf 2 & \color{red}\bf 3 & \color{red}\bf 4 &    &    & \color{red}\bf 4' \\
1 &   & 3 & 4 & 1' &    & 4' \\
1 & 2 &   & 4 & 1' & 2' &    \\
1 & 2 & 3 &   &    & 2' &    \\
\end{array}
}

\newcommand{\MaC}{\Mpair{\MxaC}{\MyaC}}

\newcommand{\MxaD}
{
\begin{array}{ccccccc}
1 &   &   &   &    & 2' & 3' \\
1 & 2 &   &   & 1' &    & 3' \\
1 &   & 3 &   & 1' & 2' &    \\
  &   &   & 4 & 1' & 2' & 3' \\
  & 2 & 3 & 4 & 1' &    &    \\
1 &   & 3 & 4 &    & 2' &    \\
1 & 2 &   & 4 &    &    & 3' \\
\color{blue}\it 1 & \color{blue}\it 2 & \color{blue}\it 3 &   &    &    &    \\
\hline
\color{red}\bf 1 & \color{red}\bf 2 & \color{red}\bf 3 & \color{red}\bf 4 &    &    & \color{red}\bf 3' \\
1 &   & 3 & 4 & 1' &    &    \\
1 & 2 &   & 4 & 1' & 2' &    \\
1 & 2 & 3 &   &    & 2' & 3' \\
\end{array}
}

\newcommand{\MyaD}
{
\begin{array}{ccccccc}
1 &   &   &   & 1' & 2' & 3' \\
1 & 2 &   &   &    &    & 3' \\
1 &   & 3 &   &    & 2' &    \\
1 &   &   & 4 &    & 2' & 3' \\
1 & 2 & 3 &   & 1' &    &    \\
  & 2 &   & 4 & 1' &    & 3' \\
  &   & 3 & 4 & 1' & 2' &    \\
\color{blue}\it 1 & \color{blue}\it 2 & \color{blue}\it 3 & \color{blue}\it 4 &    &    &    \\
\hline
\color{red}\bf 1 & \color{red}\bf 2 & \color{red}\bf 3 & \color{red}\bf 4 &    &    & \color{red}\bf 3' \\
1 &   & 3 & 4 & 1' &    &    \\
1 & 2 &   & 4 & 1' & 2' &    \\
1 & 2 & 3 &   &    & 2' & 3' \\
\end{array}
}

\newcommand{\MaD}{\Mpair{\MxaD}{\MyaD}}

\newcommand{\MxAB}
{
\begin{array}{ccccccc}
1 &   &   &   & 1' & 3' & 4' \\
  & 2 &   &   & 1' & 3' & 4' \\
  &   & 3 &   & 1' &    & 4' \\
  &   &   & 4 & 1' & 3' &    \\
  & 2 & 3 & 4 & 1' &    &    \\
\color{blue}\it 1 &   & \color{blue}\it 3 & \color{blue}\it 4 & \color{blue}\it 1' &    &    \\
1 & 2 &   & 4 &    & 3' &    \\
1 & 2 & 3 &   &    &    & 4' \\
\hline
  & 2 & 3 & 4 &    & 3' & 4' \\
\color{red}\bf 1 &   & \color{red}\bf 3 & \color{red}\bf 4 & \color{red}\bf 1' &    & \color{red}\bf 4' \\
1 & 2 &   & 4 & 1' &    &    \\
1 & 2 & 3 &   & 1' & 3' &    \\
\end{array}
}

\newcommand{\MyAB}
{
\begin{array}{ccccccc}
  &   &   &   & 1' & 3' & 4' \\
1 & 2 &   &   &    & 3' & 4' \\
\color{blue}\it 1 &   & \color{blue}\it 3 &   & \color{blue}\it 1' &    & \color{blue}\it 4' \\
1 &   &   & 4 & 1' & 3' &    \\
  & 2 & 3 &   & 1' &    & 4' \\
  & 2 &   & 4 & 1' & 3' &    \\
  &   & 3 & 4 & 1' &    &    \\
1 & 2 & 3 & 4 &    &    &    \\
\hline
  & 2 & 3 & 4 &    & 3' & 4' \\
\color{red}\bf 1 &   & \color{red}\bf 3 & \color{red}\bf 4 & \color{red}\bf 1' &    & \color{red}\bf 4' \\
1 & 2 &   & 4 & 1' &    &    \\
1 & 2 & 3 &   & 1' & 3' &    \\
\end{array}
}

\newcommand{\MAB}{\Mpair{\MxAB}{\MyAB}}

\newcommand{\MxAC}
{
\begin{array}{ccccccc}
1 &   &   &   & 1' & 2' & 4' \\
  & 2 &   &   & 1' &    & 4' \\
  &   & 3 &   & 1' & 2' & 4' \\
  &   &   & 4 & 1' & 2' &    \\
  & \color{blue}\it 2 & \color{blue}\it 3 & \color{blue}\it 4 & \color{blue}\it 1' &    &    \\
1 &   & 3 & 4 &    & 2' &    \\
1 & 2 &   & 4 & 1' &    &    \\
1 & 2 & 3 &   &    &    & 4' \\
\hline
  & \color{red}\bf 2 & \color{red}\bf 3 & \color{red}\bf 4 & \color{red}\bf 1' &    & \color{red}\bf 4' \\
1 &   & 3 & 4 & 1' &    & 4' \\
1 & 2 &   & 4 & 1' & 2' &    \\
1 & 2 & 3 &   & 1' & 2' &    \\
\end{array}
}

\newcommand{\MyAC}
{
\begin{array}{ccccccc}
  &   &   &   & 1' & 2' & 4' \\
1 & 2 &   &   & 1' &    & 4' \\
1 &   & 3 &   &    & 2' & 4' \\
1 &   &   & 4 & 1' & 2' &    \\
  & \color{blue}\it 2 & \color{blue}\it 3 &   & \color{blue}\it 1' &    & \color{blue}\it 4' \\
  & 2 &   & 4 & 1' &    &    \\
  &   & 3 & 4 & 1' & 2' &    \\
1 & 2 & 3 & 4 &    &    &    \\
\hline
  & \color{red}\bf 2 & \color{red}\bf 3 & \color{red}\bf 4 & \color{red}\bf 1' &    & \color{red}\bf 4' \\
1 &   & 3 & 4 & 1' &    & 4' \\
1 & 2 &   & 4 & 1' & 2' &    \\
1 & 2 & 3 &   & 1' & 2' &    \\
\end{array}
}

\newcommand{\MAC}{\Mpair{\MxAC}{\MyAC}}

\newcommand{\Mlabel}{\begin{array}{l} (\Sp{\Smon}^4_1)_I, \\ (\Sp{\Smon}^4_2)_I \end{array}}

\begin{table}
\footnotesize
\setlength{\arraycolsep}{0.35pt}
\[
\begin{array}{|c|c|c|c|}
\hline
I & \{1, 1'\} & \{1, 2\} & \{1, 3\} \\
\hline
\Mlabel & \MaA & \Mab & \Mac \\
\hline
\sigma & \mathrm{id} & (3 \; 3') & (2 \; 2') \\
\hline \hline
I & \{1, 2'\} & \{1, 3'\} & \{1, 4'\} \\
\hline
\Mlabel & \MaB & \MaC & \MaD \\
\hline
\sigma & (3 \; 3') & (2 \; 2') & (2 \; 2') \\
\hline \hline
I & \{1', 2'\} & \{1', 3'\} & \\
\cline{1-3}
\Mlabel & \MAB & \MAC & \\
\cline{1-3}
\cline{1-3}
\sigma & (3 \; 3') & (2 \; 2') & \\
\hline
\end{array}
\]
\caption{Set systems $(\Sp{\Smon}^4_1)_I$ and $(\Sp{\Smon}^4_2)_I$. In each set system, the block written in bold and red, if any, is not minimal, because it includes the block written in italics and blue. The non-bold non-red blocks of $(\Sp{\Smon}^4_1)_I$ are mapped onto the non-bold non-red blocks of $(\Sp{\Smon}^4_2)_I$ by the permutation $\sigma$.}
\label{table:minorsM4}
\end{table}

\end{example}

\subsection{Construction for clique functions}

We now extend the Sperner systems $\Sp{\Smon}^m_i$ to systems over sets of cardinality $2m + 1$ and $2m + 2$.
The Boolean functions associated with these systems will be members of the clone $M_c U_\infty$ of monotone constant-preserving $1$-separating functions.
Let us denote $E_m^0 := E_m \cup \{0\}$ and $E_m^{00'} := E_m \cup \{0, 0'\}$.

\begin{definition}
For $m \geq 3$ and $i \in \{1, 2\}$, let $\Sp{\Sclique}^{2m+1}_i$ be the set system over $E_m^0$ and let $\Sp{\Sclique}^{2m+2}_i$ be the set system over $E_m^{00'}$ given by
\begin{align*}
\Sp{\Sclique}^{2m+1}_i &:= \{S \cup \{0\} : S \in \Sp{\Smon}^m_i\}, \\
\Sp{\Sclique}^{2m+2}_i &:= \Sp{\Sclique}^{2m+1}_i \cup \{\{0, 0'\}\}.
\end{align*}
\end{definition}

\begin{remark}
\label{rem:UmiMcUinf}
By definition, for all $n \geq 7$ and $i \in \{1, 2\}$, there is an element of the ground set of $\Sp{\Sclique}^n_i$ -- namely $0$ -- that is contained in every block of $\Sp{\Sclique}^n_i$.
Hence, the Boolean function associated with $\Sp{\Sclique}^n_i$ is a member of the clone $M_c U_\infty$.
\end{remark}

It remains to show that $\Sp{\Sclique}^n_1$ and $\Sp{\Sclique}^n_2$ are nonisomorphic and strongly hypomorphic.

\begin{proposition}
\label{prop:Uminoniso}
The Sperner systems $\Sp{\Sclique}^{2m+1}_1$ and $\Sp{\Sclique}^{2m+1}_2$ are nonisomorphic, and $\Sp{\Sclique}^{2m+2}_1$ and $\Sp{\Sclique}^{2m+2}_2$ are nonisomorphic, for every $m \geq 3$.
\end{proposition}

\begin{proof}
Suppose, on the contrary, that there exists an isomorphism $\sigma$ between $\Sp{\Sclique}^{2m+1}_1$ and $\Sp{\Sclique}^{2m+1}_2$. Counting the number of occurrences of each element in the blocks of the two systems, we see that $\sigma$ must fix $0$. Then the restriction of $\sigma$ to $E_m$ yields an isomorphism between $\Sp{\Smon}^m_1$ and $\Sp{\Smon}^m_2$, a contradiction to Proposition~\ref{prop:Cnoniso}.

A similar argument shows that there does not exist any isomorphism between $\Sp{\Sclique}^{2m+2}_1$ and $\Sp{\Sclique}^{2m+2}_2$.
\end{proof}

\begin{proposition}
\label{prop:U2m+1hypo}
For all $m \geq 3$, the Sperner systems $\Sp{\Sclique}^{2m+1}_1$ and $\Sp{\Sclique}^{2m+1}_2$ are strongly hypomorphic.
\end{proposition}

\begin{proof}
Let $I \in \couples[E_m^0]$. If $0 \notin I$, then for each $i \in \{1, 2\}$, it holds that
\begin{multline*}
(\Sp{\Sclique}^{2m+1}_i)_I =
\{(S \cup \{0\})_I : S \in \Sp{\Smon}^m_i\} = \\
\{S_I \cup \{0\} : S \in \Sp{\Smon}^m_i\} =
\{S \cup \{0\} : S \in (\Sp{\Smon}^m_i)_I\}.
\end{multline*}
Consequently, $(\Sp{\Sclique}^{2m+1}_i)^*_I = \{S \cup \{0\} : S \in (\Sp{\Smon}^m_i)^*_I\}$.
By Proposition~\ref{prop:Mhypo}, there exists an isomorphism $\sigma$ of $(\Sp{\Smon}^m_1)^*_I$ to $(\Sp{\Smon}^m_2)^*_I$. The extension of $\sigma$ to $E_m^0$, keeping $0$ as a fixed point, is clearly an isomorphism between $(\Sp{\Sclique}^{2m+1}_1)^*_I$ and $(\Sp{\Sclique}^{2m+1}_2)^*_I$.

Let us then consider the case when $0 \in I$. Then $I = \{0, x\}$ for some $x \in E_m$, and $(\Sp{\Sclique}^{2m+1}_i)_I = \{S \cup \{x\} : S \in \Sp{\Smon}^m_i\}$.
We will specify an element $r \in \nset{m}$ and sets $S_1, S_2 \subseteq \nset{m}$ such that $\{r + 1, (r + 1)'\} \cap \{x\} = \emptyset$ and $\Sbuildall^m_{S_i} \cup \{x\} \subseteq \Sbuildmon^m_r \cup \{x\}$ for $i \in \{1, 2\}$.
This means that we can ignore $\Sbuildmon^m_r \cup \{x\}$ when we consider the minimal blocks of $(\Sp{\Sclique}^{2m+1}_1)_I$ and $(\Sp{\Sclique}^{2m+1}_2)_I$.
The transposition or $r + 1$ and $(r + 1)'$ will be an isomorphism between $(\Sp{\Sclique}^{2m+1}_1)^*_I$ and $(\Sp{\Sclique}^{2m+1}_2)^*_I$.

If $x = p$ for some $p \in \nset{m}$, then choose $r := p$ and, according to the parity of $m$, choose $S_1$ and $S_2$ among the sets $\nset{m}$ and $\nset{m} \setminus \{p + 2\}$.

If $x = p'$ for some $p \in \nset{m}$, then choose $r := p$ and, according to the parity of $m$, choose $S_1$ and $S_2$ among the sets $\nset{m} \setminus \{p\}$ and $\nset{m} \setminus \{p, p + 2\}$.
\end{proof}

\begin{proposition}
\label{prop:U2m+2hypo}
For all $m \geq 3$, the Sperner systems $\Sp{\Sclique}^{2m+2}_1$ and $\Sp{\Sclique}^{2m+2}_2$ are strongly hypomorphic.
\end{proposition}

\begin{proof}
Let $I \in \couples[E_m^{0, 0'}]$.
If $0' \notin I$, then for each $i \in \{1, 2\}$, it holds that $(\Sp{\Sclique}^{2m+2}_i)_I = (\Sp{\Sclique}^{2m+1}_i)_I \cup \{0, 0'\}$; hence $(\Sp{\Sclique}^{2m+2}_i)^*_I = (\Sp{\Sclique}^{2m+1}_i)^*_I \cup \{0, 0'\}$.
By Proposition~\ref{prop:U2m+1hypo}, $(\Sp{\Sclique}^{2m+2}_i)^*_I$ and $(\Sp{\Sclique}^{2m+1}_i)^*_I$ are isomorphic, an isomorphism being given by the transposition $\tau_s$ for some $s \in \nset{m}$.
Consequently, $\tau_s$ is an isomorphism between $(\Sp{\Sclique}^{2m+2}_1)_I$ and $(\Sp{\Sclique}^{2m+2}_2)_I$.

Let us then consider the case when $0' \in I$.
If $I = \{0, 0'\}$, then $(\Sp{\Sclique}^{2m+2}_1)^*_I = \{\{0\}\} = (\Sp{\Sclique}^{2m+2}_2)^*_I$.
If $I = \{0', x\}$ for some $x \in E_m$, then
$(\Sp{\Sclique}^{2m+2}_i)^*_I = \{\{0, x\}\} \cup \{S \cup \{0\} : S \in \Sp{\Smon}^m_i, x \notin S\}$.
Furthermore, $x \in \{r, r'\}$ for some $r \in \nset{m}$ and $\Sbuildmon^m_{r+1} \cup \{0\} \notin (\Sp{\Sclique}^{2m+2}_i)^*_I$.
Therefore the transposition of $r + 2$ and $(r + 2)'$ is an isomorphism between $(\Sp{\Sclique}^{2m+2}_1)^*_I$ and $(\Sp{\Sclique}^{2m+2}_2)^*_I$.
\end{proof}

\subsection{Construction for self-dual functions}

Our next construction will result in Sperner systems whose associated Boolean functions are members of the clone $SM$ of self-dual monotone functions.
We are going to use again the systems $\Sp{\Sbuildall}^m_1$ and $\Sp{\Sbuildall}^m_2$ as building blocks, and we are going to add to these systems some new blocks, which will destroy the isomorphism but maintain hypomorphism as well as $m$-homogeneity.
We need to develop some notation and tools in order to be able to efficiently specify which $m$-element subsets of $E_m$ we are going to choose as the blocks of our set systems, and to be able to show that the resulting systems have the desired properties.

\begin{definition}
Let $X, Y \subseteq \nset{m}$ with $X \cap Y = \emptyset$. We define $\XYset{X}{Y}$ by
\begin{align*}
  \XYset{X}{Y}  := \{ J \subseteq E_m :{} & \forall x \in X \colon \card{\{x,x'\} \cap J} = 2, \\
                                        & \forall y \in Y \colon \card{\{y,y'\} \cap J} = 0, \\
        & \forall z \in \nset{m} \setminus (X \cup Y) \colon \card{\{z,z'\} \cap J} = 1 \}.
\end{align*}
For notational simplicity, we will write $\XYrot{X}{Y}$ for $\rot{\XYset{X}{Y}}$.
\end{definition}

Recall that $\rot{\Sp{\Sgena}} = \bigcup_{i \in \nset{m}} \rho^i(\Sp{\Sgena})$, and
observe that
\[
\XYrot{X}{Y} = \bigcup_{q \in \nset{m}} \XYset{X + q}{Y + q}.
\]

\begin{example}
Let $m = 5$. Then
\begin{align*}
\XYset{\{1\}}{\{2,3\}} = \{ & 
    \{1,1',4,5\}, \{1,1',4,5'\},\{1,1',4',5\}, \{1,1',4',5'\} \}, \\
\XYrot{\{1,2\}}{\{3,4\}} = \{ & 
   \{1,1',2,2',5\}, \{1,1',2,2',5'\}, \{2,2',3,3',1\}, \{2,2',3,3',1'\}, \\
&  \{3,3',4,4',2\}, \{3,3',4,4',2'\}, \{4,4',5,5',3\}, \{4,4',5,5',3'\}, \\
&  \{5,5',1,1',4\}, \{5,5',1,1',4'\} \}.
\end{align*}
\end{example}

\begin{remark}
\label{rem:emptyemptyG}
$\XYset{\emptyset}{\emptyset} = \XYrot{\emptyset}{\emptyset} = \Sp{\Sbuildall}^m_1 \cup \Sp{\Sbuildall}^m_2$.
\end{remark}

\begin{lemma}
\label{lem:XYXqYq}
For all $X, Y \subseteq \nset{m}$ with $X \cap Y = \emptyset$ and $q \in [m]$ we have $\XYrot{X}{Y} = \XYrot{X+q}{Y+q}$.
\end{lemma}

\begin{proof}
Follows immediately from the definition.
\end{proof}

\begin{lemma}
\label{lem:XYdisjoint}
Let $X_1, Y_1, X_2, Y_2 \subseteq \nset{m}$ be sets such that $X_i \cap Y_i = \emptyset$ for $i \in \{1, 2\}$.
Then either $\XYrot{X_1}{Y_1} = \XYrot{X_2}{Y_2}$ or $\XYrot{X_1}{Y_1} \cap \XYrot{X_2}{Y_2} = \emptyset$.
\end{lemma}

\begin{proof}
Assume $\XYrot{X_1}{Y_1} \cap \XYrot{X_2}{Y_2} \neq \emptyset$, and let $J \in \XYrot{X_1}{Y_1} \cap \XYrot{X_2}{Y_2}$.
Then there are some $q_1, q_2 \in \nset{m}$ with $J \in (\XYset{X_1+q_1}{Y_1+q_1}) \cap (\XYset{X_2+q_2}{Y_2+q_2})$.
We have for $i \in \{1, 2\}$
\begin{gather*}
    \forall x \in X_i + q_i \colon \card{\{x, x'\} \cap J} = 2, \\
    \forall z \in \nset{m} \setminus (X_i + q_i) \colon \card{\{z, z'\} \cap J} \leq 1.
\end{gather*}
This implies $X_1 +q_1 = X_2 + q_2$, that is, $X_2 = X_1 + (q_1-q_2)$.
Similarly, we get $Y_1 +q_1 = Y_2 + q_2$, that is, $Y_2 = Y_1 + (q_1-q_2)$.
  
Thus $\XYrot{X_2}{Y_2} = \XYrot{X_1 + (q_1-q_2)}{Y_1 + (q_1-q_2)} = \XYrot{X_1}{Y_1}$, where the last equality holds by Lemma~\ref{lem:XYXqYq}.
\end{proof}

\begin{lemma}
\label{lem:XYcard}
For all $X, Y \subseteq \nset{m}$ such that $X \cap Y = \emptyset$, the set system $\XYrot{X}{Y}$ is $k$-homogeneous for $k = m + \card{X} - \card{Y}$.
\end{lemma}

\begin{proof}
For every $J \in \XYrot{X}{Y}$, it holds that
\[
\card{J} = 2 \card{X} + \card{\nset{m} \setminus (X \cup Y)}
= 2 \card{X} + m - (\card{X} + \card{Y})
= m + \card{X} - \card{Y}.
\qedhere
\]
\end{proof}

Let $\QsetsE := \{ \XYrot{X}{Y} : \card{X} = \card{Y}, X \cap Y = \emptyset \}$ 
and $\Qsets  := \QsetsE \setminus \{ \XYrot{\emptyset}{\emptyset} \}$.

\begin{lemma}
\label{lem:Qsetsunion}
Let $J \subseteq E_m$.
Then $\card{J} = m$ if and only if there is some $\XYrot{X}{Y} \in \QsetsE$ with $J \in \XYrot{X}{Y}$.
\end{lemma}

\begin{proof}
Assume first that $J \in \XYrot{X}{Y} \in \QsetsE$. Then $\card{X} = \card{Y}$, and $X \cap Y = \emptyset$.
By Lemma~\ref{lem:XYcard}, we have $\card{J} = m + \card{X} - \card{Y} = m$.

For the converse implication, assume that $\card{J} = m$.
Set $X := \{ x \in \nset{m} : x, x' \in J \}$ and $Y := \{ y \in \nset{m} : y, y' \notin J \}$. Then $X \cap Y = \emptyset$ and $J \in \XYrot{X}{Y}$.
Suppose, on the contrary, that $\card{X} \neq \card{Y}$.
Then Lemma~\ref{lem:XYcard} yields $\card{J} = m + \card{X} - \card{Y} \neq m$, which contradicts our assumption that $\card{J} = m$.
Thus $\card{X} = \card{Y}$, and we conclude that $\XYrot{X}{Y} \in \QsetsE$.
\end{proof}

\begin{lemma}
\label{lem:XYcomplements}
For all $X, Y \subseteq \nset{m}$ and $J \subseteq E_m$, it holds that $J \in \XYset{X}{Y}$ if and only if $\complementX{J} \in \XYset{Y}{X}$.
Furthermore, $J \in \XYrot{X}{Y}$ if and only if $\complementX{J} \in \XYrot{Y}{X}$.
\end{lemma}

\begin{proof}
Follows immediately from the definition.
\end{proof}

\begin{lemma}
\label{lem:gcd}
Let $X \subseteq \nset{m}$ and $q \in \mathbb{Z}$. Let $d = \gcd(q, m)$.
\begin{compactenum}[\rm (i)]
\item\label{item:gcd:1}
$X = X + q$ if and only if $X = \bigcup_{x \in X} (x + q \mathbb{Z} / m)$.

\item\label{item:gcd:2}
$q \mathbb{Z} / m = d \mathbb{Z} / m$.

\item\label{item:gcd:3}
Consequently, $X = X + q$ if and only if $X = X + d$.
\end{compactenum}
\end{lemma}

\begin{proof}
\eqref{item:gcd:1}
If $X = \bigcup_{x \in X} (x + q \mathbb{Z} /m)$, then it clearly holds that $X = X + q$.
If $X = X + q$, then an easy induction shows that $x \in X$ implies $x + nq \in X$ for all $n \in \mathbb{Z}$. Consequently, $X = \bigcup_{x \in X} (x + q \mathbb{Z} / m)$.

\eqref{item:gcd:2}
Since $q = cd$ for some $c \in \mathbb{Z}$, it holds that $nq = ncd$ for any $n \in \mathbb{Z}$. Hence $q \mathbb{Z} / m \subseteq d \mathbb{Z} / m$.
It is well known that there exist numbers $\alpha, \beta \in \mathbb{Z}$ such that $\alpha q + \beta m = d$.
Thus $nd = n \alpha q + n \beta m \equiv n \alpha q \pmod{m}$ for any $n \in \mathbb{Z}$. Hence $d \mathbb{Z} / m \subseteq q \mathbb{Z} / m$.

\eqref{item:gcd:3}
By part~\eqref{item:gcd:1}, $X = X + q$ if and only if $X = \bigcup_{x \in X} (x + q \mathbb{Z} / m)$. By part~\eqref{item:gcd:2}, this condition is equivalent to $X = \bigcup_{x \in X} (x + d \mathbb{Z} / m)$, which, again by part~\eqref{item:gcd:1}, is equivalent to $X = X + d$.
\end{proof}

\begin{lemma}
\label{lem:noComplement}
Assume that $m$ is odd.
Let $\XYrot{X}{Y} \in \Qsets$ and $J \in \XYrot{X}{Y}$.
Then $\complementX{J} \notin \XYrot{X}{Y}$.
\end{lemma}

\begin{proof}
Without loss of generality, we can assume that $J \in \XYset{X}{Y}$.
Suppose, on the contrary, that $\complementX{J} \in \XYrot{X}{Y}$. Then there is some $q \in \nset{m}$ with $\complementX{J} \in \XYset{X+q}{Y+q}$.

Let $j \in X$. Then $j, j' \in J$ and $j, j' \notin \complementX{J}$, and thus $j \in Y + q$.
Similarly, let $j \in Y$. Then $j, j' \notin J$ and $j, j' \in \complementX{J}$, and thus $j \in X + q$.
From this it follows that $X = Y + q$ and $Y = X + q$.
Consequently, $X = X + 2q$.

Because $m$ is odd, we have $\gcd(m,2q) = \gcd(m,q)$.
Lemma~\ref{lem:gcd} implies that $X = X + q$.
Thus $X = Y$.
Since $X \neq \emptyset$, we have $X \cap Y = X \neq \emptyset$, which is a contradiction.
We conclude that our assumption $\complementX{J} \in \XYrot{X}{Y}$ was wrong.
\end{proof}

The previous lemmas can be summarized as follows.

\begin{proposition}
\label{prop:qsetsPartition}
The family $\QsetsE$ is a partition of all $m$-element subsets of $E_m$,
and for $\XYrot{X}{Y} \in \Qsets$ we have 
$\XYrot{Y}{X} = \complementX{\left( \XYrot{X}{Y} \right)}$.
Furthermore, if $m$ is odd, then $\XYrot{X}{Y} \neq \XYrot{Y}{X}$, i.e.,
a subset $S \subseteq E_m$ and its complement $\complementX{S}$ are in distinct parts, unless $S \in \XYrot{\emptyset}{\emptyset}$.
\end{proposition}

\begin{proof}
The members of the family $\QsetsE$ are pairwise disjoint by Lemma~\ref{lem:XYdisjoint} and their union is the set of all $m$-element subsets of $E_m$ by Lemma~\ref{lem:Qsetsunion}. By Lemma~\ref{lem:XYcomplements}, $\complementX{\XYrot{X}{Y}} = \XYrot{Y}{X}$. If $m$ is odd, then $\XYrot{X}{Y} \neq \XYrot{Y}{X}$ by Lemma~\ref{lem:noComplement}.
\end{proof}

\begin{lemma}
\label{lem:smallerX}
\label{lem:biggerY}
\begin{inparaenum}[\rm (i)]
\item
Let $X_1, X_2, Y \subseteq \nset{m}$ with $X_2 \subseteq X_1$ and $X_1 \cap Y = \emptyset$.
Let $J_1 \in \XYrot{X_1}{Y}$. Then there is some $J_2 \in \XYrot{X_2}{Y}$ with $J_2 \subseteq J_1$.

\item
Let $X, Y_1, Y_2 \subseteq \nset{m}$ with $Y_1 \subseteq Y_2$ and $X \cap Y_2 = \emptyset$.
Let $J_1 \in \XYrot{X}{Y_1}$. Then there is some $J_2 \in \XYrot{X}{Y_2}$ with $J_2 \subseteq J_1$.
\end{inparaenum}
\end{lemma}

\begin{proof}
(i)
Without loss of generality, we may assume that $J_1 \in \XYset{X_1}{Y}$.
Let $J_2 := J_1 \setminus (X_1 \setminus X_2)'$.
It is clear from the definition that $J_2 \subseteq J_1$.
We now show that $J_2 \in \XYrot{X_2}{Y}$, or, more specifically, $J_2 \in \XYset{X_2}{Y}$.

\begin{itemize}
\item Since $X_1 \cup X_1' \subseteq J_1$, we have $X_2 \cup X_2' \subseteq J_1 \setminus (X_1 \setminus X_2)' = J_2$.
\item Since $(Y \cup Y') \cap J_1 = \emptyset$, we have $(Y \cup Y') \cap J_2 = (Y \cup Y') \cap (J_1 \setminus (X_1 \setminus X_2)') \subseteq (Y \cup Y') \cap J_1 = \emptyset$.
\item For all $z \in \nset{m} \setminus (X_2 \cup Y) = (X_1 \setminus X_2) \cup (\nset{m} \setminus (X_1 \cup Y))$, exactly one of $z$ and $z'$ is an element of $J_2$, because
  \begin{itemize}
    \item if $z \in X_1 \setminus X_2$, then $z \in J_2$ and $z' \notin J_2$;
    \item if $z \in \nset{m} \setminus (X_1 \cup Y)$, then either $z \in J_1$ or $z' \in J_1$ but not both. Since $J_1$ and $J_2$ coincide on this part, this property holds also for $J_2$.
  \end{itemize}
\end{itemize}

(ii)
Without loss of generality, we may assume that $J_1 \in \XYset{X}{Y_1}$. Let $J_2 := J_1 \setminus (Y_2 \cup Y_2')$. It is clear that $J_2 \subseteq J_1$ and $J_2 \in \XYset{X}{Y_2} \subseteq \XYrot{X}{Y_2}$.
\end{proof}

\begin{lemma}
\label{lem:XYtransposition}
For all $j \in \nset{m}$ the set $\XYrot{X}{Y}$ is closed under the transposition $\tau_j$.
\end{lemma}

\begin{proof}
Let $J \in \XYrot{X}{Y}$. Without loss of generality, we may assume that $J \in \XYset{X}{Y}$.
Let $j \in \nset{m}$.
If $j \in X$, then we have $j, j' \in J$ and thus $\tau_j(J) = J$.
If $j \in Y$, then we have $j, j' \notin J$ and thus $\tau_j(J) = J$.
If $j \in \nset{m} \setminus (X \cup Y)$, then either $j$ or $j'$ is in $J$ but not both. The same is true for $\tau_j(J)$, and we have that $\tau_j(J) \in \XYset{X}{Y}$.
\end{proof}

We define signatures for pairs $(X,Y)$ of disjoint subsets $X$ and $Y$ of $\nset{m}$, as well as for set systems $\XYrot{X}{Y}$.
The signature will serve as a selection criterion for choosing the blocks in the Sperner systems that we are going to construct.
Denote by $\#_a(w)$ the number of occurrences of a letter $a$ in a word $w$, i.e., for $w = w_1 \dots w_n$, we have $\#_a(w) = \card{\{i \in \nset{n} : w_i = a\}}$

\begin{definition}
Let $X, Y \subseteq \nset{m}$ with $X \cap Y = \emptyset$, and let $Z := \nset{m} \setminus (X \cup Y)$.

The \emph{full signature} $\fullsig(X,Y)$ of $(X,Y)$ is the string $d_1 \dots d_m$ over the alphabet $\{x, y, z\}$ defined by
\[
d_i =
\begin{cases}
x, & \text{if $i \in X$,} \\
y, & \text{if $i \in Y$,} \\
z, & \text{if $i \in Z$.}
\end{cases}
\]

Define the map $\psi \colon \{x, y, z\}^* \setminus \{z\}^* \to \{x, y, \alpha, \beta\}^*$ as follows: $\psi$ maps any string $d_1 \dots d_n$ that does not comprise entirely of $z$'s to the unique string $d'_1 \dots d'_n$ satisfying the following conditions (we do addition modulo $n$, so that $d_0 = d_n$):
\begin{itemize}
\item if $d_i \in \{x, y\}$, then $d'_i = d_i$;
\item if $d_i = z$, then $d'_i \in \{\alpha, \beta\}$ and $d'_i \neq d'_{i-1}$;
\item if $d_i = z$ and $d_{i-1} = x$, then $d'_i = \beta$;
\item if $d_i = z$ and $d_{i-1} = y$, then $d'_i = \alpha$.
\end{itemize}
Define the map $\varphi \colon \{x, y, \alpha, \beta\}^* \to \mathcal{P}(\{\alpha, \beta\})$ as follows:
\[
\varphi(d_1 \dots d_n) = \{e \in \{\alpha, \beta\} : \text{$\#_e(d_1 \dots d_n)$ is odd}\}.
\]
In order to simplify notation, we will write the possible values of $\varphi$ as $\emptyset$, $\alpha$, $\beta$, and $\alpha\beta$ instead of $\emptyset$, $\{\alpha\}$, $\{\beta\}$, and $\{\alpha, \beta\}$.

For $(X,Y) \neq (\emptyset, \emptyset)$, we define the \emph{reduced signature} $\reducedsig(X,Y)$ of $(X,Y)$ as $\varphi(\psi(\fullsig(X,Y)))$.

The \emph{signature} $\reducedsig \XYrot{X}{Y}$ of a set system $\XYrot{X}{Y}$ is defined as $\reducedsig(X,Y)$. In view of Lemmas~\ref{lem:XYXqYq} and~\ref{lem:XYdisjoint}, $\reducedsig \XYrot{X}{Y}$ is well defined, because if $\XYrot{X}{Y} = \XYrot{X'}{Y'}$, then $X' = X + q$ and $Y' = Y + q$ for some $q \in \nset{m}$, and it is clear that $\reducedsig(X,Y) = \reducedsig(X + q, Y + q)$.
\end{definition}

\begin{example}
Let $m = 9$, $X = \{2, 5, 9\}$, and $Y = \{3, 8\}$. Then $Z := \nset{m} \setminus (X \cup Y) = \{1, 4, 6, 7\}$.
Then $\fullsig(X,Y) = zxyzxzzyx$ and $\psi(\fullsig(X,Y)) = \beta xy \alpha x \beta \alpha yx$.
We have $\#_\alpha(\psi(\fullsig(X,Y))) = 2$, $\#_\beta(\psi(\fullsig(X,Y))) = 2$, that is,
both $\alpha$ and $\beta$ occur in $\psi(\fullsig(X,Y))$ an even number of times, so $\reducedsig(X,Y) = \emptyset$.

Let then $m = 9$, $X = \{2, 5\}$, and $Y = \{3, 8\}$. Then $Z := \nset{m} \setminus (X \cup Y) = \{1, 4, 6, 7, 9\}$.
Then $\fullsig(X,Y) = zxyzxzzyz$ and $\psi(\fullsig(X,Y)) = \beta xy \alpha x \beta \alpha y \alpha$.
We have $\#_\alpha(\psi(\fullsig(X,Y))) = 3$, $\#_\beta(\psi(\fullsig(X,Y))) = 2$, so $\reducedsig(X,Y) = \alpha$.
\end{example}

\begin{remark}
Let $X, Y \subseteq \nset{m}$ with $X \cap Y = \emptyset$, and let $Z := \nset{m} \setminus (X \cup Y)$.
If $\card{Z}$ is odd, then $\reducedsig(X,Y) \in \{\alpha, \beta\}$.
If $\card{Z}$ is even, then $\reducedsig(X,Y) \in \{\emptyset, \alpha\beta\}$.
\end{remark}

\begin{lemma}
\label{lem:XYYXalphabeta}
Assume that $m$ is odd and $X, Y \subseteq \nset{m}$ such that $\card{X} = \card{Y} \geq 1$ and $X \cap Y = \emptyset$.
Then $\{\reducedsig(X,Y), \reducedsig(Y,X)\} = \{\alpha, \beta\}$.
\end{lemma}

\begin{proof}
Let $Z := \nset{m} \setminus (X \cup Y)$.
Since $m$ is odd and $\card{X} = \card{Y}$ we see that $\card{Z}$ is odd.
Thus $\reducedsig(X,Y), \reducedsig(Y,X) \in \{\alpha, \beta\}$.

The roles of $X$ and $Y$ are interchanged between $(X,Y)$ and $(Y,X)$, so it is clear from the definition that the full signature $\fullsig(Y,X)$ is obtained from $\fullsig(X,Y)$ by changing every occurrence of $x$ into $y$ and changing every occurrence of $y$ into $x$, keeping the $z$'s unchanged.
It is easy to see that $\psi(\fullsig(Y,X))$ can be obtained from $\psi(\fullsig(X,Y))$ by changing every $\alpha$ into $\beta$ and changing every $\beta$ into $\alpha$.
This means that
\begin{align*}
\#_\alpha(\psi(\fullsig(X,Y))) &= \#_\beta(\psi(\fullsig(Y,X))), \\
\#_\beta(\psi(\fullsig(X,Y))) &= \#_\alpha(\psi(\fullsig(Y,X))).
\end{align*}
Therefore $\reducedsig(X,Y) \neq \reducedsig(Y,X)$, which implies that $\{\reducedsig(X,Y), \reducedsig(Y,X)\} = \{\alpha, \beta\}$.
\end{proof}

\begin{lemma}
\label{lem:xyzsubstrings}
Let $w$ be a string over the alphabet $\{x, y, z\}$, and assume that $\#_x(w) > \#_y(w) \geq 1$. Then $w$, considered as a circular string, contains one of the following strings as a substring:
$yxx \delta$,
$x z^n \delta$,
$y z^n x \delta$,
for some $n \geq 1$ and $\delta \in \{x, y\}$.
\end{lemma}

\begin{proof}
If the letter $z$ does not occur in $w$, then it follows from the assumption $\#_x(w) > \#_y(w) \geq 1$ that $w$ must have two consecutive $x$'s; hence $yxx \delta$ is a substring of $w$ for some $\delta \in \{x, y\}$. We may then assume that there is at least one occurrence of the letter $z$ in $w$. If there is an occurrence of $z$ preceded by $x$, then $w$ has a substring of the form $xz^n \delta$ for some $n \geq 1$ and $\delta \in \{x, y\}$. Assume then that no occurrence of $z$ is preceded by $x$.
If there is an occurrence of $z$ followed by $x$, then $w$ has a substring of the form $y z^n x \delta$ for some $n \geq 1$ and $\delta \in \{x, y, z\}$; note that $\delta$ cannot, however, be equal to $z$, because this would contradict our assumption that no occurrence of $z$ is preceded by $x$; thus, we are done in this case.

Finally, assume that no occurrence of $z$ is followed by $x$. It follows from our assumptions that no occurrence of $x$ is preceded or followed by $z$; that is, every maximal run of $x$'s is preceded and followed by $y$. If there is a run of at least two $x$'s, then $yxx \delta$ is a substring of $w$ for some $\delta \in \{x, y\}$. Otherwise, every occurrence of $x$ is preceded and followed by $y$. This implies that $\#_x(w) < \#_y(w)$, which contradicts our assumption that $\#_x(w) > \#_y(w)$, and we conclude that this last case is not possible.
\end{proof}

\begin{lemma}
\label{lemma:partysig}
Assume that $X, Y \subseteq \nset{m}$, $2 \leq \card{X} = \card{Y} + 1$, and $X \cap Y = \emptyset$.
Then there exist $X_1, Y_1, X_2, Y_2$ with $X_i \subseteq X$, $Y \subseteq Y_i$, and $\card{X_i} = \card{Y_i}$ for $i \in \{1, 2\}$ such that $\reducedsig(X_1,Y_1) = \alpha$ and $\reducedsig(X_2,Y_2) = \beta$.
\end{lemma}

\begin{proof}
Let $Z := \nset{m} \setminus (X \cup Y)$.
Note that for $i \in \{1, 2\}$, the conditions $X_i \subseteq X$, $Y \subseteq Y_i$, and $\card{X_i} = \card{Y_i}$ hold only if $\card{X_i}, \card{Y_i} \in \{\card{Y}, \card{X}\}$. More precisely, a necessary condition for $X_i$ and $Y_i$ ($i \in \{1, 2\}$) to have the desired properties is that
\begin{enumerate}
\item $X_i = X$ and $Y_i = Y \cup \{j\}$ for some $j \in Z$, or
\item $X_i = X \setminus \{j\}$ and $Y_i = Y$ for some $j \in X$.
\end{enumerate}
Our approach for finding sets $X_1, Y_1, X_2, Y_2$ with the desired properties is to make small modifications to the full signature $\fullsig(X,Y)$ and to see how these affect the reduced signature $\reducedsig(X,Y)$.
The only permitted modifications are 
\begin{enumerate}
\item replacement of a single occurrence of $z$ by $y$, and
\item replacement of a single occurrence of $x$ by $z$.
\end{enumerate}
These correspond to the two cases for $X_i$ and $Y_i$ above.

Let $w := \fullsig(X,Y)$. Then $w$ has an even number of occurrences of $z$; hence $\#_\alpha(\psi(w)) \equiv \#_\beta(\psi(w)) \pmod{2}$.
We are going to find two permitted modifications of $w$ that result in strings $u$ and $v$ with the property that
for each $\gamma \in \{\alpha, \beta\}$ the number of occurrences of $\gamma$ in $\psi(u)$ and the number of occurrences of $\gamma$ in $\psi(v)$ have different parities; more specifically,
\begin{align*}
\#_\alpha(\psi(w)) &= \#_\alpha(\psi(u)) \not\equiv \#_\alpha(\psi(v)) \pmod{2}, \\
\#_\beta(\psi(w)) &= \#_\beta(\psi(v)) \not\equiv \#_\beta(\psi(u)) \pmod{2}.
\end{align*}
Consequently, the words $u$ and $v$ are (in some order) the full signatures $\fullsig(X_1,Y_1)$ and $\fullsig(X_2,Y_2)$ for some sets $X_1, Y_1, X_2, Y_2$ such that $\card{X_i} = \card{Y_i}$ for $i \in \{1, 2\}$ and satifying $\reducedsig(X_1,Y_1) = \alpha$, $\reducedsig(X_2,Y_2) = \beta$.

By Lemma~\ref{lem:xyzsubstrings}, $w$, considered as a circular string, contains one of the following strings as a substring:
$yxx \delta$,
$x z^n \delta$,
$y z^n x \delta$,
for some $n \geq 1$ and $\delta \in \{x, y\}$.
In each case, in order to obtain the strings $u$ and $v$ as described above, we will find suitable permitted modifications within this substring and the remaining part of $w$ stays unchanged. For easy reference, Table~\ref{table:xyz} illustrates the argument that follows. In the table, we present substrings of $\fullsig(X,Y) = w$ and the corresponding substrings of the strings $u$ and $v$ obtained from $w$ by applying a permitted modification, as well as the corresponding substrings of $\psi(w)$, $\psi(u)$ and $\psi(v)$. In the last column labeled by $\Delta$, we indicate the changes in the number of $\alpha$'s and $\beta$'s; more precisely,
\[
(\#_\alpha(t) - \#_\alpha(w), \, \#_\beta(t) - \#_\beta(w)),
\]
where $t$ is either $u$ or $v$, depending on the row.
\begin{table}
\begin{center}
\begin{tabular}{cccc}
& substring & $\psi({\cdot})$ & $\Delta$ \\
\hline
$w$ & $yxx \delta$ & $yxx \delta$ &  \\
$u$ & $yxz \delta$ & $yx \beta \delta$ & $(0, 1)$ \\
$v$ & $yzx \delta$ & $y \alpha x \delta$ & $(1, 0)$ \\
\hline
$w$ & $x z^{2k} \delta$ & $x (\beta \alpha)^k \delta$ & \\
$u$ & $x y z^{2k-1} \delta$ & $xy \alpha (\beta \alpha)^{k-1} \delta$ & $(0, -1)$ \\
$v$ & $xz y z^{2k-2} \delta$ & $x \beta y (\alpha \beta)^{k-1} \delta$ & $(-1, 0)$ \\
\hline
$w$ & $x z^{2k + 1} \delta$ & $x \beta (\alpha \beta)^k \delta$ & \\
$u$ & $x y z^{2k} \delta$ & $x y (\alpha \beta)^k \delta$ & $(0, -1)$ \\
$v$ & $z z^{2k + 1} \delta$ & $(\alpha \beta)^{k+1} \delta$ & $(1, 0)$ \\
    &                       & or & \\
    &                       & $(\beta \alpha)^{k+1} \delta$ & $(1, 0)$ \\
\hline
$w$ & $y z^{2k} x \delta$ & $y (\alpha \beta)^k x \delta$ & \\
$u$ & $y y z^{2k-1} x \delta$ & $yy (\alpha \beta)^{k-1} \alpha x \delta$ & $(0, -1)$ \\
$v$ & $y z^{2k+1} \delta$ & $y (\alpha \beta)^k \alpha \delta$ & $(1, 0)$ \\
\hline
$w$ & $y z^{2k+1} x \delta$ & $y (\alpha \beta)^k \alpha x \delta$ & \\
$u$ & $y z^{2k+2} \delta$ & $y (\alpha \beta)^{k+1} \delta$ & $(0, 1)$ \\
$v$ & $y y z^{2k} x \delta$ & $yy (\alpha \beta)^k x \delta$ & $(-1, 0)$ \\
\hline
\end{tabular}
\end{center}
\caption{Possible substrings of $\fullsig(X,Y) = w$ and of strings $u$ and $v$ obtained by applying permitted modifications, as well as applications of $\psi$ to these strings. The column $\Delta$ indicates the changes in the number of $\alpha$'s and $\beta$'s.}
\label{table:xyz}
\end{table}

If $w$ contains a substring $yxx \delta$ for some $\delta \in \{x, y\}$, then let $u$ be the string obtained from $w$ by replacing the second $x$ in this substring by $z$, and let $v$ be the string obtained by replacing the first $x$ by $z$. Then $\psi(w) = \dots yxx \delta \dots$, $\psi(u) = \dots yx \beta \delta \dots$, $\psi(v) = \dots y \alpha x \delta \dots$, and the prefix and the suffix represented by the ellipses are the same in all three strings. The strings $\psi(u)$ and $\psi(w)$ have an equal number of occurrences of $\alpha$, and $\psi(u)$ has one more occurrence of $\beta$ than $\psi(w)$ does. The strings $\psi(v)$ and $\psi(w)$ have an equal number of occurrences of $\beta$, and $\psi(v)$ has one more occurrence of $\alpha$ than $\psi(w)$ does; in symbols,
\begin{align*}
\#_\alpha(\psi(u)) &= \#_\alpha(\psi(w)), & \#_\beta(\psi(u)) &= \#_\beta(\psi(w)) + 1, \\
\#_\alpha(\psi(v)) &= \#_\alpha(\psi(w)) + 1, & \#_\beta(\psi(v)) &= \#_\beta(\psi(w)).
\end{align*}

If $w$ contains a substring $x z^{2k} \delta$ for some $k \geq 1$ and $\delta \in \{x, y\}$, then let $u$ be the string obtained from $w$ by replacing the first $z$ in this substring by $y$, and let $v$ be the string obtained by replacing the second $z$ by $y$. Then $\psi(w) = \dots x (\beta \alpha)^k \delta \dots$, $\psi(u) = \dots xy \alpha (\beta \alpha)^{k-1} \delta \dots$, $\psi(v) = \dots x \beta y (\alpha \beta)^{k-1} \delta \dots$, so it holds that
\begin{align*}
\#_\alpha(\psi(u)) &= \#_\alpha(\psi(w)), & \#_\beta(\psi(u)) &= \#_\beta(\psi(w)) - 1, \\
\#_\alpha(\psi(v)) &= \#_\alpha(\psi(w)) - 1, & \#_\beta(\psi(v)) &= \#_\beta(\psi(w)).
\end{align*}

If $w$ contains a substring $x z^{2k+1} \delta$ for some $k \geq 0$ and $\delta \in \{x, y\}$, then let $u$ be the string obtained from $w$ by replacing the first $z$ in this substring by $y$, and let $v$ be the string obtained by replacing the first $x$ by $z$. Then $\psi(w) = \dots x \beta (\alpha \beta)^k \delta \dots$, $\psi(u) = \dots xy (\alpha \beta)^k \delta \dots$, and, depending on the prefix, we have either $\psi(v) = \dots (\alpha \beta)^{k+1} \delta \dots$ or $\psi(v) = \dots (\beta \alpha)^{k+1} \delta \dots$. In either case, we have
\begin{align*}
\#_\alpha(\psi(u)) &= \#_\alpha(\psi(w)), & \#_\beta(\psi(u)) &= \#_\beta(\psi(w)) - 1, \\
\#_\alpha(\psi(v)) &= \#_\alpha(\psi(w)) + 1, & \#_\beta(\psi(v)) &= \#_\beta(\psi(w)).
\end{align*}

If $w$ contains a substring $y z^{2k} x \delta$ for some $k \geq 1$ and $\delta \in \{x, y\}$, then let $u$ be the string obtained from $w$ by replacing the first $z$ in this substring by $y$, and let $v$ be the string obtained by replacing the first $x$ by $z$. Then $\psi(w) = \dots y (\alpha \beta)^k x \delta \dots$, $\psi(u) = \dots yy (\alpha \beta)^{k-1} \alpha x \delta \dots$, $\psi(v) = \dots y (\alpha \beta)^k \alpha \delta \dots$, so it holds that
\begin{align*}
\#_\alpha(\psi(u)) &= \#_\alpha(\psi(w)), & \#_\beta(\psi(u)) &= \#_\beta(\psi(w)) - 1, \\
\#_\alpha(\psi(v)) &= \#_\alpha(\psi(w)) + 1, & \#_\beta(\psi(v)) &= \#_\beta(\psi(w)).
\end{align*}

If $w$ contains a substring $y z^{2k+1} x \delta$ for some $k \geq 0$ and $\delta \in \{x, y\}$, then let $u$ be the string obtained from $w$ by replacing the first $x$ in this substring by $z$, and let $v$ be the string obtained by replacing the first $z$ by $y$. Then $\psi(w) = \dots y (\alpha \beta)^k \alpha x \delta \dots$, $\psi(u) = \dots y (\alpha \beta)^{k+1} \delta \dots$, $\psi(v) = \dots yy (\alpha \beta)^k x \delta \dots$, so it holds that
\begin{align*}
\#_\alpha(\psi(u)) &= \#_\alpha(\psi(w)), & \#_\beta(\psi(u)) &= \#_\beta(\psi(w)) + 1, \\
\#_\alpha(\psi(v)) &= \#_\alpha(\psi(w)) - 1, & \#_\beta(\psi(v)) &= \#_\beta(\psi(w)).
\end{align*}

We have exhausted all possibilities, and in each case we found strings $u$ and $v$ with the desired properties. This completes the proof.
\end{proof}

We now have all necessary notions and tools to define our next family of Sperner systems. 

\begin{definition}
For an odd integer $m \geq 3$, let
\begin{align*}
C_m & := \{1, 3, \dots, m-2\}, \\
A_m & := \{m\} \cup (C_m + 1) \cup (C_m + 1)', \\
B_m & := \{m\} \cup C_m \cup C_m'.
\end{align*}
Define the set system $\Sp{\Sbuildsd}^m$, as well as $\Sp{\Ssd}^m_1$ and $\Sp{\Ssd}^m_2$ by
\begin{align*}
\Sp{\Sbuildsd}^m  := & \rot{\{A_m, B_m\}}  \cup {} \\
   & \{ Q \in \Qsets \setminus \{ \XYrot{C_m}{C_m+1}, \XYrot{C_m+1}{C_m} \} : \reducedsig Q = \beta \}, \\
\Sp{\Ssd}^m_1 := & \Sp{\Sbuildsd}^m \cup \Sp{\Sbuildall}^m_1, \\
\Sp{\Ssd}^m_2 := & \Sp{\Sbuildsd}^m \cup \Sp{\Sbuildall}^m_2.
\end{align*}
\end{definition}

\begin{remark}
\label{rem:ABtau}
Note that $\tau_p(A_m) = A_m$ and $\tau_p(B_m) = B_m$ for all $p \in \nset{m} \setminus \{m\}$ but $\tau_m(A_m) \neq A_m$ and $\tau_m(B_m) \neq B_m$.
Note also that $\rot{\{A_m\}}$ and $\rot{\{B_m\}}$ comprise a half of the blocks of $\XYrot{C_m + 1}{C_m}$ and $\XYrot{C_m}{C_m + 1}$, respectively. Moreover, $\complementX{\rot{\{A_m\}}} = \XYrot{C_m}{C_m + 1} \setminus \rot{\{B_m\}}$ and $\complementX{\rot{\{B_m\}}} = \XYrot{C_m + 1}{C_m} \setminus \rot{\{A_m\}}$.
\end{remark}

It remains to show that for all $m \geq 3$, the Sperner systems $\Sp{\Ssd}^m_1$ and $\Sp{\Ssd}^m_2$ are nonisomorphic and strongly hypomorphic.

\begin{proposition}
\label{prop:Sminoniso}
For every odd $m \geq 3$, the Sperner systems $\Sp{\Ssd}^m_1$ and $\Sp{\Ssd}^m_2$ are nonisomorphic.
\end{proposition}

\begin{proof}
Suppose, on the contrary, that there exists an isomorphism $\sigma$ of $\Sp{\Ssd}^m_1$ to $\Sp{\Ssd}^m_2$.

We have already observed in the proof of Proposition~\ref{prop:Cnoniso} that each element of $E_m$ occurs in exactly $2^{m-2}$ blocks of both $\Sp{\Sbuildall}^m_1$ and $\Sp{\Sbuildall}^m_2$.

Let $X, Y \subseteq \nset{m}$ with $X \cap Y = \emptyset$, and let $Z := \nset{m} \setminus (X \cup Y)$. Let us count how many times each element of $E_m$ occurs in the blocks of $\XYrot{X}{Y}$. We have that $\XYset{X}{Y}$ has $2^{\card{Z}}$ blocks, and for each $x \in X$, $y \in Y$ and $z \in Z$, the elements $x$ and $x'$ occur in all $2^{\card{Z}}$ blocks, the elements $y$ and $y'$ occur in none of the blocks, and the elements $z$ and $z'$ occur in $2^{\card{Z} - 1}$ blocks.
Consider then the rotated versions $\rho^i(\XYset{X}{Y})$ of $\XYset{X}{Y}$.
For each element $\ell \in \nset{m}$,
there are exactly $\card{X}$ elements $i \in \nset{m}$ for which $\ell \in \rho^i(X)$,
there are exactly $\card{Y}$ elements $i \in \nset{m}$ for which $\ell \in \rho^i(Y)$,
there are exactly $\card{Z}$ elements $i \in \nset{m}$ for which $\ell \in \rho^i(Z)$.
Thus, each element of $E_m$ occurs in $\card{X} \cdot 2^{\card{Z}} + \card{Z} \cdot 2^{\card{Z} - 1}$ blocks of $\XYrot{X}{Y}$.

However, each element of $\nset{m}$ occurs in $\card{C_m} + 1$ blocks of $\rot{\{A_m\}}$ while each element of $\nset{m}'$ occurs only in $\card{C_m}$ blocks of $\rot{\{A_m\}}$. The same is true for the blocks of $\rot{\{B_m\}}$.

By this counting argument, we conclude that $\sigma$ must map the set $\nset{m}$ onto itself, and it must map the set $\nset{m}'$ onto itself.

The number of unprimed odd blocks of $\Sp{\Ssd}^m_1$ is different from the number of unprimed odd blocks of $\Sp{\Ssd}^m_2$, since both have $\Sp{\Sbuildsd}^m$ in common and the blocks of $\Sp{\Sbuildall}^m_1$ are unprimed odd while the blocks of $\Sp{\Sbuildall}^m_2$ are unprimed even. Since the image of any unprimed odd set under $\sigma$ is unprimed odd, we conclude that $\sigma(\Sp{\Ssd}^m_1) \neq \Sp{\Ssd}^m_2$, a contradiction.
\end{proof}

For $j \in C_m + 1$, denote
\[
Q^m_j  := \XYrot{C_m \setminus \{m-2\}}{(C_m + 1) \setminus \{j\}}.
\]

\begin{lemma}
\label{lem:Qjsigy}
For any $j \in C_m + 1$, we have $\reducedsig Q^m_j = \beta$.
Consequently, $\{Q^m_j : j \in C_m+1\} \subseteq \Sp{\Sbuildsd}^m$.
\end{lemma}

\begin{proof}
It is easy to verify that $\psi(\fullsig (C_m \setminus \{m - 2\}, (C_m + 1) \setminus \{j\})$ has precisely two occurrences of $\alpha$, namely, at positions $m-2$ and $m$, and precisely one occurrence of $\beta$, namely, at position $j$. Therefore, $\reducedsig(C_m \setminus \{m - 2\}, (C_m + 1) \setminus \{j\}) = \beta$, and the claim follows.
\end{proof}

\begin{proposition}
\label{prop:Smihypo}
For every odd integer $m \geq 3$, the Sperner systems $\Sp{\Ssd}^m_1$ and $\Sp{\Ssd}^m_2$ are strongly hypomorphic.
\end{proposition}

\begin{proof}
Since $m$ is odd, we have that
\[
\{\{p, q\} : p, q \in \nset{m}\} = \{\{q, q\} : q \in \nset{m}\} \cup \{\{1, j\} + q : j \in C + 1, q \in \nset{m}\}.
\]
By Remark~\ref{rem:G1G2iiprime}, we have for all $i \in \nset{m}$ that $(\Sp{\Sbuildall}^m_1)_{\{i, i'\}} = (\Sp{\Sbuildall}^m_2)_{\{i, i'\}}$.
Hence,
\begin{multline*}
(\Sp{\Ssd}^m_1)_{\{i, i'\}} =
(\Sp{\Ssd}^m \cup \Sp{\Sbuildall}^m_1)_{\{i, i'\}} =
(\Sp{\Ssd}^m)_{\{i, i'\}} \cup (\Sp{\Sbuildall}^m_1)_{\{i, i'\}} = \\
(\Sp{\Ssd}^m)_{\{i, i'\}} \cup (\Sp{\Sbuildall}^m_2)_{\{i, i'\}} =
(\Sp{\Ssd}^m \cup \Sp{\Sbuildall}^m_2)_{\{i, i'\}} =
(\Sp{\Ssd}^m_2)_{\{i, i'\}}.
\end{multline*}
Consequently, $(\Sp{\Ssd}^m_1)^*_{\{i, i'\}} = (\Sp{\Ssd}^m_2)^*_{\{i, i'\}}$.

Thus, it suffices to consider the cases when $I$ is of the form $\{1, j\}$, $\{1, j'\}$, $\{1', j\}$ or $\{1', j'\}$ for some $j \in C + 1$. We will show that then both $(A_m)_I$ and $(B_m)_I$ includes a block of $(Q^m_\ell)_I$ for some $\ell \in C_m + 1$, which is a block of $(\Sp{\Sbuildsd}^m)_I$ by Lemma~\ref{lem:Qjsigy}. This means that we can ignore $(A_m)_I$ and $(B_m)_I$ when we consider the minimal elements of $(\Sp{\Sbuildsd}^m)_I$. But then the transposition $\tau_m$ maps the set $(\Sp{\Sbuildsd}^m)_I \setminus \{(A_m)_I, (B_m)_I\}$ onto itself by Lemma~\ref{lem:XYtransposition} and Remark~\ref{rem:ABtau} and it maps $(\Sp{\Sbuildall}^m_1)_I$ onto $(\Sp{\Sbuildall}^m_2)_I$ by Remark~\ref{rem:G1G2isomorphic}. Thus $\tau_m$ is an isomorphism between $(\Sp{\Ssd}^m_1)_I$ and $(\Sp{\Ssd}^m_2)_I$.

Consider first $(A_m)_I$.
If $I = \{1, j\}$ or $I = \{1, j'\}$, then we have
$(A_m)_I = (\{1, m\} \cup (C_m + 1) \cup (C_m + 1)')_I$.
If $I = \{1', j\}$ or $I = \{1', j'\}$, then we have
$(A_m)_I = (\{1', m\} \cup (C_m + 1) \cup (C_m + 1)')_I$.
Both $\{1, m\} \cup (C_m + 1) \cup (C_m + 1)'$ and $\{1', m\} \cup (C_m + 1) \cup (C_m + 1)'$ are blocks of $\XYrot{C_m + 1}{(C_m + 2) \setminus \{m\}} = \XYrot{C_m}{(C_m + 1) \setminus \{m - 1\}}$, and every block of this set system is a superset of a block of $Q^m_{m-1}$.
Thus $(A_m)_I$ includes a block of $(Q^m_{m-1})_I$.

Consider then $(B_m)_I$.
If $I = \{1, j\}$ or $I = \{1', j\}$, then we have
$(B_m)_I = (\{j, m\} \cup C_m \cup C_m')_I$.
If $I = \{1, j'\}$ or $I = \{1', j'\}$, then we have
$(B_m)_I = (\{j', m\} \cup C_m \cup C_m')_I$.
Both $\{j, m\} \cup C_m \cup C_m'$ and $\{j', m\} \cup C_m \cup C_m'$ are blocks of $\XYrot{C_m}{(C_m + 1) \setminus \{j\}}$, and every block of this set system is a superset of a block of $Q^m_j$.
Thus $(B_m)_I$ includes a block of $(Q^m_j)_I$.
\end{proof}

We still want to verify that the Boolean functions associated with the Sperner systems $\Sp{\Ssd}^m_i$ are self-dual, i.e., members of the clone $SM$.

\begin{lemma}
\label{lem:abovemidlayer}
For all $J \subseteq E_m$ with $\card{J} \geq m + 1$, there is some $H \in \Sp{\Sbuildsd}^m$ with $H \subseteq J$.
\end{lemma}

\begin{proof}
It suffices to prove the claim for sets $J$ such that $\card{J} = m + 1$.
Under this assumption, we have that $J \in \XYrot{X}{Y}$ for some $X, Y \subseteq \nset{m}$ with $\card{X} = \card{Y} + 1$ and $X \cap Y = \emptyset$.

By Lemma~\ref{lemma:partysig}, there exist some $X_2$ and $Y_2$ with $X_2 \subseteq X$, $Y \subseteq Y_2$, $\card{X_2} = \card{Y_2}$ such that $\reducedsig(X_2,Y_2) = \beta$.
If $\{X_2, Y_2\} \neq \{C_m, C_m + 1\}$ then $\XYrot{X_2}{Y_2} \subseteq \Sp{\Sbuildsd}^m$, and by Lemma~\ref{lem:smallerX} there is some $H \in \XYrot{X_2}{Y_2}$ with $H \subseteq J$.

If $\{X_2, Y_2\} = \{C_m, C_m+1\}$, then the condition $\reducedsig(X_2,Y_2) = \beta$ implies $\XYrot{X_2}{Y_2} = \XYrot{C_m+1}{C_m}$. There are two different possibilities for $\XYrot{X}{Y}$.
\begin{itemize}
\item
If $\XYrot{X}{Y} = \XYrot{\{m\} \cup (C_m + 1)}{C_m}$, then $J = \rho^i(\{m, m'\} \cup (C_m + 1) \cup (C_m + 1)')$ for some $i \in \nset{m}$, and thus $\rho^i(\{A_m\}) \subseteq J$.

\item
Otherwise, $\XYrot{X}{Y} = \XYrot{C_m + 1}{C_m \setminus \{j\}}$ for some $j \in C_m$. Let then $X_2' := (C_m + 1) \setminus \{j+1\}$ and $Y_2' := C_m \setminus \{j\}$. Since $\reducedsig(X_2',Y_2') = \beta$, we can use $X_2'$ and $Y_2'$ instead of $X_2$ and $Y_2$, and we have $\XYrot{X_2'}{Y_2'} \subseteq \Sp{\Sbuildsd}^m$ and some $H \in \XYrot{X_2'}{Y_2'}$ satisfying $H \subseteq J$.
\qedhere
\end{itemize}
\end{proof}

\begin{lemma}
\label{lem:midlayercomplements}
Let $m \geq 3$ be an odd integer, and let $i \in \{1, 2\}$. For every subset $J \subseteq E_m$ with $\card{J} = m$, the Sperner system $\Sp{\Ssd}^m_i$ contains exactly one of the sets $J$ and $\complementX{J}$.
\end{lemma}

\begin{proof}
Let $J \subseteq E_m$ with $\card{J} = m$.
By Proposition~\ref{prop:qsetsPartition}, there exist unique $\XYrot{X}{Y} \in \QsetsE$ such that $J \in \XYrot{X}{Y}$ and $\complementX{J} \in \XYrot{Y}{X}$. Since $m$ is odd, $\XYrot{X}{Y} \cap \XYrot{Y}{X} = \emptyset$.
Assume first that $\XYrot{X}{Y} \in \QsetsE \setminus \{\XYrot{\emptyset}{\emptyset}, \XYrot{C_m}{C_m + 1}, \XYrot{C_m + 1}{C_m}\}$.
By Lemma~\ref{lem:XYYXalphabeta}, one of $\reducedsig \XYrot{X}{Y}$ and $\reducedsig \XYrot{Y}{X}$ is $\alpha$ and the other is $\beta$.
By the definition of $\Sp{\Sbuildsd}^m$, exactly one of $J$ and $\complementX{J}$ is in $\Sp{\Sbuildsd}^m$ and hence in $\Sp{\Ssd}^m_i$.

Assume then that $\XYrot{X}{Y} \in \{\XYrot{C_m}{C_m + 1}, \XYrot{C_m + 1}{C_m}\}$.
By Remark~\ref{rem:ABtau}, either $J \in \rot{\{A_m, B_m\}}$ or $J \in \complementX{\rot{\{A_m, B_m\}}}$.
It then follows from the definition of $\Sp{\Sbuildsd}^m$ that exactly one of $J$ and $\complementX{J}$ is in $\Sp{\Sbuildsd}^m$ and hence in $\Sp{\Ssd}^m_i$.

Finally, assume that $\XYrot{X}{Y} = \XYrot{\emptyset}{\emptyset}$. By Remark~\ref{rem:emptyemptyG}, $\XYrot{\emptyset}{\emptyset} = \Sp{\Sbuildall}^m_1 \cup \Sp{\Sbuildall}^m_2$.
Since $m$ is odd, $\complementX{\Sp{\Sbuildall}^m_1} = \Sp{\Sbuildall}^m_2$ by Remark~\ref{rem:Gcomplements}.
We conclude that one of $J$ and $\complementX{J}$ is in $\Sp{\Sbuildall}^m_1$ and the other is in $\Sp{\Sbuildall}^m_2$; hence exactly one of $J$ and $\complementX{J}$ is in $\Sp{\Ssd}^m_i$.
\end{proof}

\begin{proposition}
\label{prop:SmiSM}
For every odd $m \geq 3$ and $i \in \{1, 2\}$, the Boolean function associated with the Sperner system $\Sp{\Ssd}^m_i$ is self-dual.
\end{proposition}

\begin{proof}
The Sperner system $\Sp{\Ssd}^m_i$ is $m$-homogeneous.
By Lemma~\ref{lem:abovemidlayer}, every set $J \subseteq E_m$ with $\card{J} > m$ includes a block of $\Sp{\Sbuildsd}^m$.
By Lemma~\ref{lem:midlayercomplements}, for every set $J \subseteq E_m$ with $\card{J} = m$ it holds that exactly one of $J$ and $\complementX{J}$ is in $\Sp{\Ssd}^m_i$.
It follows that the Boolean function associated with $\Sp{\Ssd}^m_i$ is self-dual.
\end{proof}


\section{Reconstructible Sperner systems}
\label{sec:positive}

Having found several infinite families of nonreconstructible Sperner systems, we now prove some positive results about reconstructibility. Our aim is to show that the members of the clones $\Lambda$ and $V$ of semilattice polynomial operations of sufficiently large arity are reconstructible.

Let $\Sp{\Sgena}$ be a set system over $A$. Denote by $U_{\Sp{\Sgena}}$ the union of the blocks of $\Sp{\Sgena}$. (Note that $U_{\Sp{\Sgena}}$ is a subset of $A$ and it may be a proper subset.) The elements of $U_{\Sp{\Sgena}}$ are said to be \emph{essential} in $\Sp{\Sgena}$ and the elements of $A \setminus U_{\Sp{\Sgena}}$ are said to be \emph{inessential} in $\Sp{\Sgena}$. Note that $\Sp{\Sgena}$ is also a set system over any set containing $U_{\Sp{\Sgena}}$.

\begin{lemma}
\label{lem:dummyid}
Let $\Sp{\Sgena}$ be a Sperner system over $A$, and assume that an element $a$ of $A$ is inessential in $\Sp{\Sgena}$.
If $I \in \couples[A]$ and $I \not\subseteq U_{\Sp{\Sgena}}$, then $\Sp{\Sgena}^*_I$ is isomorphic to $\Sp{\Sgena}$ when viewed as a set system over $A \setminus \{a\}$.
\end{lemma}

\begin{proof}
If $I \not\subseteq U_{\Sp{\Sgena}}$, then (at least) one of the elements of $I$ is not a member of any block of $\Sp{\Sgena}$. If one of the elements of $I$ is a member of some blocks of $\Sp{\Sgena}$, we can choose this element as the representative of the class $I$ of $\theta_I$, and we have that $S_I = S$ for every block $S$ of $\Sp{\Sgena}$. Otherwise $I \cap U_{\Sp{\Sgena}} = \emptyset$, and again we clearly have that $S_I = S$ for every block $S$ of $\Sp{\Sgena}$. The claim thus follows.
\end{proof}

\begin{lemma}
\label{lem:dummyminors}
Let $\Sp{\Sgena}$ be a Sperner system over $A$.
Then $\card{U_{\Sp{\Sgena}^*_I}} \leq \card{U_{\Sp{\Sgena}}}$ for all $I \in \couples[A]$, and
the inequality is strict if and only if $I \subseteq U_{\Sp{\Sgena}}$.
\end{lemma}

\begin{proof}
If $I \not\subseteq U_{\Sp{\Sgena}}$, then it follows from Lemma~\ref{lem:dummyid} that $\card{U_{\Sp{\Sgena}^*_I}} = \card{U_{\Sp{\Sgena}}}$. If $I \subseteq U_{\Sp{\Sgena}}$, then $U_{\Sp{\Sgena}^*_I}$ is a subset of $U_{\Sp{\Sgena}} / \theta_I$, and we have $\card{U_{\Sp{\Sgena}^*_I}} \leq \card{U_{\Sp{\Sgena}} / \theta_I} = \card{U_{\Sp{\Sgena}}} - 1 < \card{U_{\Sp{\Sgena}}}$.
\end{proof}

For any $S \subseteq \nset{n}$, denote by $\vect{e}_S$ the $n$-tuple $(a_1, \dots, a_n)$ satisfying $a_i = 1$ if and only if $i \in S$ and $a_i = 0$ otherwise.

\begin{lemma}
Let $\Sp{\Sgena}$ be a Sperner system over $\nset{n}$, and let $\mathbf{A}$ be a bounded distributive lattice.
For every $i \in \nset{n}$, the $i$-th argument is essential in $t_{\Sp{\Sgena}}^\mathbf{A}$ if and only if $i$ is essential in $\Sp{\Sgena}$.
\end{lemma}

\begin{proof}
If $i$ is not essential in $\Sp{\Sgena}$, then it is clear that $t_{\Sp{\Sgena}}^\mathbf{A}$ does not depend on the $i$-th argument.
Assume then that $i$ is essential in $\Sp{\Sgena}$.
Let $S$ be a block of $\Sp{\Sgena}$ containing $i$.
Then $t_{\Sp{\Sgena}}^\mathbf{A}(\vect{e}_S) = 1$ and $t_{\Sp{\Sgena}}^\mathbf{A}(\vect{e}_{S \setminus \{i\}}) = 0$, which shows that the $i$-th argument is essential in $t_{\Sp{\Sgena}}^\mathbf{A}$.
\end{proof}

The following is an immediate consequence of the result proved in~\cite[Theorem~6]{CouLehlattice} and~\cite[Theorem~3.10]{CLWmonotone}

\begin{theorem}
\label{thm:CLlattice6}
If $n \geq 2$ and $\Sp{\Sgena}$ is a Sperner system on $\nset{n}$ such that every element of $\nset{n}$ is essential in $\Sp{\Sgena}$, then there exists $I \in \couples$ such that $\Sp{\Sgena}^*_I$ has $n - 1$ essential elements precisely unless $n = 3$ and $\Sp{\Sgena} = \{\{1, 2\}, \{1, 3\}, \{2, 3\}\}$.
\end{theorem}

A function $f \colon A^n \to B$ is \emph{totally symmetric} if for every permutation $\sigma$ of $A$ we have $f(\vect{a}) = f(\vect{a} \sigma)$ for all $\vect{a} \in A^n$.
Analogously, a Sperner system $\Sp{\Sgena}$ on $A$ is \emph{totally symmetric} if $\sigma(\Sp{\Sgena}) = \Sp{\Sgena}$ for every permutation $\sigma$ of $A$.
It is obvious that $\Sp{\Sgena}$ is totally symmetric is and only if $t_{\Sp{\Sgena}}^\mathbf{A}$ is totally symmetric.

\begin{theorem}[{\cite[Theorem~5.1]{LehtonenSymmetric}}]
Suppose $n \geq \card{A} + 2$, and let $f \colon A^n \to B$. If $f$ is totally symmetric, then $f$ is reconstructible.
\end{theorem}

In the special case of lattice term operations of the two-element lattice, this translates into the following.

\begin{corollary}
\label{cor:symmetric}
Every totally symmetric Sperner system over a set with at least four elements is reconstructible.
\end{corollary}

Let us recall a useful theorem by Willard~\cite{Willard}.

\begin{theorem}[{Willard~\cite[Theorem~2.6]{Willard}}]
\label{thm:Willard2.6}
Suppose $n \geq \max(\card{A}, 3) + 2$. Let $f \colon A^n \to B$ and assume that $f$ depends on all of its $n$ arguments. If every identification minor of $f$ that depends on $n - 1$ arguments is totally symmetric, then $f$ is determined by either $\operatorname{supp}$ or $\operatorname{oddsupp}$.
\end{theorem}

The precise definition of ``$f$ is determined by $\operatorname{supp}$'' or ``$f$ is determined by $\operatorname{oddsupp}$'' is not important to us here, but what is important is the fact that functions determined by either $\operatorname{supp}$ or $\operatorname{oddsupp}$ are totally symmetric. Theorem~\ref{thm:Willard2.6} then translates into the following in the special case of lattice term functions of the two-element lattice.

\begin{corollary}
\label{cor:Willard}
Suppose $n \geq 5$. Let $\Sp{\Sgena}$ be a Sperner system on $\nset{n}$ and assume that every element of $\nset{n}$ is essential in $\Sp{\Sgena}$. If for all $I \in \couples$ it holds that $\Sp{\Sgena}^*_I$ is totally symmetric whenever $\Sp{\Sgena}^*_I$ has $n-1$ essential elements, then $\Sp{\Sgena}$ is totally symmetric.
\end{corollary}

The bound $n \geq 5$ in Corollary~\ref{cor:Willard} is sharp, as witnessed by the Sperner system $\Sp{\Sgena}$ of Example~\ref{ex:Sp4}.

\begin{theorem}
\label{thm:dummyweaklyrec}
The class of Sperner systems over $A$ with inessential elements is weakly reconstructible.
\end{theorem}

\begin{proof}
Let $\Sp{\Sgena}$ and $\Sp{\Sgenb}$ be Sperner systems over $A$, and assume that both have inessential elements.
By Lemmas~\ref{lem:dummyid} and~\ref{lem:dummyminors}, the deck of $\Sp{\Sgena}$ comprises copies of $\Sp{\Sgena}$ and Sperner systems with fewer essential elements; similarly, the deck of $\Sp{\Sgenb}$ comprises copies of $\Sp{\Sgenb}$ and Sperner systems with fewer essential elements. If $\deck \Sp{\Sgena} = \deck \Sp{\Sgenb}$, then we necessarily have that $\Sp{\Sgena} \equiv \Sp{\Sgenb}$.
\end{proof}

\begin{theorem}
\label{thm:1homogeneous1block}
Let $\Sp{\Sgena}$ be a Sperner system over $\nset{n}$, and assume that $n \geq 4$.
If $\Sp{\Sgena}$ is $1$-homogeneous or has exactly one block, then $\Sp{\Sgena}$ is reconstructible.
\end{theorem}

\begin{proof}
Under our assumptions, it holds that $\Sp{\Sgena}$ is isomorphic to $\{\{1\}, \dots, \{m\}\}$ or $\{\nset{m}\}$ for some $m \in \nset{n}$. We may assume, without loss of generality, that $\Sp{\Sgena}$ is equal to one of $\{\{1\}, \dots, \{m\}\}$ or $\{\nset{m}\}$.
If $\Sp{\Sgena} = \{\{1\}, \dots, \{m\}\}$, then the deck of $\Sp{\Sgena}$ comprises $\{\{1\}, \dots, \{m - 1\}\}$ with multiplicity $\binom{m}{2}$ and $\{\{1\}, \dots, \{m\}\}$ with multiplicity $\binom{n}{2} - \binom{m}{2}$.
If $\Sp{\Sgena} = \{\nset{m}\}$, then the deck of $\Sp{\Sgena}$ comprises $\{\nset{m - 1}\}$ with multiplicity $\binom{m}{2}$ and $\{\nset{m}\}$ with multiplicity $\binom{n}{2} - \binom{m}{2}$.

Assume first that $m = n$. Then $\Sp{\Sgena}$ is totally symmetric, and it is reconstructible by Corollary~\ref{cor:symmetric}.

Assume then that $m < n$. Then $\Sp{\Sgena}$ has inessential elements. Let $\Sp{\Sgenb}$ be a reconstruction of $\Sp{\Sgena}$.
If $\Sp{\Sgenb}$ has inessential elements, then it follows from Theorem~\ref{thm:dummyweaklyrec} that $\Sp{\Sgena} \equiv \Sp{\Sgenb}$.
Suppose then, on the contrary, that all elements of $\nset{n}$ are essential in $\Sp{\Sgenb}$.
It follows from Proposition~\ref{thm:CLlattice6} that $\Sp{\Sgenb}$ has a card with $n - 1$ essential elements.
Then we necessarily have that $m = n - 1$ and this card is $\{\{1\}, \dots, \{n - 1\}\}$ or $\{\nset{n - 1}\}$.
We need to split the analysis on two cases depending on the value of $n$.

Consider first the case that $n \geq 5$. Since every card of $\Sp{\Sgenb}$ that has $n - 1$ essential elements is totally symmetric, it follows from Corollary~\ref{cor:Willard} that $\Sp{\Sgenb}$ is totally symmetric as well. Corollary~\ref{cor:symmetric} then implies that $\Sp{\Sgenb}$ is reconstructible; hence $\Sp{\Sgena} \equiv \Sp{\Sgenb}$. But $\Sp{\Sgena}$ and $\Sp{\Sgenb}$ have a different number of essential elements, so $\Sp{\Sgena}$ is certainly not equivalent to $\Sp{\Sgenb}$. We have reached a contradiction.

It remains to consider the case that $n = 4$.
Assume first that $\Sp{\Sgena} = \{1, 2, 3\}$; hence $\deck \Sp{\Sgena}$ comprises $\{1, 2, 3\}$ with multiplicity $3$ and $\{1, 2\}$ with multiplicity $3$.
In order for $\Sp{\Sgenb}$ to have $\{1, 2, 3\}$ as a card, $\Sp{\Sgenb}$ must contain a block with at least three elements.
If $\Sp{\Sgenb}$ contains a block with four elements, then $\Sp{\Sgenb} = \{1, 2, 3, 4\}$; but this system is totally symmetric and hence reconstructible by Corollary~\ref{cor:symmetric}, so this is not possible.
Thus $\Sp{\Sgenb}$ contains a block with three elements; say, $\{1, 2, 3\}$ is a block of $\Sp{\Sgenb}$. Since no element of $\nset{n}$ is inessential in $\Sp{\Sgenb}$, there must be another block that contains the element $4$.
Then $\Sp{\Sgenb}^*_{\{1,2\}}$ has at least two blocks: $\{1, 3\}$ and a block containing $4$. This is not possible, since all cards of $\Sp{\Sgena}$ (and hence of $\Sp{\Sgenb}$) have exactly one block.
We have reached a contradiction.

Assume then that $\Sp{\Sgena} = \{\{1\}, \{2\}, \{3\}\}$; hence $\deck \Sp{\Sgena}$ comprises $\{\{1\}, \{2\}, \{3\}\}$ with multiplicity $3$ and $\{\{1\}, \{2\}\}$ with multiplicity $3$.
If $\Sp{\Sgenb}$ has a block $S$ with $\card{S} \geq 3$, then for any two-element subset $I$ of $S$, the card $\Sp{\Sgenb}^*_I$ has a block with at least two elements; this is not possible.
If $\Sp{\Sgenb}$ has a block $S$ with $\card{S} = 2$, then $S$ is a two-element block of the card $\Sp{\Sgenb}^*_{\nset{4} \setminus S}$; this is not possible.
Thus all blocks of $\Sp{\Sgenb}$ are singletons, so, in fact, $\Sp{\Sgenb} = \{\{1\}, \{2\}, \{3\}, \{4\}\}$; but this is totally symmetric and hence reconstructible by Corollary~\ref{cor:symmetric}. But $\Sp{\Sgena}$ and $\Sp{\Sgenb}$ are clearly not isomorphic.
We have reached a contradiction also in this case.

We conclude that all reconstructions of $\Sp{\Sgena}$ are isomorphic to $\Sp{\Sgena}$, that is, $\Sp{\Sgena}$ is reconstructible. This completes the proof.
\end{proof}

Examples~\ref{ex:Sp2} and~\ref{ex:Sp3} show that the bound $n \geq 4$ in Theorem~\ref{thm:1homogeneous1block} is sharp.

Since we are aiming at proving the reconstructibility of the clones $\Lambda$ and $V$, we must not forget constant functions.

\begin{theorem}[{\cite[Example~3.4]{LehtonenSymmetric}}]
\label{thm:constantreconstructible}
If $f \colon A^n \to B$ is a constant function and $n > \card{A}$, then $f$ is reconstructible.
\end{theorem}

\begin{corollary}
\label{cor:LambdaV}
If $\mathcal{C}$ is one of the clones $\Lambda$ and $V$ of Boolean functions, then the class $\mathcal{C}^{(\geq 4)}$ is reconstructible.
\end{corollary}

\begin{proof}
The clone $V$ comprises the constant functions and the functions of the form $t_{\Sp{\Sgena}}^\mathbf{B}$, where $\Sp{\Sgena}$ is a $1$-homogeneous Sperner system. The clone $\Lambda$ comprises the constant functions and the functions of the form $t_{\Sp{\Sgena}}^\mathbf{B}$, where $\Sp{\Sgena}$ is a Sperner system with exactly one block.
The claim then follows from Proposition~\ref{prop:termopreconstr} and Theorems~\ref{thm:1homogeneous1block} and~\ref{thm:constantreconstructible}.
\end{proof}


\section{Reconstructibility of Post classes}
\label{sec:classification}

Our results on the reconstructibility of Sperner systems straightforwardly translate into equivalent statements about distributive lattice polynomial functions, in particular monotone Boolean functions.
In fact, we are now very close to having a complete classification of the clones of Boolean functions in regard to reconstructibility. We only lack information on the reconstructibility of the subclones of the clone $L$ of linear functions that are not subclones of the clone $M$ of monotone functions. This information is provided by the result from~\cite{LehtonenLinear} that we quote below. Note that the members of the clone $L$ are precisely the affine functions over the finite field of order $2$.

\begin{theorem}[{\cite[Theorem~4.7]{LehtonenLinear}}]
\label{thm:affine}
Let $(G; +, \cdot)$ be a finite field of order $q$. The affine functions of arity at least $\max(q, 3) + 1$ over $(G; +, \cdot)$ are reconstructible.
\end{theorem}

Recall that the \emph{dual} of a Boolean function $f \colon \{0, 1\}^n \to \{0, 1\}$ is the function $f^\mathrm{d} \colon \{0, 1\}^n \to \{0, 1\}$ given by $f^\mathrm{d}(\vect{a}) = \overline{f(\overline{\vect{a}})}$ for all $\vect{a} \in \{0, 1\}^n$. Here $\overline{0} = 1$ and $\overline{1} = 0$. 

\begin{lemma}
\label{lem:dual}
For any $f \colon A^n \to B$ and any $I \in \couples$, we have $(f^\mathrm{d})_I = (f_I)^\mathrm{d}$.
\end{lemma}

\begin{proof}
For all $\vect{a} \in \{0, 1\}^{n-1}$,
\[
(f^\mathrm{d})_I(\vect{a})
= f^\mathrm{d}(\vect{a} \delta_I)
= \overline{f(\overline{\vect{a} \delta_I})}
= \overline{f(\overline{\vect{a}} \delta_I)}
= \overline{f_I(\overline{\vect{a}})}
= (f_I)^\mathrm{d}(\vect{a}).
\qedhere
\]
\end{proof}

\begin{remark}
\label{rem:dual}
It follows from Lemma~\ref{lem:dual} that $f$ is reconstructible if and only if $f^\mathrm{d}$ is reconstructible, and a class $\mathcal{C} \subseteq \cl{O}_{\{0,1\}}$ is reconstructible, weakly reconstructible, or recognizable if and only if its dual class $\mathcal{C}^\mathrm{d} := \{f^\mathrm{d} : f \in \mathcal{C}\}$ has the same property.
\end{remark}

\begin{theorem}
Let $\mathcal{C}$ be a clone on $\{0, 1\}$.
If $\mathcal{C}$ is included in $\Lambda$, $V$ or $L$, then $\mathcal{C}^{(\geq 4)}$ is reconstructible.
Otherwise, $\mathcal{C}^{(\geq n)}$ is not weakly reconstructible for any $n \geq 1$, and $\mathcal{C}$ contains pairs of nonequivalent strongly hypomorphic functions of arbitrarily high arity.
\end{theorem}

\begin{proof}
Corollary~\ref{cor:LambdaV} and Theorem~\ref{thm:affine} provide the claim about the clones included in $\Lambda$, $V$ or $L$.
Reading off of Post's lattice (see Figure~\ref{fig:PostsLattice} in Appendix~\ref{app:Post}), we see that if $\mathcal{C}$ is a clone that is not included in $\Lambda$, $V$ or $L$, then $\mathcal{C}$ includes $SM$, $M_c U_\infty$ or $M_c W_\infty$.
Propositions~\ref{prop:Sminoniso}, \ref{prop:Smihypo} and \ref{prop:SmiSM} establish that the clone $SM$ contains pairs of nonequivalent strongly hypomorphic functions of arbitrarily high arity.
Remark~\ref{rem:UmiMcUinf} and Propositions~\ref{prop:Uminoniso}, \ref{prop:U2m+1hypo} and \ref{prop:U2m+2hypo} do this for the clone $M_c U_\infty$.
By Remark~\ref{rem:dual}, the same holds for the clone $M_c W_\infty$, which is the dual of $M_c U_\infty$.
The claim thus follows, and the proof is complete.
\end{proof}


\section{Further nonreconstructible functions}
\label{sec:further}

The nonreconstructible functions that we have seen so far are all order-preserving, because they arise as (truncated) term operations of a bounded distributive lattice, as described in Section~\ref{sub:relationship}. It is not difficult to see that there exist also nonreconstructible functions that are not order-preserving. Namely, there are various simple ways of building new nonreconstructible functions from any given one, the resulting functions being not necessarily order-preserving. We briefly mention here, without proof, examples of a few such methods. Throughout this section, we let $f, g \colon A^n \to B$.

\begin{example}
\textbf{Relabeling the domain and codomain.}
For any permutations $\phi \colon A \to A$ and $\psi \colon B \to B$, define $f^{\psi,\phi} \colon A^n \to B$ by the rule
\[
f^{\psi,\phi}(a_1, \dots, a_n) = \psi(f(\phi(a_1), \dots, \phi(a_n))).
\]
We have $(f^{\psi,\phi})_I = (f_I)^{\psi,\phi}$ for any $I \in \couples$, and $f^{\psi,\phi} \equiv g^{\psi,\phi}$ if and only if $f \equiv g$. In other words, if $f$ is not reconstructible, then neither is $f^{\psi,\phi}$.
\end{example}

\begin{example}
\textbf{Modifying the diagonal.}
For any map $\Delta \colon A \to B$, define $f^\Delta \colon A^n \to B$ by the rule
\[
f^\Delta(a_1, \dots, a_n) =
\begin{cases}
\Delta(a_1), & \text{if $a_i = a_j$ for all $i, j \in \nset{n}$,} \\
f(a_1, \dots, a_n), & \text{otherwise.}
\end{cases}
\]
We have $(f^\Delta)_I = (f_I)^\Delta$ for any $I \in \couples$, and $f^\Delta \equiv g^\Delta$ if and only if $f \equiv g$. In other words, if $f$ is not reconstructible, then neither is $f^\Delta$.
\end{example}

\begin{example}
\textbf{Extending the domain and codomain.}
For any sets $A'$ and $B'$ with $A' \supseteq A$ and $B' \supseteq B$ and for any map $\theta \colon \mathcal{P}(A') \to B'$, define $f' \colon (A')^n \to B'$ by the rule
\[
f'(a_1, \dots, a_n) =
\begin{cases}
f(a_1, \dots, a_n), & \text{if $(a_1, \dots, a_n) \in A^n$,} \\
\theta(\{a_1, \dots, a_n\}), & \text{otherwise.}
\end{cases}
\]
We have $(f')_I = (f_I)'$ for any $I \in \couples$, and $f' \equiv g'$ if and only if $f \equiv g$. In other words, if $f$ is not reconstructible, then neither is $f'$
\end{example}

\begin{example}
\textbf{Duplicating and padding.}
For any map $\theta \colon \mathcal{P}(A \times \{0, 1\}) \to B$, define $\hat{f} \colon (A \times \{0, 1\})^n \to B$ by the rule
\[
\hat{f}((a_1, b_1), \dots, (a_n, b_n)) =
\begin{cases}
f(a_1, \dots, a_n), & \text{if $b_i = b_j$ for all $i, j \in \nset{n}$,} \\
\theta(\{(a_1, b_1), \dots, (a_n, b_n)\}), & \text{otherwise.}
\end{cases}
\]
We have $(\hat{f})_I = \widehat{(f_I)}$ for any $I \in \couples$, and $\hat{f} \equiv \hat{g}$ if and only if $f \equiv g$. In other words, if $f$ is not reconstructible, then neither is $\hat{f}$.
\end{example}


\section{Connection to a reconstruction problem for hypergraphs}
\label{sec:hypergraph}

Our work bears a surprising connection to a completely different reconstruction problem that has been formulated for hypergraphs and vertex-deleted subhypergraphs. Set systems being essentially the same thing as hypergraphs, Sperner systems are certainly fit to serve as objects of the reconstruction problem for hypergraphs; however, cards are formed in a different way.

\begin{definition}
The \emph{reconstruction problem for hypergraphs and vertex-deleted subhypergraphs} comprises the following data.
The objects are all hypergraphs with a finite vertex set.
The equivalence is given by the isomorphism between hypergraphs.
The size of a hypergraph is the number of its vertices.
For each $n \in \mathbb{N}$, the index set $I_n$ is $\nset{n}$. We may assume that there is a fixed bijection between $\nset{n}$ and any $n$-element set, and we identify each element of $\nset{n}$ with its image under this bijection.
For each hypergraph $G$ and $v \in V(G)$, the derived object $G_v$ is the vertex-deleted hypergraph $G - v$ whose vertex set is $V(G) \setminus \{v\}$ and edge set is $\{e \in E(G) : v \notin e\}$.
Hence the cards of a hypergraph $G$ are the isomorphism types $(G - v) / {\equiv}$ of the vertex-deleted hypergraphs $G - v$ for $v \in V(G)$, and the deck of $G$ is the multiset $\{(G - v) / {\equiv} : v \in V(G)\}$.
\end{definition}

Kocay~\cite{Kocay} and Kocay and Lui~\cite{KocLui} have presented infinite families of nonreconstructible hypergraphs.
One of the families of Sperner systems that we have constructed in this paper, namely $\Sp{\Smon}^m_i$, $m \geq 3$, $i \in \{1, 2\}$ (see Definition~\ref{def:monotone}), actually turns out to be another example of an infinite family of nonreconstructible hypergraphs.

We have already shown in Proposition~\ref{prop:Cnoniso} that $\Sp{\Smon}^m_1$ and $\Sp{\Smon}^m_2$ are nonisomorphic. It remains to show that $\Sp{\Smon}^m_1$ and $\Sp{\Smon}^m_2$ are hypomorphic (with respect to the reconstruction problem for hypergraphs and vertex-deleted subhypergraphs); we will show that they are in fact strongly hypomorphic.

\begin{proposition}
\label{prop:hypergraph}
For all $m \geq 3$, the hypergraphs $\Sp{\Smon}^m_1$ and $\Sp{\Smon}^m_2$ are strongly hypomorphic with respect to the reconstruction problem for hypergraphs and vertex-deleted subhypergraphs.
\end{proposition}

\begin{proof}
Let $m \geq 3$, and let $v \in E_m$. We claim that if $r \in \nset{m}$ and $v \notin \{r, r'\}$, then $\tau_r(\Sp{\Sbuildall}^m_1 - v) = \Sp{\Sbuildall}^m_2 - v$. For, if $S \in \Sp{\Sbuildall}^m_1 - v$, then $S$ is unprimed odd and $v \notin S$; hence $\tau_r(S)$ is unprimed even and $v \notin S$, i.e., $\tau_r(S) \in \Sp{\Sbuildall}^m_2 - v$. Thus $\tau_r(\Sp{\Sbuildall}^m_1 - v) \subseteq \Sp{\Sbuildall}^m_2 - v$. A similar argument shows that $\tau_r(\Sp{\Sbuildall}^m_2 - v) \subseteq \Sp{\Sbuildall}^m_1 - v$. Since $\tau_r$ is an involution, we have $\Sp{\Sbuildall}^m_2 - v \subseteq \tau_r(\Sp{\Sbuildall}^m_1 - v)$. Hence $\tau_r(\Sp{\Sbuildall}^m_1 - v) = \Sp{\Sbuildall}^m_2 - v$, as claimed.

Consider then $\Sp{\Sbuildmon}^m - v$. If $v = p$ for some $p \in \nset{m}$, then $\Sp{\Sbuildmon}^m - v = \{\Sbuildmon^m_p\}$. If $v = p'$ for some $p \in \nset{m}$, then $\Sp{\Sbuildmon}^m - v = \{\Sbuildmon^m_p, \Sbuildmon^m_{p-1}\}$. In either case, $\tau_{p+2}(\Sp{\Sbuildmon}^m - v) = \Sp{\Sbuildmon}^m - v$.

Consequently, choosing $p \in \nset{m}$ such that $v \in \{p, p'\}$, we have
\begin{align*}
\tau_{p+2}(\Sp{\Smon}^m_1 - v)
&= \tau_{p+2}((\Sp{\Sbuildall}^m_1 \cup \Sp{\Sbuildmon}^m) - v)
= \tau_{p+2}((\Sp{\Sbuildall}^m_1 - v) \cup (\Sp{\Sbuildmon}^m - v)) \\
&= \tau_{p+2}(\Sp{\Sbuildall}^m_1 - v) \cup \tau_{p+2}(\Sp{\Sbuildmon}^m - v)
= (\Sp{\Sbuildall}^m_2) - v \cup (\Sp{\Sbuildmon}^m - v) \\
&= (\Sp{\Sbuildall}^m_2 \cup \Sp{\Sbuildmon}^m) - v
= \Sp{\Smon}^m_2 - v.
\end{align*}
We conclude that $\Sp{\Smon}^m_1 - v \equiv \Sp{\Smon}^m_2 - v$ for all $v \in E_m$.
\end{proof}

\begin{corollary}
\label{cor:nonreconstructiblehypergraphs}
For $m \geq 3$, the Sperner systems $\Sp{\Smon}^m_1$ and $\Sp{\Smon}^m_2$ are nonisomorphic and strongly hypomorphic with respect to the reconstruction problem for hypergraphs and vertex-deleted subhypergraphs.
\end{corollary}

\begin{proof}
Follows from Propositions~\ref{prop:Cnoniso} and~\ref{prop:hypergraph}.
\end{proof}


\appendix
\section{Post classes}
\label{app:Post}

\begin{figure}
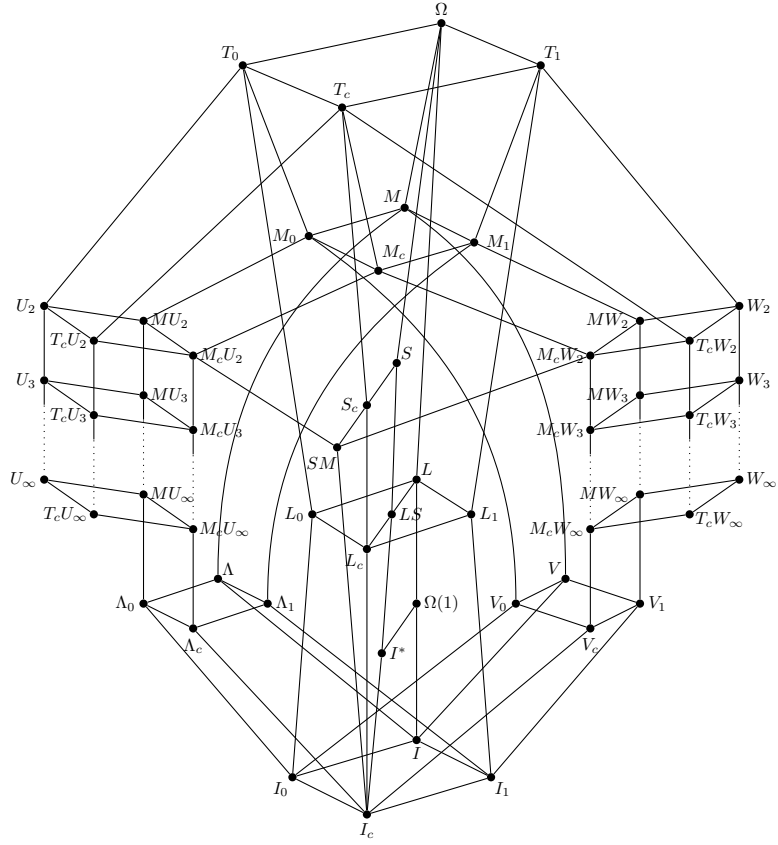

\PostsLattice{0.66}
\caption{Post's lattice.}
\label{fig:PostsLattice}
\end{figure}

The clones on the two-element set $\{0, 1\}$ were completely described by Post~\cite{Post}, and they are often called \emph{Post classes.} These clones are listed below, following the terminology and notation of~\cite{FolPog} and~\cite{JGK}, and the lattice of clones is presented in Figure~\ref{fig:PostsLattice}.

\begin{itemize}
\item $\Omega$ denotes the clone of all Boolean functions.
\end{itemize}
For $a \in \{0,1\}$, a Boolean function $f$ is \emph{$a$-preserving} if $f(a, \dots, a) = a$.
\begin{itemize}
\item $T_0$ denotes the clone of all $0$-preserving functions.
\item $T_1$ denotes the clone of all $1$-preserving functions.
\item $T_c = T_0 \cap T_1$.
\end{itemize}
A Boolean function $f \colon \{0, 1\}^n \to \{0, 1\}$ is \emph{monotone} if $f(a_1, \dots, a_n) \leq f(b_1, \dots, b_n)$ whenever $a_i \leq b_i$ for all $i \in \nset{n}$.
\begin{itemize}
\item $M$ denotes the clone of all monotone functions.
\item $M_0 = M \cap T_0$, $M_1 = M \cap T_1$, $M_c = M \cap T_c$.
\end{itemize}
Let us write $\overline{0} = 1$ and $\overline{1} = 0$. The \emph{dual} of a Boolean function $f \colon \{0, 1\}^n \to \{0, 1\}$ is the function $f^\mathrm{d} \colon \{0, 1\}^n \to \{0, 1\}$ given by $f^\mathrm{d}(\vect{a}) = \overline{f(\overline{\vect{a}})}$ for all $\vect{a} \in \{0, 1\}^n$. A Boolean function $f$ is \emph{self-dual} if $f = f^\mathrm{d}$.
\begin{itemize}
\item $S$ denotes the clone of all self-dual functions.
\item $S_c = S \cap T_c$, $SM = S \cap M$.
\item $L$ denotes the clone of all polynomial operations of the group of addition modulo $2$.
\item $L_0 = L \cap T_0$, $L_1 = L \cap T_1$, $LS = L \cap S$, $L_c = L \cap T_c$.
\end{itemize}
Let $a \in \{0, 1\}$. A set $S \subseteq \{0, 1\}^n$ is said to be \emph{$a$-separating} if there is an index $i \in \nset{n}$ such that for every $(a_1, \dots, a_n) \in S$ we have $a_i = a$. A function $f$ is said to be \emph{$a$-separating} if $f^{-1}(a)$ is $a$-separating. For an integer $m \geq 2$, a function $f$ is said to be \emph{$a$-separating of rank $m$} if every subset of $f^{-1}(a)$ of cardinality at most $m$ is $a$-separating.
\begin{itemize}
\item For $m \geq 2$, $U_m$ and $W_m$ denote the clones of all $1$- and $0$-separating functions of rank $m$, respectively.
\item $U_\infty$ and $W_\infty$ denote the clones of all $1$- and $0$-separating functions, respectively.
\item For $m \in \{2, \dots, \infty\}$, $T_cU_m = T_c \cap U_m$, $T_cW_m = T_c \cap W_m$, $MU_m = M \cap U_m$, $MW_m = M \cap W_m$, $M_cU_m = M_c \cap U_m$, $M_cW_m = M_c \cap W_m$.
\item $\Lambda$ denotes the clone of all polynomial operations of the two-element meet-semilattice $(\{0, 1\}; \wedge)$.
\item $\Lambda_0 = \Lambda \cap T_0$, $\Lambda_1 = \Lambda \cap T_1$, $\Lambda_c = \Lambda \cap T_c$.
\item $V$ denotes the clone of all polynomial operations of the two-element join-semilattice $(\{0, 1\}; \vee)$.
\item $V_0 = V \cap T_0$, $V_1 = V \cap T_1$, $V_c = V \cap T_c$.
\item $\Omega(1)$ denotes the clone of all projections, negated projections and constant functions.
\item $I^* = \Omega(1) \cap S$, $I = \Lambda \cap V$, $I_0 = I \cap T_0$, $I_1 = I \cap T_1$, $I_c = I \cap T_c$.
\end{itemize}


\section{Sperner systems over small sets}
\label{app:Sperner}

We present in the following tables all Sperner systems, up to isomorphism, over the $n$-element set $\nset{n}$, for each $n \in \{2, 3, 4, 5\}$, as well as their decks.
The rows of each table are labeled with Sperner systems over $\nset{n}$, while the columns are labeled with Sperner systems over $\nset{n-1}$. We use a shorthand notation: all set brackets are omitted, each block is presented by writing its elements in juxtaposition without any separating symbol, and blocks are separated by commas.
The number at row $\Sp{\Sgena}$ column $\Sp{\Sgenb}$ indicates the multiplicity of $\Sp{\Sgenb}$ in the deck of $\Sp{\Sgena}$. (For the sake of clarity, we have not written down any $0$'s.)
We have omitted from the tables the Sperner systems $\emptyset$ and $\{\emptyset\}$; the deck of $\emptyset$ comprises $\binom{n}{2}$ occurrences of $\emptyset$, and the deck of $\{\emptyset\}$ comprises $\binom{n}{2}$ occurrences of $\{\emptyset\}$, and these two systems are not cards of any Sperner system with nonempty blocks.
The nonreconstructible Sperner systems are indicated in the tables by an asterisk.
Note that every Sperner system over a five-element set is reconstructible.

\newcommand{\sep}{\discretionary{,}{\mbox{\qquad}}{,\mbox{\,}}}
\newcommand{\columnlabels}{}

\noindent
\begin{minipage}[t]{0.39\textwidth}
\renewcommand{\columnlabels}{
\multicolumn{1}{|c|}{\raisebox{0ex}{\normalsize $n = 2$}} & 
1\\
\hline
}

\noindent
\begin{tabular}[t]{|p{1.75cm}|*{1}{c}|}
\hline
\columnlabels
1 * & 1 \\
1\sep 2 * & 1 \\
12 * & 1 \\
\hline
\end{tabular}

\renewcommand{\columnlabels}{
\multicolumn{1}{|c|}{\raisebox{0ex}{\normalsize $n = 3$}} & 
1 & 
1\sep 2 & 
12\\
\hline
}
\bigskip\bigskip\bigskip
\noindent
\begin{tabular}{|p{1.75cm}|*{3}{c}|}
\hline
\columnlabels
1 * & 3 &   &   \\
1\sep 2 & 1 & 2 &   \\
1\sep 2\sep 3 &   & 3 &   \\
1\sep 23 & 2 & 1 &   \\
12 & 1 &   & 2 \\
12\sep 13 & 2 &   & 1 \\
12\sep 13\sep 23 * & 3 &   &   \\
123 &   &   & 3 \\
\hline
\end{tabular}
\end{minipage}
\hfill
\begin{minipage}[t]{0.6\textwidth}
\raggedleft
\renewcommand{\columnlabels}{
\multicolumn{1}{|c|}{\raisebox{4ex}{\normalsize $n = 4$}} & 
\begin{sideways} 1 \end{sideways} & 
\begin{sideways} 1\sep 2 \end{sideways} & 
\begin{sideways} 1\sep 2\sep 3 \end{sideways} & 
\begin{sideways} 1\sep 23 \end{sideways} & 
\begin{sideways} 12 \end{sideways} & 
\begin{sideways} 12\sep 13 \end{sideways} & 
\begin{sideways} 12\sep 13\sep 23\,\, \end{sideways} & 
\begin{sideways} 123 \end{sideways}\\
\hline
}
\noindent
{\footnotesize
\begin{tabular}[t]{|l|*{8}{c}|}
\hline
\columnlabels
1 & 6 &   &   &   &   &   &   &   \\
1\sep 2 & 1 & 5 &   &   &   &   &   &   \\
1\sep 2\sep 3 &   & 3 & 3 &   &   &   &   &   \\
1\sep 2\sep 3\sep 4 &   &   & 6 &   &   &   &   &   \\
1\sep 2\sep 34 &   & 4 & 1 & 1 &   &   &   &   \\
1\sep 23 & 2 & 1 &   & 3 &   &   &   &   \\
1\sep 23\sep 24 & 1 & 2 &   & 3 &   &   &   &   \\
1\sep 23\sep 24\sep 34 &   & 3 &   & 3 &   &   &   &   \\
1\sep 234 & 3 &   &   & 3 &   &   &   &   \\
12 & 1 &   &   &   & 5 &   &   &   \\
12\sep 13 & 2 &   &   &   & 1 & 3 &   &   \\
12\sep 13\sep 14 & 3 &   &   &   &   & 3 &   &   \\
12\sep 13\sep 14\sep 23 & 2 &   &   & 2 &   &   & 2 &   \\
12\sep 13\sep 14\sep 23\sep 24 & 1 &   &   & 4 &   &   & 1 &   \\
12\sep 13\sep 14\sep 23\sep 24\sep 34 &   &   &   & 6 &   &   &   &   \\
12\sep 13\sep 14\sep 234 * & 3 &   &   &   &   &   & 3 &   \\
12\sep 13\sep 23 * & 3 &   &   &   &   &   & 3 &   \\
12\sep 13\sep 24 & 1 &   &   & 2 &   & 2 & 1 &   \\
12\sep 13\sep 24\sep 34 &   &   &   & 4 &   & 2 &   &   \\
12\sep 13\sep 234 & 2 &   &   &   &   & 2 & 2 &   \\
12\sep 34 &   &   &   & 2 &   & 4 &   &   \\
12\sep 134 & 1 &   &   &   & 2 & 3 &   &   \\
12\sep 134\sep 234 & 1 &   &   &   &   & 4 & 1 &   \\
123 &   &   &   &   & 3 &   &   & 3 \\
123\sep 124 &   &   &   &   & 4 & 1 &   & 1 \\
123\sep 124\sep 134 &   &   &   &   & 3 & 3 &   &   \\
123\sep 124\sep 134\sep 234 &   &   &   &   &   & 6 &   &   \\
1234 &   &   &   &   &   &   &   & 6 \\
\hline
\end{tabular}
}
\end{minipage}

\bigskip

\renewcommand{\columnlabels}{
\multicolumn{1}{|c|}{\raisebox{8ex}{\normalsize $n = 5$}} & 
\begin{sideways} 1 \end{sideways} & 
\begin{sideways} 1\sep 2 \end{sideways} & 
\begin{sideways} 1\sep 2\sep 3 \end{sideways} & 
\begin{sideways} 1\sep 2\sep 3\sep 4 \end{sideways} & 
\begin{sideways} 1\sep 2\sep 34 \end{sideways} & 
\begin{sideways} 1\sep 23 \end{sideways} & 
\begin{sideways} 1\sep 23\sep 24 \end{sideways} & 
\begin{sideways} 1\sep 23\sep 24\sep 34 \end{sideways} & 
\begin{sideways} 1\sep 234 \end{sideways} & 
\begin{sideways} 12 \end{sideways} & 
\begin{sideways} 12\sep 13 \end{sideways} & 
\begin{sideways} 12\sep 13\sep 14 \end{sideways} & 
\begin{sideways} 12\sep 13\sep 14\sep 23 \end{sideways} & 
\begin{sideways} 12\sep 13\sep 14\sep 23\sep 24 \end{sideways} & 
\begin{sideways} 12\sep 13\sep 14\sep 23\sep 24\sep 34\,\, \end{sideways} & 
\begin{sideways} 12\sep 13\sep 14\sep 234 \end{sideways} & 
\begin{sideways} 12\sep 13\sep 23 \end{sideways} & 
\begin{sideways} 12\sep 13\sep 24 \end{sideways} & 
\begin{sideways} 12\sep 13\sep 24\sep 34 \end{sideways} & 
\begin{sideways} 12\sep 13\sep 234 \end{sideways} & 
\begin{sideways} 12\sep 34 \end{sideways} & 
\begin{sideways} 12\sep 134 \end{sideways} & 
\begin{sideways} 12\sep 134\sep 234 \end{sideways} & 
\begin{sideways} 123 \end{sideways} & 
\begin{sideways} 123\sep 124 \end{sideways} & 
\begin{sideways} 123\sep 124\sep 134 \end{sideways} & 
\begin{sideways} 123\sep 124\sep 134\sep 234 \end{sideways} & 
\begin{sideways} 1234 \end{sideways}\\
\hline
}

\noindent
{\scriptsize
\setlength{\tabcolsep}{0pt}
\begin{longtable}{|@{\,\,}p{4cm}@{\extracolsep{\fill}}|*{28}{c}@{\,\,}|}\hline
\columnlabels
\endfirsthead
\multicolumn{29}{c}{\textit{Continued from previous page}} \\
\hline
\columnlabels
\endhead
\hline \multicolumn{29}{r}{\textit{Continued on next page}} \\
\endfoot
\hline
\endlastfoot

1 & 10 &   &   &   &   &   &   &   &   &   &   &   &   &   &   &   &   &   &   &   &   &   &   &   &   &   &   &   \\
1\sep 2 & 1 & 9 &   &   &   &   &   &   &   &   &   &   &   &   &   &   &   &   &   &   &   &   &   &   &   &   &   &   \\
1\sep 2\sep 3 &   & 3 & 7 &   &   &   &   &   &   &   &   &   &   &   &   &   &   &   &   &   &   &   &   &   &   &   &   &   \\
1\sep 2\sep 3\sep 4 &   &   & 6 & 4 &   &   &   &   &   &   &   &   &   &   &   &   &   &   &   &   &   &   &   &   &   &   &   &   \\
1\sep 2\sep 3\sep 4\sep 5 &   &   &   & 10 &   &   &   &   &   &   &   &   &   &   &   &   &   &   &   &   &   &   &   &   &   &   &   &   \\
1\sep 2\sep 3\sep 45 &   &   & 6 & 1 & 3 &   &   &   &   &   &   &   &   &   &   &   &   &   &   &   &   &   &   &   &   &   &   &   \\
1\sep 2\sep 34 &   & 4 & 1 &   & 4 & 1 &   &   &   &   &   &   &   &   &   &   &   &   &   &   &   &   &   &   &   &   &   &   \\
1\sep 2\sep 34\sep 35 &   & 2 & 2 &   & 5 &   & 1 &   &   &   &   &   &   &   &   &   &   &   &   &   &   &   &   &   &   &   &   &   \\
1\sep 2\sep 34\sep 35\sep 45 &   &   & 3 &   & 6 &   &   & 1 &   &   &   &   &   &   &   &   &   &   &   &   &   &   &   &   &   &   &   &   \\
1\sep 2\sep 345 &   & 6 &   &   & 3 &   &   &   & 1 &   &   &   &   &   &   &   &   &   &   &   &   &   &   &   &   &   &   &   \\
1\sep 23 & 2 & 1 &   &   &   & 7 &   &   &   &   &   &   &   &   &   &   &   &   &   &   &   &   &   &   &   &   &   &   \\
1\sep 23\sep 24 & 1 & 2 &   &   &   & 3 & 4 &   &   &   &   &   &   &   &   &   &   &   &   &   &   &   &   &   &   &   &   &   \\
1\sep 23\sep 24\sep 25 & 1 & 3 &   &   &   &   & 6 &   &   &   &   &   &   &   &   &   &   &   &   &   &   &   &   &   &   &   &   &   \\
1\sep 23\sep 24\sep 25\sep 34 &   & 2 &   &   & 2 & 1 & 2 & 3 &   &   &   &   &   &   &   &   &   &   &   &   &   &   &   &   &   &   &   &   \\
1\sep 23\sep 24\sep 25\sep 34\sep 35 &   & 1 &   &   & 4 &   & 2 & 3 &   &   &   &   &   &   &   &   &   &   &   &   &   &   &   &   &   &   &   &   \\
1\sep 23\sep 24\sep 25\sep 34\sep 35\sep 45 &   &   &   &   & 6 &   &   & 4 &   &   &   &   &   &   &   &   &   &   &   &   &   &   &   &   &   &   &   &   \\
1\sep 23\sep 24\sep 25\sep 345 &   & 3 &   &   &   &   & 3 & 3 & 1 &   &   &   &   &   &   &   &   &   &   &   &   &   &   &   &   &   &   &   \\
1\sep 23\sep 24\sep 34 &   & 3 &   &   &   & 3 &   & 4 &   &   &   &   &   &   &   &   &   &   &   &   &   &   &   &   &   &   &   &   \\
1\sep 23\sep 24\sep 35 &   & 1 &   &   & 2 & 2 & 4 & 1 &   &   &   &   &   &   &   &   &   &   &   &   &   &   &   &   &   &   &   &   \\
1\sep 23\sep 24\sep 35\sep 45 &   &   &   &   & 4 &   & 6 &   &   &   &   &   &   &   &   &   &   &   &   &   &   &   &   &   &   &   &   &   \\
1\sep 23\sep 24\sep 345 &   & 2 &   &   &   & 2 & 3 & 2 & 1 &   &   &   &   &   &   &   &   &   &   &   &   &   &   &   &   &   &   &   \\
1\sep 23\sep 45 &   &   &   &   & 2 & 4 & 4 &   &   &   &   &   &   &   &   &   &   &   &   &   &   &   &   &   &   &   &   &   \\
1\sep 23\sep 245 & 1 & 1 &   &   &   & 4 & 3 &   & 1 &   &   &   &   &   &   &   &   &   &   &   &   &   &   &   &   &   &   &   \\
1\sep 23\sep 245\sep 345 &   & 1 &   &   &   & 2 & 4 & 1 & 2 &   &   &   &   &   &   &   &   &   &   &   &   &   &   &   &   &   &   &   \\
1\sep 234 & 3 &   &   &   &   & 3 &   &   & 4 &   &   &   &   &   &   &   &   &   &   &   &   &   &   &   &   &   &   &   \\
1\sep 234\sep 235 & 2 &   &   &   &   & 4 & 1 &   & 3 &   &   &   &   &   &   &   &   &   &   &   &   &   &   &   &   &   &   &   \\
1\sep 234\sep 235\sep 245 & 1 &   &   &   &   & 3 & 3 &   & 3 &   &   &   &   &   &   &   &   &   &   &   &   &   &   &   &   &   &   &   \\
1\sep 234\sep 235\sep 245\sep 345 &   &   &   &   &   &   & 6 &   & 4 &   &   &   &   &   &   &   &   &   &   &   &   &   &   &   &   &   &   &   \\
1\sep 2345 & 4 &   &   &   &   &   &   &   & 6 &   &   &   &   &   &   &   &   &   &   &   &   &   &   &   &   &   &   &   \\
12 & 1 &   &   &   &   &   &   &   &   & 9 &   &   &   &   &   &   &   &   &   &   &   &   &   &   &   &   &   &   \\
12\sep 13 & 2 &   &   &   &   &   &   &   &   & 1 & 7 &   &   &   &   &   &   &   &   &   &   &   &   &   &   &   &   &   \\
12\sep 13\sep 14 & 3 &   &   &   &   &   &   &   &   &   & 3 & 4 &   &   &   &   &   &   &   &   &   &   &   &   &   &   &   &   \\
12\sep 13\sep 14\sep 15 & 4 &   &   &   &   &   &   &   &   &   &   & 6 &   &   &   &   &   &   &   &   &   &   &   &   &   &   &   &   \\
12\sep 13\sep 14\sep 15\sep 23 & 2 &   &   &   &   & 2 & 1 &   &   &   &   &   & 5 &   &   &   &   &   &   &   &   &   &   &   &   &   &   &   \\
12\sep 13\sep 14\sep 15\sep 23\sep 24 & 1 &   &   &   &   & 2 & 3 &   &   &   &   &   & 1 & 3 &   &   &   &   &   &   &   &   &   &   &   &   &   &   \\
12\sep 13\sep 14\sep 15\sep 23\sep 24\sep 25 & 1 &   &   &   &   &   & 6 &   &   &   &   &   &   & 3 &   &   &   &   &   &   &   &   &   &   &   &   &   &   \\
12\sep 13\sep 14\sep 15\sep 23\sep 24\sep 25\sep 34 &   &   &   &   &   & 1 & 4 & 3 &   &   &   &   &   &   & 2 &   &   &   &   &   &   &   &   &   &   &   &   &   \\
12\sep 13\sep 14\sep 15\sep 23\sep 24\sep 25\sep 34\sep 35 &   &   &   &   &   &   & 3 & 6 &   &   &   &   &   &   & 1 &   &   &   &   &   &   &   &   &   &   &   &   &   \\
12\sep 13\sep 14\sep 15\sep 23\sep 24\sep 25\sep 34\sep 35\sep 45 &   &   &   &   &   &   &   & 10 &   &   &   &   &   &   &   &   &   &   &   &   &   &   &   &   &   &   &   &   \\
12\sep 13\sep 14\sep 15\sep 23\sep 24\sep 25\sep 345 &   &   &   &   &   &   & 6 &   & 1 &   &   &   &   &   & 3 &   &   &   &   &   &   &   &   &   &   &   &   &   \\
12\sep 13\sep 14\sep 15\sep 23\sep 24\sep 34 &   &   &   &   &   & 3 & 3 & 1 &   &   &   &   &   &   & 3 &   &   &   &   &   &   &   &   &   &   &   &   &   \\
12\sep 13\sep 14\sep 15\sep 23\sep 24\sep 35 &   &   &   &   &   & 2 & 3 & 2 &   &   &   &   &   & 2 & 1 &   &   &   &   &   &   &   &   &   &   &   &   &   \\
12\sep 13\sep 14\sep 15\sep 23\sep 24\sep 35\sep 45 &   &   &   &   &   &   & 4 & 4 &   &   &   &   &   & 2 &   &   &   &   &   &   &   &   &   &   &   &   &   &   \\
12\sep 13\sep 14\sep 15\sep 23\sep 24\sep 345 &   &   &   &   &   & 2 & 3 &   & 1 &   &   &   &   & 2 & 2 &   &   &   &   &   &   &   &   &   &   &   &   &   \\
12\sep 13\sep 14\sep 15\sep 23\sep 45 &   &   &   &   &   & 4 &   & 2 &   &   &   &   &   & 4 &   &   &   &   &   &   &   &   &   &   &   &   &   &   \\
12\sep 13\sep 14\sep 15\sep 23\sep 245 & 1 &   &   &   &   & 2 & 1 &   & 1 &   &   &   & 2 & 3 &   &   &   &   &   &   &   &   &   &   &   &   &   &   \\
12\sep 13\sep 14\sep 15\sep 23\sep 245\sep 345 &   &   &   &   &   & 2 & 1 &   & 2 &   &   &   &   & 4 & 1 &   &   &   &   &   &   &   &   &   &   &   &   &   \\
12\sep 13\sep 14\sep 15\sep 234 & 3 &   &   &   &   &   &   &   & 1 &   &   &   & 3 &   &   & 3 &   &   &   &   &   &   &   &   &   &   &   &   \\
12\sep 13\sep 14\sep 15\sep 234\sep 235 & 2 &   &   &   &   &   &   &   & 2 &   &   &   & 4 & 1 &   & 1 &   &   &   &   &   &   &   &   &   &   &   &   \\
12\sep 13\sep 14\sep 15\sep 234\sep 235\sep 245 & 1 &   &   &   &   &   &   &   & 3 &   &   &   & 3 & 3 &   &   &   &   &   &   &   &   &   &   &   &   &   &   \\
12\sep 13\sep 14\sep 15\sep 234\sep 235\sep 245\sep 345 &   &   &   &   &   &   &   &   & 4 &   &   &   &   & 6 &   &   &   &   &   &   &   &   &   &   &   &   &   &   \\
12\sep 13\sep 14\sep 15\sep 2345 & 4 &   &   &   &   &   &   &   &   &   &   &   &   &   &   & 6 &   &   &   &   &   &   &   &   &   &   &   &   \\
12\sep 13\sep 14\sep 23 & 2 &   &   &   &   & 2 &   &   &   &   &   &   & 4 &   &   &   & 2 &   &   &   &   &   &   &   &   &   &   &   \\
12\sep 13\sep 14\sep 23\sep 24 & 1 &   &   &   &   & 4 &   &   &   &   &   &   &   & 4 &   &   & 1 &   &   &   &   &   &   &   &   &   &   &   \\
12\sep 13\sep 14\sep 23\sep 24\sep 34 &   &   &   &   &   & 6 &   &   &   &   &   &   &   &   & 4 &   &   &   &   &   &   &   &   &   &   &   &   &   \\
12\sep 13\sep 14\sep 23\sep 24\sep 35 &   &   &   &   &   & 3 & 2 & 1 &   &   &   &   & 1 & 2 & 1 &   &   &   &   &   &   &   &   &   &   &   &   &   \\
12\sep 13\sep 14\sep 23\sep 24\sep 35\sep 45 &   &   &   &   &   &   & 5 & 2 &   &   &   &   & 1 & 2 &   &   &   &   &   &   &   &   &   &   &   &   &   &   \\
12\sep 13\sep 14\sep 23\sep 24\sep 345 &   &   &   &   &   & 4 &   &   & 1 &   &   &   & 1 & 2 & 2 &   &   &   &   &   &   &   &   &   &   &   &   &   \\
12\sep 13\sep 14\sep 23\sep 25 & 1 &   &   &   &   & 2 & 2 &   &   &   &   &   & 4 & 1 &   &   &   &   &   &   &   &   &   &   &   &   &   &   \\
12\sep 13\sep 14\sep 23\sep 25\sep 45 &   &   &   &   &   & 1 & 4 & 1 &   &   &   &   & 2 & 2 &   &   &   &   &   &   &   &   &   &   &   &   &   &   \\
12\sep 13\sep 14\sep 23\sep 45 &   &   &   &   &   & 3 & 1 & 1 &   &   &   &   & 3 & 2 &   &   &   &   &   &   &   &   &   &   &   &   &   &   \\
12\sep 13\sep 14\sep 23\sep 245 & 1 &   &   &   &   & 2 &   &   & 1 &   &   &   & 3 & 2 &   &   & 1 &   &   &   &   &   &   &   &   &   &   &   \\
12\sep 13\sep 14\sep 23\sep 245\sep 35 &   &   &   &   &   & 2 & 2 &   & 1 &   &   &   & 2 & 2 & 1 &   &   &   &   &   &   &   &   &   &   &   &   &   \\
12\sep 13\sep 14\sep 23\sep 245\sep 345 &   &   &   &   &   & 2 &   &   & 2 &   &   &   & 3 & 2 & 1 &   &   &   &   &   &   &   &   &   &   &   &   &   \\
12\sep 13\sep 14\sep 25 & 1 &   &   &   &   & 2 & 1 &   &   &   &   & 1 & 2 &   &   &   &   & 3 &   &   &   &   &   &   &   &   &   &   \\
12\sep 13\sep 14\sep 25\sep 35 &   &   &   &   &   & 2 & 3 &   &   &   &   & 1 &   & 1 &   &   &   & 1 & 2 &   &   &   &   &   &   &   &   &   \\
12\sep 13\sep 14\sep 25\sep 35\sep 45 &   &   &   &   &   &   & 6 &   &   &   &   & 1 &   &   &   &   &   &   & 3 &   &   &   &   &   &   &   &   &   \\
12\sep 13\sep 14\sep 234 & 3 &   &   &   &   &   &   &   &   &   &   &   &   &   &   & 4 & 3 &   &   &   &   &   &   &   &   &   &   &   \\
12\sep 13\sep 14\sep 234\sep 25 & 1 &   &   &   &   & 2 & 1 &   &   &   &   &   & 5 &   &   & 1 &   &   &   &   &   &   &   &   &   &   &   &   \\
12\sep 13\sep 14\sep 234\sep 25\sep 35 &   &   &   &   &   & 2 & 3 &   &   &   &   &   & 1 & 3 &   & 1 &   &   &   &   &   &   &   &   &   &   &   &   \\
12\sep 13\sep 14\sep 234\sep 25\sep 35\sep 45 &   &   &   &   &   &   & 6 &   &   &   &   &   &   & 3 &   & 1 &   &   &   &   &   &   &   &   &   &   &   &   \\
12\sep 13\sep 14\sep 234\sep 235 & 2 &   &   &   &   &   &   &   & 1 &   &   &   & 3 &   &   & 2 & 2 &   &   &   &   &   &   &   &   &   &   &   \\
12\sep 13\sep 14\sep 234\sep 235\sep 45 &   &   &   &   &   & 2 & 1 &   & 1 &   &   &   & 2 & 3 &   & 1 &   &   &   &   &   &   &   &   &   &   &   &   \\
12\sep 13\sep 14\sep 234\sep 235\sep 245 & 1 &   &   &   &   &   &   &   & 2 &   &   &   & 4 & 1 &   & 1 & 1 &   &   &   &   &   &   &   &   &   &   &   \\
12\sep 13\sep 14\sep 234\sep 235\sep 245\sep 345 &   &   &   &   &   &   &   &   & 3 &   &   &   & 3 & 3 &   & 1 &   &   &   &   &   &   &   &   &   &   &   &   \\
12\sep 13\sep 14\sep 235 & 2 &   &   &   &   &   &   &   & 1 &   &   & 1 & 2 &   &   & 1 &   & 1 &   & 2 &   &   &   &   &   &   &   &   \\
12\sep 13\sep 14\sep 235\sep 45 &   &   &   &   &   & 2 & 1 &   & 1 &   &   & 1 &   & 2 &   &   &   & 2 & 1 &   &   &   &   &   &   &   &   &   \\
12\sep 13\sep 14\sep 235\sep 245 & 1 &   &   &   &   &   &   &   & 2 &   &   & 1 & 2 & 1 &   &   &   & 2 &   & 1 &   &   &   &   &   &   &   &   \\
12\sep 13\sep 14\sep 235\sep 245\sep 345 &   &   &   &   &   &   &   &   & 3 &   &   & 1 &   & 3 &   &   &   & 3 &   &   &   &   &   &   &   &   &   &   \\
12\sep 13\sep 14\sep 2345 & 3 &   &   &   &   &   &   &   &   &   &   & 1 &   &   &   & 3 &   &   &   & 3 &   &   &   &   &   &   &   &   \\
12\sep 13\sep 23 & 3 &   &   &   &   &   &   &   &   &   &   &   &   &   &   &   & 7 &   &   &   &   &   &   &   &   &   &   &   \\
12\sep 13\sep 23\sep 45 &   &   &   &   &   & 3 &   & 1 &   &   &   &   & 6 &   &   &   &   &   &   &   &   &   &   &   &   &   &   &   \\
12\sep 13\sep 24 & 1 &   &   &   &   & 2 &   &   &   &   & 2 &   &   &   &   &   & 1 & 4 &   &   &   &   &   &   &   &   &   &   \\
12\sep 13\sep 24\sep 34 &   &   &   &   &   & 4 &   &   &   &   & 2 &   &   &   &   &   &   &   & 4 &   &   &   &   &   &   &   &   &   \\
12\sep 13\sep 24\sep 35 &   &   &   &   &   & 2 & 2 &   &   &   &   & 1 & 2 &   &   &   &   & 2 & 1 &   &   &   &   &   &   &   &   &   \\
12\sep 13\sep 24\sep 35\sep 45 &   &   &   &   &   &   & 5 &   &   &   &   &   & 5 &   &   &   &   &   &   &   &   &   &   &   &   &   &   &   \\
12\sep 13\sep 24\sep 345 &   &   &   &   &   & 2 &   &   & 1 &   & 2 &   & 1 &   &   &   &   & 2 & 2 &   &   &   &   &   &   &   &   &   \\
12\sep 13\sep 45 &   &   &   &   &   & 2 & 1 &   &   &   &   & 2 &   &   &   &   &   & 4 &   &   & 1 &   &   &   &   &   &   &   \\
12\sep 13\sep 145 & 2 &   &   &   &   &   &   &   &   &   & 4 & 3 &   &   &   &   &   &   &   &   &   & 1 &   &   &   &   &   &   \\
12\sep 13\sep 145\sep 23 & 2 &   &   &   &   &   &   &   & 1 &   &   &   & 3 &   &   &   & 4 &   &   &   &   &   &   &   &   &   &   &   \\
12\sep 13\sep 145\sep 23\sep 245 & 1 &   &   &   &   &   &   &   & 2 &   &   &   & 4 & 1 &   &   & 2 &   &   &   &   &   &   &   &   &   &   &   \\
12\sep 13\sep 145\sep 23\sep 245\sep 345 &   &   &   &   &   &   &   &   & 3 &   &   &   & 6 &   & 1 &   &   &   &   &   &   &   &   &   &   &   &   &   \\
12\sep 13\sep 145\sep 24 & 1 &   &   &   &   & 2 &   &   &   &   &   & 1 & 2 &   &   &   & 1 & 2 &   & 1 &   &   &   &   &   &   &   &   \\
12\sep 13\sep 145\sep 24\sep 34 &   &   &   &   &   & 4 &   &   &   &   &   & 1 &   & 2 &   &   &   &   & 2 & 1 &   &   &   &   &   &   &   &   \\
12\sep 13\sep 145\sep 24\sep 35 &   &   &   &   &   & 2 & 2 &   &   &   &   &   & 4 & 1 &   & 1 &   &   &   &   &   &   &   &   &   &   &   &   \\
12\sep 13\sep 145\sep 24\sep 345 &   &   &   &   &   & 2 &   &   & 1 &   &   & 1 & 2 & 1 &   &   &   & 1 & 1 & 1 &   &   &   &   &   &   &   &   \\
12\sep 13\sep 145\sep 234 & 2 &   &   &   &   &   &   &   &   &   &   & 1 &   &   &   & 2 & 2 &   &   & 3 &   &   &   &   &   &   &   &   \\
12\sep 13\sep 145\sep 234\sep 25 & 1 &   &   &   &   & 2 &   &   &   &   &   &   & 4 &   &   & 2 & 1 &   &   &   &   &   &   &   &   &   &   &   \\
12\sep 13\sep 145\sep 234\sep 25\sep 35 &   &   &   &   &   & 4 &   &   &   &   &   &   &   & 4 &   & 2 &   &   &   &   &   &   &   &   &   &   &   &   \\
12\sep 13\sep 145\sep 234\sep 235 & 2 &   &   &   &   &   &   &   &   &   &   &   &   &   &   & 4 & 4 &   &   &   &   &   &   &   &   &   &   &   \\
12\sep 13\sep 145\sep 234\sep 235\sep 245 & 1 &   &   &   &   &   &   &   & 1 &   &   &   & 3 &   &   & 3 & 2 &   &   &   &   &   &   &   &   &   &   &   \\
12\sep 13\sep 145\sep 234\sep 235\sep 245\sep 345 &   &   &   &   &   &   &   &   & 2 &   &   &   & 4 & 1 &   & 3 &   &   &   &   &   &   &   &   &   &   &   &   \\
12\sep 13\sep 145\sep 234\sep 245 & 1 &   &   &   &   &   &   &   & 1 &   &   & 1 & 2 &   &   & 1 & 1 & 1 &   & 2 &   &   &   &   &   &   &   &   \\
12\sep 13\sep 145\sep 234\sep 245\sep 35 &   &   &   &   &   & 2 &   &   & 1 &   &   &   & 3 & 2 &   & 2 &   &   &   &   &   &   &   &   &   &   &   &   \\
12\sep 13\sep 145\sep 234\sep 245\sep 345 &   &   &   &   &   &   &   &   & 2 &   &   & 1 & 2 & 1 &   & 1 &   & 2 &   & 1 &   &   &   &   &   &   &   &   \\
12\sep 13\sep 145\sep 245 & 1 &   &   &   &   &   &   &   & 1 &   &   & 2 & 1 &   &   &   &   & 2 &   & 2 &   &   & 1 &   &   &   &   &   \\
12\sep 13\sep 145\sep 245\sep 345 &   &   &   &   &   &   &   &   & 2 &   &   & 2 &   & 1 &   &   &   & 4 &   &   &   &   & 1 &   &   &   &   &   \\
12\sep 13\sep 145\sep 2345 & 2 &   &   &   &   &   &   &   &   &   &   & 2 &   &   &   & 1 &   &   &   & 4 &   &   & 1 &   &   &   &   &   \\
12\sep 13\sep 234 & 2 &   &   &   &   &   &   &   &   &   & 2 &   &   &   &   &   & 2 &   &   & 4 &   &   &   &   &   &   &   &   \\
12\sep 13\sep 234\sep 45 &   &   &   &   &   & 2 & 1 &   &   &   &   & 1 & 2 &   &   & 1 &   & 3 &   &   &   &   &   &   &   &   &   &   \\
12\sep 13\sep 234\sep 235 & 2 &   &   &   &   &   &   &   &   &   &   & 1 &   &   &   &   & 4 &   &   & 3 &   &   &   &   &   &   &   &   \\
12\sep 13\sep 234\sep 235\sep 45 &   &   &   &   &   & 2 & 1 &   &   &   &   &   & 5 &   &   & 2 &   &   &   &   &   &   &   &   &   &   &   &   \\
12\sep 13\sep 234\sep 235\sep 245 & 1 &   &   &   &   &   &   &   & 1 &   &   & 1 & 2 &   &   &   & 2 & 1 &   & 2 &   &   &   &   &   &   &   &   \\
12\sep 13\sep 234\sep 235\sep 245\sep 345 &   &   &   &   &   &   &   &   & 2 &   &   & 1 & 4 &   &   &   &   &   & 1 & 2 &   &   &   &   &   &   &   &   \\
12\sep 13\sep 234\sep 245 & 1 &   &   &   &   &   &   &   & 1 &   & 2 &   & 1 &   &   &   & 1 & 2 &   & 2 &   &   &   &   &   &   &   &   \\
12\sep 13\sep 234\sep 245\sep 345 &   &   &   &   &   &   &   &   & 2 &   & 2 &   & 2 &   &   &   &   & 2 & 1 & 1 &   &   &   &   &   &   &   &   \\
12\sep 13\sep 245 & 1 &   &   &   &   &   &   &   & 1 &   & 2 &   &   &   &   &   &   & 3 &   & 2 &   & 1 &   &   &   &   &   &   \\
12\sep 13\sep 245\sep 345 &   &   &   &   &   &   &   &   & 2 &   & 2 &   &   &   &   &   &   & 4 & 1 &   &   & 1 &   &   &   &   &   &   \\
12\sep 13\sep 2345 & 2 &   &   &   &   &   &   &   &   &   & 2 &   &   &   &   &   &   &   &   & 5 &   & 1 &   &   &   &   &   &   \\
12\sep 34 &   &   &   &   &   & 2 &   &   &   &   & 4 &   &   &   &   &   &   &   &   &   & 4 &   &   &   &   &   &   &   \\
12\sep 134 & 1 &   &   &   &   &   &   &   &   & 2 & 3 &   &   &   &   &   &   &   &   &   &   & 4 &   &   &   &   &   &   \\
12\sep 134\sep 35 &   &   &   &   &   & 2 &   &   &   &   & 2 & 1 &   &   &   &   &   & 2 &   & 1 & 2 &   &   &   &   &   &   &   \\
12\sep 134\sep 135 & 1 &   &   &   &   &   &   &   &   & 1 & 4 & 1 &   &   &   &   &   &   &   &   &   & 3 &   &   &   &   &   &   \\
12\sep 134\sep 135\sep 45 &   &   &   &   &   & 2 &   &   &   &   &   & 2 & 1 &   &   &   &   & 2 &   & 2 & 1 &   &   &   &   &   &   &   \\
12\sep 134\sep 135\sep 145 & 1 &   &   &   &   &   &   &   &   &   & 3 & 3 &   &   &   &   &   &   &   &   &   & 3 &   &   &   &   &   &   \\
12\sep 134\sep 135\sep 145\sep 234 & 1 &   &   &   &   &   &   &   &   &   &   & 2 &   &   &   & 1 & 1 &   &   & 4 &   &   & 1 &   &   &   &   &   \\
12\sep 134\sep 135\sep 145\sep 234\sep 235 & 1 &   &   &   &   &   &   &   &   &   &   & 1 &   &   &   & 3 & 2 &   &   & 3 &   &   &   &   &   &   &   &   \\
12\sep 134\sep 135\sep 145\sep 234\sep 235\sep 245 & 1 &   &   &   &   &   &   &   &   &   &   &   &   &   &   & 6 & 3 &   &   &   &   &   &   &   &   &   &   &   \\
12\sep 134\sep 135\sep 145\sep 234\sep 235\sep 245\sep 345 &   &   &   &   &   &   &   &   & 1 &   &   &   & 3 &   &   & 6 &   &   &   &   &   &   &   &   &   &   &   &   \\
12\sep 134\sep 135\sep 145\sep 234\sep 235\sep 345 &   &   &   &   &   &   &   &   & 1 &   &   & 1 & 2 &   &   & 3 &   & 1 &   & 2 &   &   &   &   &   &   &   &   \\
12\sep 134\sep 135\sep 145\sep 234\sep 345 &   &   &   &   &   &   &   &   & 1 &   &   & 2 & 1 &   &   & 1 &   & 2 &   & 2 &   &   & 1 &   &   &   &   &   \\
12\sep 134\sep 135\sep 145\sep 345 &   &   &   &   &   &   &   &   & 1 &   &   & 3 &   &   &   &   &   & 3 &   &   &   &   & 3 &   &   &   &   &   \\
12\sep 134\sep 135\sep 145\sep 2345 & 1 &   &   &   &   &   &   &   &   &   &   & 3 &   &   &   &   &   &   &   & 3 &   &   & 3 &   &   &   &   &   \\
12\sep 134\sep 135\sep 234 & 1 &   &   &   &   &   &   &   &   &   & 2 & 1 &   &   &   &   & 1 &   &   & 3 &   &   & 2 &   &   &   &   &   \\
12\sep 134\sep 135\sep 234\sep 45 &   &   &   &   &   & 2 &   &   &   &   &   & 1 & 2 &   &   & 2 &   & 2 &   & 1 &   &   &   &   &   &   &   &   \\
12\sep 134\sep 135\sep 234\sep 235 & 1 &   &   &   &   &   &   &   &   &   &   & 2 &   &   &   &   & 2 &   &   & 4 &   &   & 1 &   &   &   &   &   \\
12\sep 134\sep 135\sep 234\sep 235\sep 45 &   &   &   &   &   & 2 &   &   &   &   &   &   & 4 &   &   & 4 &   &   &   &   &   &   &   &   &   &   &   &   \\
12\sep 134\sep 135\sep 234\sep 235\sep 345 &   &   &   &   &   &   &   &   & 1 &   &   & 2 & 2 &   &   &   &   &   &   & 4 & 1 &   &   &   &   &   &   &   \\
12\sep 134\sep 135\sep 234\sep 245 & 1 &   &   &   &   &   &   &   &   &   & 2 &   &   &   &   & 2 & 1 &   &   & 4 &   &   &   &   &   &   &   &   \\
12\sep 134\sep 135\sep 234\sep 245\sep 345 &   &   &   &   &   &   &   &   & 1 &   & 2 &   & 1 &   &   & 2 &   & 2 &   & 2 &   &   &   &   &   &   &   &   \\
12\sep 134\sep 135\sep 234\sep 345 &   &   &   &   &   &   &   &   & 1 &   & 2 & 1 & 1 &   &   &   &   & 1 &   & 2 & 1 &   & 1 &   &   &   &   &   \\
12\sep 134\sep 135\sep 245 & 1 &   &   &   &   &   &   &   &   &   & 2 &   &   &   &   & 1 &   &   &   & 5 &   & 1 &   &   &   &   &   &   \\
12\sep 134\sep 135\sep 245\sep 345 &   &   &   &   &   &   &   &   & 1 &   & 2 &   &   &   &   & 1 &   & 3 &   & 2 &   & 1 &   &   &   &   &   &   \\
12\sep 134\sep 135\sep 345 &   &   &   &   &   &   &   &   & 1 &   & 2 & 1 &   &   &   &   &   & 2 &   &   & 1 & 1 & 2 &   &   &   &   &   \\
12\sep 134\sep 135\sep 2345 & 1 &   &   &   &   &   &   &   &   &   & 2 & 1 &   &   &   &   &   &   &   & 2 &   & 1 & 3 &   &   &   &   &   \\
12\sep 134\sep 234 & 1 &   &   &   &   &   &   &   &   &   & 4 &   &   &   &   &   & 1 &   &   &   &   &   & 4 &   &   &   &   &   \\
12\sep 134\sep 234\sep 345 &   &   &   &   &   &   &   &   & 1 &   & 4 &   & 1 &   &   &   &   &   &   &   & 2 &   & 2 &   &   &   &   &   \\
12\sep 134\sep 235 & 1 &   &   &   &   &   &   &   &   &   & 2 &   &   &   &   &   &   &   &   & 4 &   & 2 & 1 &   &   &   &   &   \\
12\sep 134\sep 235\sep 45 &   &   &   &   &   & 2 &   &   &   &   & 2 &   &   &   &   & 2 &   & 4 &   &   &   &   &   &   &   &   &   &   \\
12\sep 134\sep 235\sep 345 &   &   &   &   &   &   &   &   & 1 &   & 2 &   &   &   &   &   &   & 2 &   & 2 & 1 & 2 &   &   &   &   &   &   \\
12\sep 134\sep 345 &   &   &   &   &   &   &   &   & 1 &   & 2 &   &   &   &   &   &   & 1 &   &   & 2 & 3 & 1 &   &   &   &   &   \\
12\sep 134\sep 2345 & 1 &   &   &   &   &   &   &   &   &   & 2 &   &   &   &   &   &   &   &   & 1 &   & 3 & 3 &   &   &   &   &   \\
12\sep 345 &   &   &   &   &   &   &   &   & 1 &   &   &   &   &   &   &   &   &   &   &   & 3 & 6 &   &   &   &   &   &   \\
12\sep 1345 & 1 &   &   &   &   &   &   &   &   & 3 &   &   &   &   &   &   &   &   &   &   &   & 6 &   &   &   &   &   &   \\
12\sep 1345\sep 2345 & 1 &   &   &   &   &   &   &   &   &   &   &   &   &   &   &   &   &   &   &   &   & 6 & 3 &   &   &   &   &   \\
123 &   &   &   &   &   &   &   &   &   & 3 &   &   &   &   &   &   &   &   &   &   &   &   &   & 7 &   &   &   &   \\
123\sep 124 &   &   &   &   &   &   &   &   &   & 4 & 1 &   &   &   &   &   &   &   &   &   &   &   &   & 1 & 4 &   &   &   \\
123\sep 124\sep 125 &   &   &   &   &   &   &   &   &   & 6 &   & 1 &   &   &   &   &   &   &   &   &   &   &   &   & 3 &   &   &   \\
123\sep 124\sep 125\sep 134 &   &   &   &   &   &   &   &   &   & 2 & 2 & 1 &   &   &   &   &   &   &   &   &   & 3 &   &   &   & 2 &   &   \\
123\sep 124\sep 125\sep 134\sep 135 &   &   &   &   &   &   &   &   &   & 1 & 2 & 2 &   &   &   &   &   &   &   &   &   & 4 &   &   &   & 1 &   &   \\
123\sep 124\sep 125\sep 134\sep 135\sep 145 &   &   &   &   &   &   &   &   &   &   &   & 4 &   &   &   &   &   &   &   &   &   & 6 &   &   &   &   &   &   \\
123\sep 124\sep 125\sep 134\sep 135\sep 145\sep 234 &   &   &   &   &   &   &   &   &   &   &   & 3 &   &   &   & 1 &   &   &   & 3 &   &   & 3 &   &   &   &   &   \\
123\sep 124\sep 125\sep 134\sep 135\sep 145\sep 234\sep 235 &   &   &   &   &   &   &   &   &   &   &   & 2 &   &   &   & 3 &   &   &   & 4 &   &   & 1 &   &   &   &   &   \\
123\sep 124\sep 125\sep 134\sep 135\sep 145\sep 234\sep 235\sep 245 &   &   &   &   &   &   &   &   &   &   &   & 1 &   &   &   & 6 &   &   &   & 3 &   &   &   &   &   &   &   &   \\
123\sep 124\sep 125\sep 134\sep 135\sep 145\sep 234\sep 235\sep 245\sep 345 &   &   &   &   &   &   &   &   &   &   &   &   &   &   &   & 10 &   &   &   &   &   &   &   &   &   &   &   &   \\
123\sep 124\sep 125\sep 134\sep 135\sep 145\sep 2345 &   &   &   &   &   &   &   &   &   &   &   & 4 &   &   &   &   &   &   &   &   &   &   & 6 &   &   &   &   &   \\
123\sep 124\sep 125\sep 134\sep 135\sep 234 &   &   &   &   &   &   &   &   &   &   & 2 & 2 &   &   &   &   &   &   &   & 3 &   &   & 2 &   &   &   & 1 &   \\
123\sep 124\sep 125\sep 134\sep 135\sep 234\sep 235 &   &   &   &   &   &   &   &   &   &   &   & 3 &   &   &   &   &   &   &   & 6 &   &   &   &   &   &   & 1 &   \\
123\sep 124\sep 125\sep 134\sep 135\sep 234\sep 245 &   &   &   &   &   &   &   &   &   &   & 2 & 1 &   &   &   & 2 &   &   &   & 3 &   &   & 2 &   &   &   &   &   \\
123\sep 124\sep 125\sep 134\sep 135\sep 234\sep 245\sep 345 &   &   &   &   &   &   &   &   &   &   & 2 &   &   &   &   & 4 &   &   &   & 4 &   &   &   &   &   &   &   &   \\
123\sep 124\sep 125\sep 134\sep 135\sep 245 &   &   &   &   &   &   &   &   &   &   & 2 & 1 &   &   &   & 1 &   &   &   & 2 &   & 1 & 3 &   &   &   &   &   \\
123\sep 124\sep 125\sep 134\sep 135\sep 245\sep 345 &   &   &   &   &   &   &   &   &   &   & 2 &   &   &   &   & 2 &   &   &   & 5 &   & 1 &   &   &   &   &   &   \\
123\sep 124\sep 125\sep 134\sep 135\sep 2345 &   &   &   &   &   &   &   &   &   &   & 2 & 2 &   &   &   &   &   &   &   &   &   & 1 & 4 &   &   &   & 1 &   \\
123\sep 124\sep 125\sep 134\sep 234 &   &   &   &   &   &   &   &   &   &   & 4 & 1 &   &   &   &   &   &   &   & 1 &   &   & 2 &   &   &   & 2 &   \\
123\sep 124\sep 125\sep 134\sep 234\sep 345 &   &   &   &   &   &   &   &   &   &   & 4 &   &   &   &   & 2 &   &   &   &   &   &   & 4 &   &   &   &   &   \\
123\sep 124\sep 125\sep 134\sep 235 &   &   &   &   &   &   &   &   &   &   & 2 & 1 &   &   &   &   &   &   &   & 2 &   & 2 & 2 &   &   &   & 1 &   \\
123\sep 124\sep 125\sep 134\sep 235\sep 345 &   &   &   &   &   &   &   &   &   &   & 2 &   &   &   &   & 1 &   &   &   & 4 &   & 2 & 1 &   &   &   &   &   \\
123\sep 124\sep 125\sep 134\sep 345 &   &   &   &   &   &   &   &   &   &   & 2 &   &   &   &   & 1 &   &   &   & 1 &   & 3 & 3 &   &   &   &   &   \\
123\sep 124\sep 125\sep 134\sep 2345 &   &   &   &   &   &   &   &   &   &   & 2 & 1 &   &   &   &   &   &   &   &   &   & 3 & 2 &   &   &   & 2 &   \\
123\sep 124\sep 125\sep 345 &   &   &   &   &   &   &   &   &   &   &   &   &   &   &   & 1 &   &   &   &   &   & 6 & 3 &   &   &   &   &   \\
123\sep 124\sep 125\sep 1345 &   &   &   &   &   &   &   &   &   & 3 &   & 1 &   &   &   &   &   &   &   &   &   & 3 &   &   &   & 3 &   &   \\
123\sep 124\sep 125\sep 1345\sep 2345 &   &   &   &   &   &   &   &   &   &   &   & 1 &   &   &   &   &   &   &   &   &   & 6 &   &   &   &   & 3 &   \\
123\sep 124\sep 134 &   &   &   &   &   &   &   &   &   & 3 & 3 &   &   &   &   &   &   &   &   &   &   &   &   &   &   & 4 &   &   \\
123\sep 124\sep 134\sep 234 &   &   &   &   &   &   &   &   &   &   & 6 &   &   &   &   &   &   &   &   &   &   &   &   &   &   &   & 4 &   \\
123\sep 124\sep 134\sep 235 &   &   &   &   &   &   &   &   &   &   & 3 &   &   &   &   &   &   &   &   & 1 &   & 2 & 2 &   &   & 1 & 1 &   \\
123\sep 124\sep 134\sep 235\sep 245 &   &   &   &   &   &   &   &   &   &   & 3 &   &   &   &   &   &   &   &   & 3 &   & 1 & 2 &   &   & 1 &   &   \\
123\sep 124\sep 134\sep 235\sep 245\sep 345 &   &   &   &   &   &   &   &   &   &   & 3 &   &   &   &   &   &   &   &   & 6 &   &   &   &   &   & 1 &   &   \\
123\sep 124\sep 134\sep 2345 &   &   &   &   &   &   &   &   &   &   & 3 &   &   &   &   &   &   &   &   &   &   & 3 &   &   &   & 1 & 3 &   \\
123\sep 124\sep 135 &   &   &   &   &   &   &   &   &   & 1 & 2 &   &   &   &   &   &   &   &   &   &   & 4 &   &   & 2 & 1 &   &   \\
123\sep 124\sep 135\sep 145 &   &   &   &   &   &   &   &   &   &   & 4 &   &   &   &   &   &   &   &   &   &   & 4 &   &   & 2 &   &   &   \\
123\sep 124\sep 135\sep 245 &   &   &   &   &   &   &   &   &   &   & 1 &   &   &   &   &   &   &   &   & 2 &   & 4 & 2 &   &   & 1 &   &   \\
123\sep 124\sep 135\sep 245\sep 345 &   &   &   &   &   &   &   &   &   &   &   &   &   &   &   &   &   &   &   & 5 &   & 5 &   &   &   &   &   &   \\
123\sep 124\sep 345 &   &   &   &   &   &   &   &   &   &   &   &   &   &   &   &   &   &   &   & 1 &   & 5 & 2 &   &   & 2 &   &   \\
123\sep 124\sep 1345 &   &   &   &   &   &   &   &   &   & 2 & 1 &   &   &   &   &   &   &   &   &   &   & 2 &   &   & 2 & 3 &   &   \\
123\sep 124\sep 1345\sep 235 &   &   &   &   &   &   &   &   &   &   & 2 &   &   &   &   &   &   &   &   &   &   & 3 & 2 &   &   & 2 & 1 &   \\
123\sep 124\sep 1345\sep 235\sep 245 &   &   &   &   &   &   &   &   &   &   & 4 &   &   &   &   &   &   &   &   &   &   &   & 4 &   &   & 2 &   &   \\
123\sep 124\sep 1345\sep 2345 &   &   &   &   &   &   &   &   &   &   & 1 &   &   &   &   &   &   &   &   &   &   & 4 &   &   &   & 3 & 2 &   \\
123\sep 145 &   &   &   &   &   &   &   &   &   &   &   &   &   &   &   &   &   &   &   &   &   & 6 &   &   & 4 &   &   &   \\
123\sep 1245 &   &   &   &   &   &   &   &   &   & 2 &   &   &   &   &   &   &   &   &   &   &   & 1 &   & 2 & 5 &   &   &   \\
123\sep 1245\sep 345 &   &   &   &   &   &   &   &   &   &   &   &   &   &   &   &   &   &   &   &   &   & 4 & 2 &   &   & 4 &   &   \\
123\sep 1245\sep 1345 &   &   &   &   &   &   &   &   &   & 1 &   &   &   &   &   &   &   &   &   &   &   & 2 &   &   & 4 & 3 &   &   \\
123\sep 1245\sep 1345\sep 2345 &   &   &   &   &   &   &   &   &   &   &   &   &   &   &   &   &   &   &   &   &   & 3 &   &   &   & 6 & 1 &   \\
1234 &   &   &   &   &   &   &   &   &   &   &   &   &   &   &   &   &   &   &   &   &   &   &   & 6 &   &   &   & 4 \\
1234\sep 1235 &   &   &   &   &   &   &   &   &   &   &   &   &   &   &   &   &   &   &   &   &   &   &   & 6 & 3 &   &   & 1 \\
1234\sep 1235\sep 1245 &   &   &   &   &   &   &   &   &   &   &   &   &   &   &   &   &   &   &   &   &   &   &   & 3 & 6 & 1 &   &   \\
1234\sep 1235\sep 1245\sep 1345 &   &   &   &   &   &   &   &   &   &   &   &   &   &   &   &   &   &   &   &   &   &   &   &   & 6 & 4 &   &   \\
1234\sep 1235\sep 1245\sep 1345\sep 2345 &   &   &   &   &   &   &   &   &   &   &   &   &   &   &   &   &   &   &   &   &   &   &   &   &   & 10 &   &   \\
12345 &   &   &   &   &   &   &   &   &   &   &   &   &   &   &   &   &   &   &   &   &   &   &   &   &   &   &   & 10
\end{longtable}
}


\section*{Acknowledgments}

The work reported here was in part carried out during the authors' multiple reciprocal visits to Paris and Luxembourg.


\end{document}